\documentclass[11pt]{article} 
\usepackage{a4,amssymb, amsmath, amsthm, dsfont}
\usepackage{mathrsfs,amsfonts}
\usepackage{graphics, epsfig, subfig, color}
\usepackage[mathscr]{euscript} \DeclareMathAlphabet{\mathpzc}{OT1}{pzc}{m}{it}
\usepackage{bbm}
\usepackage{mathtools}
\usepackage{nicefrac}
\usepackage{float}
\newtheorem{thm}{Theorem}[section]
\newtheorem{cor}[thm]{Corollary}
\newtheorem{lem}[thm]{Lemma}

\newtheorem{defn}[thm]{Definition}
\newtheorem{rem}[thm]{Remark}
\numberwithin{equation}{section}

\def\be{\begin{equation}}
\def\ee{\end{equation}}
\def\bse{\begin{subequations}}
\def\ese{\end{subequations}}
\renewcommand{\theequation}{\arabic{section}.\arabic{equation}}
\begin{document}
	
	\title{	Existence and Uniqueness for a Coupled Parabolic-Hyperbolic Model of MEMS}
	\author{
		Heiko Gimperlein\thanks{University of Innsbruck, Engineering Mathematics, Technikerstra\ss e 13, 6020 Innsbruck, Austria} \thanks{Department of Mathematical, Physical and Computer Sciences,
			University of Parma, 43124, Parma, Italy} \and  Runan He\thanks{Institut f\"{u}r Mathematik, Martin-Luther-Universit\"{a}t Halle-Wittenberg, 06120 Halle (Saale), Germany} \and  Andrew A.~Lacey\thanks{Maxwell Institute for Mathematical Sciences and Department of Mathematics, School of Mathematical and Computer Sciences, Heriot-Watt University, Edinburgh, EH14 4AS, United Kingdom\\Data sharing not applicable to this article as no datasets were generated or analysed during the current study.}}
	\date{}

	\maketitle
	\abstract{\noindent Local well-posedness for a nonlinear parabolic-hyperbolic coupled system modelling Micro-Electro-Mechanical System (MEMS) is studied. The particular device considered is a simple capacitor with two closely separated plates, one of which has motion modelled by a semi-linear hyperbolic equation. The gap between the plates contains a gas and the gas pressure is taken to obey a quasi-linear parabolic Reynolds' equation. 
Local-in-time existence of strict solutions of the system is  shown,  using well-known local-in-time existence results for the hyperbolic equation, 
then H\"{o}lder continuous dependence of its solution on that of the parabolic equation, and finally getting local-in-time existence for a combined abstract parabolic problem. Semigroup approaches are vital for the local--in-time existence results.	
	}\vskip 0.8cm
	
	\section{Introduction}
	In the present paper we study the short-time existence, uniqueness and smoothness of solutions for the coupled parabolic-hyperbolic problem:  
	\bse\label{2ndcp1}
	\be\label{2ndcp1-1-1}
	\frac{\partial \left(wu\right)}{\partial t} =\frac{\partial}{\partial x}\left(w^3u\frac{\partial u}{\partial x}\right),\quad x\in\Omega,\ t\geq 0;
	\ee
	\be\label{2ndcp1-1-2}
	\frac{\partial^2w}{\partial t^2}=\frac{\partial^2w}{\partial x^2}-\frac{\beta_F}{w^2}+\beta_p(u-1),\quad x\in\Omega,\ t\geq 0,
	\ee
	\be\label{2ndcp1-2}
	u(x,0)=u_0(x),\ w(x,0)=w_0(x),\ \frac{\partial w}{\partial t}(x,0)=v_0(x),\quad x\in {\Omega},
	\ee
	\be\label{2ndcp1-3}
	u(x,t)=\theta_1, \ w(x,t)=\theta_2,\quad x\in\partial\Omega,\quad t\geq 0,
	\ee
	\ese
	which models  gas pressure and membrane position in an idealised MEMS capacitor, accounting for tension, but not elasticity, in the moving part of the capacitor -- so that it is modelled as a membrane rather than a plate.
	
	Here the variables $u(x,t)$ and $w(x,t)$ represent gas pressure and gap width respectively, $\Omega\subset\mathds{R}$ is an open, bounded interval, $\beta_F$, $\beta_p$, $\theta_1$, $\theta_2>0$ are given constants; $u_0=u_0(x)$, $v_0=v_0(x)$ and $w_0=w_0(x)$ are given functions. We prove the following result giving well-posedness for short time:
	
	\begin{thm}\label{2ndcoupled system}
		Let $\sigma \in (0,\frac{1}{2})$, $u_0 \in H^{2+\sigma}(\Omega)$, $v_0 \in H_0^1(\Omega)$ and $w_0 \in H^2(\Omega)$, compatible with the boundary conditions and such that $u_0, w_0>0$. Then there exists a time interval $[0, T)$ such that the initial-boundary value problem  \eqref{2ndcp1} admits a unique strict solution $(u, w)$ on $[0, T)$ and
		\[u\in C^{\sigma+1}\left([0, T); L^2(\Omega)\right)\cap C^\sigma\left([0, T); H^2(\Omega)\right),\]
		\[w\in C^2\left([0, T); L^2(\Omega)\right)\cap C^1([0, T); H^1(\Omega))\cap C\left([0, T); H^2(\Omega)\right).\]
	\end{thm}
	\begin{rem}
		(A) Note that global-in-time solutions are not necessarily expected, as quenching singularities with $\displaystyle\inf_{x\in\Omega} w(x,t)\to 0$ may develop as $t\to T$. Their precise description is  the subject of reference \cite{gimperlein2022quenching}.\\
		(B) The Sobolev regularity  of a solution $w$ is also limited because the inhomogeneous term in \eqref{2ndcp1-1-2}  does not vanish at boundary $\partial\Omega$ of space $\Omega$. In fact, 
the linear wave equation,
		\be\label{Max-Reg-LWE}
		\frac{\partial^2 w}{\partial t^2}-\frac{\partial^2 w}{\partial x^2}=1,\quad (x, t)\in (0, 1)\times(0, \infty),
		\ee
		with homogeneous initial conditions $w(x, 0)=\frac{\partial w}{\partial t}(x, 0)=0$ 
and homogeneous boundary condition $w(0, t)=w(1, t)=0$, $t\in[0, \infty)$, has a solution $w(t)\in H^{\frac{5}{2}-\epsilon}(0, 1)$ for every $\epsilon>0$ and $t\in[0,\infty)$, but $w(t)\notin H^{\frac{5}{2}}(0, 1)$. We thus do not expect higher integer-order Sobolev regularity for $w$ in Theorem~\ref{2ndcoupled system}.		
	\end{rem}
	
	MEMS, micro electro-mechanical systems, are small devices which combine mechanical and electrical parts and effects, with particular examples including microphones, temperature sensors, resonators, accelerometers, data-storage devices etc. (see, for example, \cite{GM}, \cite{korvink2010mems} and \cite{BP}).  The particularly simple device considered here is an electrostatically actuated MEMS capacitor (see Fig.~\ref{figq5nl}), having two parallel plates maintained at different, but constant, potentials. One plate is fixed and flat, while the other, although flat at equilibrium and in the absence of an applied potential, is free to move, but held fixed at its edges. It is also maintained at a sufficiently high constant tension for it to be regarded as a membrane.

	The two plates are close, and separated by a narrow gap, of varying width
	$w$, filled by a rarefied gas, with local pressure $u$. The gas is taken to move according to Reynolds' equation, which is valid for a thin layer of viscous fluid, and behave as an isothermal ideal gas, thus giving PDE (\ref{2ndcp1-1-1}), under appropriate scaling. Taking the gas to move freely between the gap and the rest of the MEMS device gives the first boundary condition (\ref{2ndcp1-3}). With the flexible plate subject to sufficiently high tension, it behaves as a membrane and we have the first two terms of (\ref{2ndcp1-1-2}), again on scaling appropriately. Having the membrane subject to pressure $u$ from the gap and a constant pressure from the other side gives the last terms on the right-hand side of (\ref{2ndcp1-1-2}), after scaling $u$ with that constant pressure. The remaining term of the second PDE comes from  electrostatic attraction: because the gap is narrow, the electric field is (approximately) $\propto 1/w$, giving a charge density on the membrane $\propto 1/w$,  and hence an attractive force $\propto 1/w^2$ (also proportional to the square of the applied voltage difference between the membrane and the rigid plate). Holding the membrane at its edges gives the second boundary condition of (\ref{2ndcp1-3}). 
	
	\begin{figure}[H]
		\begin{center}
			\scalebox{0.6}{{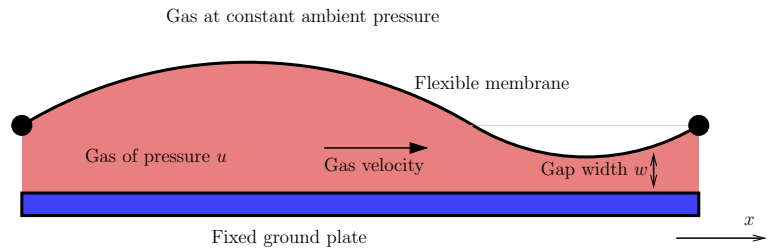}}
		\end{center}
		\caption{Schematic of a simple electrostatically actuated MEMS capacitor.}\label{figq5nl}
	\end{figure}
	
	Here we consider the case of one spatial dimension, corresponding to an idealised prismatic device, translation-invariant and of infinite extent in one direction. In this case the energy space $H^1(\Omega)$ for the wave equation is an algebra, simplifying the analysis of the nonlinearity. 

In a companion paper \cite{parabolicdispersive} we consider the two-dimensional case when the tension in the flexible electrode is not assumed to be so large. Elastic effects are then taken into account and a bi-harmonic term is additionally included in the PDE  \eqref{2ndcp1-1-2} for gap width $w$.

	In  another recent paper \cite{ellipticdispersive}, we study the wellposedness of a different MEMS model, which is simpler regarding its fluid flow:  the equation \eqref{2ndcp1-1-1} is replaced by a linear elliptic PDE for the gas pressure $u$ but, like \cite{parabolicdispersive}, includes a bi-harmonic term in the equation \eqref{2ndcp1-1-2} for gap width $w$.			Physically, such a simpler fluid model represents the operation of  MEMS devices having effectively incompressible fluid flow and an elastic plate.   Here  the dynamics of the gas pressure is replaced  by a quasi-static approximation -- applicable when the gas flow can be regarded as slow, so that, to leading order,  the pressure and density can be regarded as constant. 
The resulting model can be reduced to a perturbed semilinear dispersive equation for $w$ via standard analysis of linear elliptic equations.

In many  industrial applications, the movement of the upper plate is dominated by its tension, so that elastic effects are negligible. This approximation in the limiting case leads to a wave equation  \eqref{2ndcp1-1-2} for the gap width $w$. 
	
	The mathematical analysis of  the coupled system \eqref{2ndcp1-1-1}-\eqref{2ndcp1-3} is studied by a delicate combination of the analytical techniques available for the constituent equations. 
We here adapt the techniques from the companion paper \cite{parabolicdispersive}: refining the analysis of \eqref{2ndcp1-1-2}, reducing the coupled system to an abstract quasilinear, degenerately parabolic equation for the gas pressure $u$ and showing the wellposedness by analytic semigroup techniques. 
	
	In contrast to the model in  the companion paper \cite{parabolicdispersive}, the short-time existence of the coupled system \eqref{2ndcp1-1-1}-\eqref{2ndcp1-3} leads to the quenching singularity of the semilinear wave equation in \cite{gimperlein2022quenching}, which motivates us to study the local behaviour of the quenching solution of the wave equation in \cite{gimperlein2022quenching} formally.
	
	\
	
	MEMS-related models have been studied for a number of years. Most attention has been directed to models given by a single equation but there has been some work on systems. We now review some literature which studies such models,  numerically and/or analytically, to obtain qualitative behaviour.
	
	A key steady-state model, describing the shape of the deflected elastic membrane of a canonical MEMS system, including the steady state of the model looked at in the present paper, can be written  in the dimensionless form
	\begin{equation}\label{DeflMem}
		\Delta w=\frac{\beta_F}{w^2}, \quad w\mid_{\partial\Omega}=1.
	\end{equation}
	Here, $w$ is again  the gap between the  membrane and the rigid plate. This model has been extensively studied. For example, the paper \cite{GP} has shown that there is a $0<\beta_F^*<\frac{4\mu_0}{27}$, where $\mu_0$ is the principle eigenvalue for the associated eigenvalue problem
	\[\Delta\phi+\mu\phi=0,\quad \phi\mid_{\partial\Omega}=0,\]
	such that \eqref{DeflMem} has no solutions provided  $\beta_F>\beta_F^*$, while if   $\beta_F<\beta_F^*$ \eqref{DeflMem} has a solution.  Flores et al.~\cite{FMPS} prove that (\ref{DeflMem}) has a maximal solution for $\beta_F<\beta_F^*$ in dimensions $N \le 2$. 
	
	A generalisation of this problem has a membrane with a spatially varying permittivity profile, with the stationary equation for the membrane then given by
	\begin{equation}\label{DeflMem2}
		\Delta w=\frac{\beta_Ff(x)}{w^2},\quad w\mid_{\partial\Omega}=1.
	\end{equation}
	Ghoussoub and Guo \cite{GhGuo} show  how the change of the permittivity profile $f(x)$ affects the critical value $\beta_F^*$. Analytical and numerical techniques give upper and lower bounds on $\beta_F^*$ which depend on the permittivity profile, the nature of domain including shape, size and dimension $N$ of the domain. Taking $\beta_F=\beta_F^*$, a corresponding extremal solution of \eqref{DeflMem2} exists if $1\leq N\leq 7$, but does not exist if $N\geq8$.
	
	The dynamical problem for an elastic membrane has been of much recent interest and an elastic membrane in a MEMS capacitor has  been modelled by a semilinear  wave equation
	\begin{equation}\label{intro-3}
		\epsilon^2\frac{\partial^2w}{\partial t^2}+\frac{\partial w}{\partial t}-\Delta w=-\frac{\beta_F}{w^2},\quad w\mid_{\partial\Omega}=1,
	\end{equation}
	with the first-derivative ${\partial w}/{\partial t}$ representing damping. 
	
	Flores \cite{FG} obtains the ``pull-in'' voltage separating the  regime for which the membrane can approach a steady state, from the “touchdown” regime for which the membrane always collapses onto the rigid plate. More specifically, it is shown that  touchdown -- gap width $w$  falls to zero somewhere -- must take place in finite time if $\beta_F>\beta_F^{*}$. Also, for small enough $\beta$, depending upon the initial data (voltage is less than the ``dynamic pull-in voltage''), touchdown, i.e.~quenching, does not occur.  A related non-local model is studied in \cite{guo2014hyperbolic}.
	
	When the contribution of the inertial terms dominate, i.e.~$\epsilon^2\gg1$, from \eqref{intro-3}, we may neglect the damping term, and rescale, to obtain the simpler hyperbolic model
	\begin{equation}\label{intro-5}
		\frac{\partial^2w}{\partial t^2}-\Delta w=-\frac{\beta_F}{w^2},\quad w\mid_{\partial\Omega}=1,
	\end{equation}
	which has been extensively studied in \cite{ChL, FG, KLNT, TNLK, LZ}.
	In particular, Kavallaris et al.~\cite{TNLK} obtain local and global existence results for this  MEMS model. They also study the dynamical behaviour of solutions and  dependence on the parameter $\beta_F$. 
	The work  of Chang and Levine \cite{ChL} shows  that, for suitable initial data,  the similar second-order semilinear PDE
	\[\frac{\partial^2w}{\partial t^2}-\Delta w=-\frac{\beta_F}{w},\quad w\mid_{\partial\Omega}=1,\]
	has a weak solution $w$ in time interval $[0, T]$ for some $T>0$, and that $w$ can be extended to a time interval of the form $[0, T+\tau]$ for small $\tau$, by using an iterative scheme. Other hyperbolic MEMS equations can be seen in reference \cite{MT} and more  general semilinear wave equations are looked at in Sogge's book \cite{SC}. 
	
	When  the damping term dominates, i.e.~$\epsilon^2\ll1$, then the semilinear damped wave equation reduces to the parabolic equation studied in \cite{FMPS}:
	\begin{equation}\label{intro-4}
		\frac{\partial w}{\partial t}-\Delta w=-\frac{\beta_F}{w^2},\quad w\mid_{\partial\Omega}=1.
	\end{equation}
	In the case of two space dimensions,  Flores et al.~\cite{FMPS} prove that, for $\beta_F<\beta_F^{*}$,  the solutions $w$ of \eqref{intro-4}, with initial value $w|_{t=0}=1$, converge to the maximal steady solution $\bar{w}$ of (\ref{DeflMem}) as $t\to\infty$, while for $\beta_F>\beta_F^{*}$, touchdown occurs.   Related non-local problems can be also found in Kavallaris et al. \cite{KLN}. The work \cite{GGuo2} of Guo et al. took spatially varying permittivity so that (\ref{intro-4})  is replaced by
	\begin{equation}\label{ParabMem}
		\frac{\partial w}{\partial t}-\Delta w=-\frac{\beta_Ff(x)}{w^2},\quad w\mid_{\partial\Omega}=1,\quad w\mid_{t=0}=1,
	\end{equation}
	with $f(x)$ denoting the permittivity proflle; similar conclusions to Flores et al. \cite{FMPS} are obtained.
	
	The touchdown of a highly damped membrane corresponds to the quenching of a solution to the parabolic equation, and the quenching profile for a semilinear parabolic equation has been studied extensively in J.~S.~Guo \cite{GuoJ}. Other quenching solutions for parabolic equations can be found in  papers  \cite{boni2000quenching}, \cite{ghoussoub2008estimates}, and references \cite{drosinou2020impacts},  \cite{guo2012nonlocal}, \cite{kavallaris2008touchdown} of Kavallaris et al. To be specific, the paper \cite{kavallaris2008touchdown} studies a similar case to \eqref{ParabMem}, in which the right-hand side $-\beta_Ff(x)/w^2$ of \eqref{ParabMem} is replaced by $-\beta_F|x|^\beta/w^p$ with specific $\beta$ and $p$, the resulting equation has a global-in-time solution for initial data close to $1$, while quenching occurs for large $\beta_F$ or small initial values. Similar conclusions also hold for a non-local version of \eqref{ParabMem}. Here a non-local term $f(t)=1/\left(1+\alpha\int_{\Omega}(1/w)dx\right)^2$
	replaces the $f(x)$ above, with $\alpha$ a positive constant, and models the effect of connecting the MEMS device to a fixed capacitor as well as a constant voltage source. More details can be found in Guo and Kavallaris \cite{guo2012nonlocal}, and Guo et al.~\cite{GHW}. The recent publication \cite{drosinou2020impacts} of  Kavallaris et al.~generalises the quenching solution to stochastic parabolic equation modelling MEMS devices with random fluctuations of the potential difference.

	Recent publications on coupled systems modelling gap flow and membrane motion are less extensive, only studying the compressible version \eqref{2ndcp1-1-2} of the standard nonlinear Reynolds' equation numerically, for example in the works \cite{BY, Bl, Jr, St, SRBUCP}. In particular, Bao et al.~\cite{BYSF} study squeeze-film damping with small amplitude deflections and linearise the nonlinear Reynolds' equation \eqref{2ndcp1-1-2} around the equilibrium position. The resulting equation is a heat equation, so that analytical solutions are then available.	


Our proof of Theorem \ref{2ndcoupled system} is outlined in Subsection \ref{outlinesection}. It relies on techniques for quasilinear parabolic equations developed, for example, by Amann, Lunardi and Sinestrari \cite{AH, LS1, SE}. Semigroup methods of this kind have become a powerful tool for MEMS-related models defined by a single equation or by an elliptic-parabolic coupled system, see \cite{CL}, \cite{ELW} and the recent survey \cite{LW}. We here combine such parabolic techniques for the quasilinear Reynolds' equation \eqref{2ndcp1-1-1} with semigroup techniques for the semilinear hyperbolic equation \eqref{2ndcp1-1-2}.

The plan of the paper is as follows: In Section \ref{2ndSec:prerequisite}, we introduce the relevant Sobolev spaces and some of their basic properties, we also introduce the mild solution and strict solution for the general evolution equation and their existence results, then we discuss the steps in the proof of Theorem \ref{2ndcoupled system}. In Section \ref{2nd-order problem}, we use a solving strategy for the system \eqref{2ndcp1-1-1}, \eqref{2ndcp1-1-2} based on decoupling the equations for the gap width $w$ and the pressure $u$. We first consider the semilinear hyperbolic equation \eqref{2ndcp1-1-2} for the gap width $w$ with an arbitrarily given pressure $u$ and use semigroup techniques for \eqref{2ndcp1-1-2} to show the local well-posedness of \eqref{2ndcp1-1-2}. While the regularity theory of semilinear hyperbolic equations has been of much recent interest, we here require detailed properties of the solution operator $u\mapsto w(u)$ in order to analyse the nonlinear Reynolds' equation \eqref{2ndcp1-1-1} with abstract coefficients involving $w(u)$. For example, we prove appropriate H\"{o}lder continuity of the solution operator $u\mapsto w(u)$ in Section \ref{SolnOp}. In Section \ref{2ndsection4}, we investigate the local well-posedness of \eqref{2ndcp1-1-1} for $u$ with abstract coefficients involving $w(u)$ by using techniques for quasilinear parabolic equations.

\subsection*{Notation}
Recall that $\Omega\subset\mathds{R}$ is an open and bounded subset. Denote by $C=C(\Omega)$ a positive constant that may vary from line to line below but depends only on $\Omega$.
\begin{defn}
	Denote by $X$ a Banach space equipped with norm $\|\cdot\|_{X}$, and set $k\in\mathbb{N}$ and $T\in(0, \infty)$.  $\mathcal{B}(X)$ is a space of bounded linear operators on $X$. In the following, we shall be particularly interested in $X=L^2(\Omega)$, $L^\infty(\Omega)$,  $H^k(\Omega)$, etc. 	
     $\mathcal{B}([0,T];X)$ denotes a space which consists of all  measurable, almost everywhere bounded functions $u: [0,T]\to  X$, $t\mapsto u(t)$ with norm $\|u\|_{\mathcal{B}([0,T];X)}= \sup_{t\in[0, T]}\|u(t)\|_{X}$. If $X$ is a function space as above, we write $u(t): \Omega\to\mathds{R}$ with $x\mapsto[u(t)](x)=u(x,t)$. The closed subspace of continuous functions is denoted by $C([0,T];X)$, and
	\[C^k([0,T];X)=\left\{u:[0,T]\rightarrow X:  \tfrac{d^ju}{dt^j}\in C([0,T];X), j\in[0, k]\right\},  \|u\|_{C^k([0,T];X)}=\hspace*{-0.2cm} \sup_{t\in[0, T]}\sum_{j=0}^{k}\left\|\tfrac{d^ju(t)}{dt^j}\right\|_{X}.\]
	The definition extends to non-integer order $k+\alpha$, $\alpha\in(0, 1)$, by setting 
	\[C^\alpha([0,T];X)=\left\{u:[0,T]\rightarrow X: [u]_{C^\alpha([0, T];X)}= \sup_{0\leq t<t+h\leq T}\tfrac{\|u(t+h)-u(t)\|_X}{|h|^\alpha}<\infty\right\},\]
	\[\|u\|_{C^{\alpha}([0,T];X)}=\|u\|_{C([0,T];X)}+[u]_{C^{\alpha}([0, T];X)} , \]
	\[C^{\alpha+k}([0,T];X)=\left\{u\in C^k([0,T];X): \quad \tfrac{d^ku}{dt^k}\in C^{\alpha}([0,T];X)\right\},\]
	\[\|u\|_{C^{\alpha+k}([0,T];X)}=\|u\|_{C^k([0,T];X)}+\left[\tfrac{d^ku}{dt^k}\right]_{C^{\alpha}([0, T];X)} . \]
	 Note that  $C\left([0, T]; B_{L^2}\left(V, r\right)\right)=\Big\{v\in C\left([0, T]; L^2\left(\Omega\right)\right): \  \sup_{t\in[0, T]}\left\|v(t)-V\right\|_{L^2(\Omega)}\leq r\Big\}$ and  $C^\alpha\left([0, T]; B_{L^2}\left(V, r\right)\right)=\Big\{v\in C^\alpha\left([0, T]; L^2\left(\Omega\right)\right): \  \sup_{t\in[0, T]}\left\|v(t)-V\right\|_{L^2(\Omega)}\leq r\Big\}$ with  $V\in L^2(\Omega)$, analogously, $C\left([0, T]; B_{X}\left(U, r\right)\right)=\Big\{u\in C\left([0, T]; X\right):\ u|_{\partial\Omega}=U|_{\partial\Omega},\  \sup_{t\in[0, T]}\left\|u(t)-U\right\|_{X}\leq r\Big\}$  and  $C^\alpha\left([0, T]; B_{X}\left(U, r\right)\right)=\Big\{u\in C^\alpha\left([0, T]; X\right):\quad u|_{\partial\Omega}=U|_{\partial\Omega},\quad  \sup_{t\in[0, T]}\left\|u(t)-U\right\|_{X}\leq r\Big\}$ with $U\in X$, where $X=H^1(\Omega)$, $H^2(\Omega)$. 
	
	Let $\mathcal{P}: D(\mathcal{P}) \subset X \to X$ be an unbounded linear operator which generates an analytic semigroup $e^{\mathcal{P}t}$, and define intermediate space $D_{\mathcal{P}}(\alpha, \infty)$  as follows:
	\[D_{\mathcal{P}}(\alpha, \infty)=\left\{v\in X:\ \|v\|_{\alpha}=\sup_{t>0}\|t^{1-\alpha}\mathcal{P}e^{\mathcal{P}t}v\|_{X}<\infty\right\}.\]
	It is a Banach space with respect to the norm
	$\|v\|_{D_{\mathcal{P}}(\alpha, \infty)}=\|v\|_{X}+\|v\|_{\alpha}$.
	Its closed subspace
	$D_{\mathcal{P}}(\alpha)=\left\{v\in X:\  \lim_{t\rightarrow 0}t^{1-\alpha}\mathcal{P}e^{\mathcal{P}t}v=0\right\}$
	inherits the norm of $D_{\mathcal{P}}(\alpha, \infty)$. 
\end{defn}
Our main results on the well-posedness of the semilinear wave equation  will be shown by constructing a Picard iteration in the complete metric space  $Z(T)$, given by \eqref{2ini-NBD}, with  $\tilde{v}_0=v_0,\ \tilde{w}_0=w_0-\theta_2$,
\begin{align}
		Z(T):=\bigg\{\hspace*{-0.1cm}\begin{pmatrix}
			\tilde{v}\\ \tilde{w}
		\end{pmatrix}\hspace*{-0.1cm}\in C\left([0,T]; L^2(\Omega)\times H_0^1(\Omega)\right)\hspace*{-0.1cm}: \begin{pmatrix}\tilde{v}(0)\\ \tilde{w}(0)\end{pmatrix}=\begin{pmatrix}\tilde{v}_0\\ \tilde{w}_0\end{pmatrix}\hspace*{-0.1cm},\
		\sup_{t\in[0, T]}\left\|\begin{pmatrix}\tilde{v}(t)-\tilde{v}_0\\ \tilde{w}(t)-\tilde{w}_0\end{pmatrix}\right\|_{L^2\times H^{1}(\Omega)}\hspace*{-0.1cm}\leq r\bigg\} . \label{2ini-NBD}
\end{align}

\section{Preliminaries and Outline}\label{2ndSec:prerequisite}

The approach in this article is functional analytic. In this section we recall standard notions and results for abstract evolution equations and discuss the steps in the proof of Theorem \ref{2ndcoupled system}.

\subsection{Abstract Evolution Equations}

\begin{defn}
	Let $\mathscr{X} $ be a Banach space, $\mathcal{A}: D(\mathcal{A}) \subset \mathscr{X} \to \mathscr{X}$ a linear, unbounded operator which  generates a strongly continuous semigroup ($C_0$-semigroup) $\{T(t): t\geq 0\}$. Further, let $T\in(0, \infty)$, $\mathcal{G}\in C([0, T]; \mathscr{X})$ and $\Phi_0\in \mathscr{X}$. A function $\Phi$ is called a mild solution of the  inhomogeneous evolution equation
	\be\label{IEE}
	\Phi'(t)=\mathcal{A}\Phi(t)+\mathcal{G}(t), \quad t\in [0, T],\quad \Phi(0)=\Phi_0,
	\ee
	if $\Phi\in C([0, T]; \mathscr{X})$ is given by the integral formulation
	\be\label{linear solu}
	\Phi(t)=T(t)\Phi_0+\int_0^tT(t-s)\mathcal{G}(s)ds,\quad t\in[0, T].
	\ee
	A function $\Phi$ is said to be a strict solution of \eqref{IEE}, if $\Phi\in C([0, T]; D(\mathcal{A}))\cap C^1([0, T]; \mathscr{X})$ is given by the integral formulation \eqref{linear solu} and satisfies \eqref{IEE}.
\end{defn}

\begin{lem}\label{IEE-S}
	Consider a linear operator $\mathcal{A}$ on a Banach space $\mathscr{X}$ which generates a $C_0$-semigroup $\{T(t): t\geq 0\}$, $T\in(0, \infty)$. Further, let $\Phi_0\in D(\mathcal{A})$ and  $\mathcal{G}\in C([0, T]; \mathscr{X})$. If $\Phi$ is a solution of the inhomogeneous evolution equation \eqref{IEE} on $[0, T]$, then $\Phi$ is given by the integral formulation \eqref{linear solu}.
	
	Assume that, in addition, either $\mathcal{G}\in C([0, T]; D(\mathcal{A}))$ or $\mathcal{G}\in C^1([0, T]; \mathscr{X})$. Then the mild solution $\Phi$ on $[0, T]$ defined by \eqref{linear solu} uniquely solves the inhomogeneous evolution equation \eqref{IEE}, and
	\[\Phi\in C\left([0, T]; D(\mathcal{A})\right)\cap C^1\left([0, T]; \mathscr{X}\right).\]
\end{lem}
We refer to Theorem 6.9 in \cite{schnaubelt} for the proof of Lemma \ref{IEE-S}.
\begin{lem}\label{time-derivative-continuity}
	Let $\mathscr{X}$ be a Banach space and $\Phi\in C([0, T]; \mathscr{X})$ be differentiable from the right with right derivative $\Psi\in C\left([0, T]; \mathscr{X}\right)$. Then $\Phi\in C^1\left([0, T]; \mathscr{X}\right)$ and $\Phi'=\Psi$.
\end{lem}
We refer to Lemma 8.9 in \cite{schnaubelt} for the proof of Lemma \ref{time-derivative-continuity}.

\subsection{Outline of the Proof of Theorem \ref{2ndcoupled system}}\label{outlinesection}

	We now outline the key steps of the main result in this article, Theorem \ref{2ndcoupled system}. Instead of considering \eqref{2ndcp1}, we look for a unique strict short-time solution $(u, v, w)$ of the following, equivalent coupled system \eqref{2ndcp3}:
	\bse\label{2ndcp3}
	\be\label{2ndcp3-1}
	\frac{\partial u}{\partial t}=\frac{1}{w}\frac{\partial}{\partial x}\left(w^3u\frac{\partial u}{\partial x}\right)-\frac{v}{w}u,\quad x\in\Omega,\ t\geq 0;
	\ee
	\be\label{2ndcp3-2}
	\frac{\partial v}{\partial t}=\frac{\partial^2w}{\partial x^2}-\frac{\beta_F}{w^2}+\beta_p(u-1),\quad x\in\Omega,\ t\geq 0;
	\ee
	\be\label{2ndcp3-3}
	\frac{\partial w}{\partial t}=v,\quad x\in\Omega,\ t\geq 0.
	\ee
	\ese
	The initial values are given by $u(x,0)=u_0(x)$, $v(x,0)=v_0(x)$, $w(x,0)=w_0(x)$, $x\in {\Omega}$, and boundary values by $u(x,t)=\theta_1$,  $w(x,t)=\theta_2$,  $x\in\partial\Omega$, $t\geq 0$.
	
	Section \ref{2nd-order problem} shows that there exists a unique solution $(v, w)$ of the hyperbolic sub-system \eqref{2ndcp3-2}, \eqref{2ndcp3-3} for given, appropriately regular $u$, initial values $v(x,0)=v_0(x)$, $w(x,0)=w_0(x)$, $x\in {\Omega}$ and boundary values $w(x,t)=\theta_2$, $x\in\partial\Omega,\ t\geq 0$. Section \ref{SolnOp} then establishes relevant properties of solution operators $u\mapsto (v,w)=(v(u),w(u))$ for short time intervals $[0,T]$, such as Theorem \ref{2ndt1}, which is a restatement of 
Theorem \ref{2nd-cpl} in Section 4. 
	\begin{thm}\label{2ndt1}
		The solution operator
		\[ W: C\left([0, T]; B_{H^2}(u_0,  r)\right)\to C\left([0, T]; B_{L^2}(v_0,  r)\times B_{H^1}(w_0,  r)\right) \] \[u \mapsto W(u)=\left(v, w\right)=\left(v(u), w(u)\right)\]
		is Lipschitz continuous with respect to $u$, i.e.
		\begin{equation*}
			\left\|W({u}_1)-W({u}_2)\right\|_{C\left([0, T]; L^{2}(\Omega)\times H^{1}(\Omega)\right)}\leq L_{W}\|{u}_1-{u}_2\|_{C\left([0, T]; H^2(\Omega)\right)},
		\end{equation*}
		where $r>0$ is sufficiently small, $L_{W}>0$ is a Lipschitz constant,
		\[B_{L^2}(V,  r)=\left\{f\in L^2(\Omega):\  \|f-V\|_{L^2(\Omega)}\leq r\right\}, \text{~with~} V\in L^2(\Omega),\]
		\[B_{X}(U,  r)=\left\{f\in X:\ f|_{\partial\Omega}=U|_{\partial\Omega},\ \|f-U\|_{X}\leq r\right\},\quad \text{where}\ X=H^1(\Omega),\ H^2(\Omega) \text{~and~} U\in X.\]
	\end{thm}
The arguments further give information about the Fr\'{e}chet derivative of the solution operator $W$, as stated in the following two corollaries, Corollary \ref{2ndc2} and Corollary \ref{2ndc3}, which are the same as 
 Corollary \ref{2Lip-Frechet-W} and Corollary \ref{2Holder-Frechet-W-I-Cor}, respectively, in Section 4, but stated differently. 
	\begin{cor}\label{2ndc2}
		Given $u\in C\left([0, T]; B_{H^2}(u_0,  r)\right)$ and small $r>0$,  the Fr\'{e}chet derivative $W'(u)$ of $W(u)$, 
		\[{W}{'}(u): C\left([0,T]; H^2(\Omega)\cap H_0^1(\Omega)\right)\to C\left([0,T]; L^2(\Omega)\times H_0^1(\Omega)\right),\]
		\[q\mapsto W'(u)q=\left(v'(u)q, w'(u)q\right)\]
		is Lipschitz continuous with respect to $u$, i.e.~for $\left\|q\right\|_{C\left([0,T]; H^2(\Omega)\right)}\leq 1$,
		\[\left\|{W}'(u_1)q-{W}'(u_2)q\right\|_{C\left([0, T]; L^{2}(\Omega)\times H^{1}(\Omega)\right)}\leq L_{F}\left\|{u}_1-{u}_2\right\|_ {C\left([0, T]; H^2(\Omega)\right)}.\]
		Here $L_F$ is a Lipschitz constant.
	\end{cor}
	\begin{cor}\label{2ndc3}
		If $r>0$ is small and given $u\in C^\alpha\left([0, T]; B_{H^2}(u_0,  r)\right)$,
		then there exists a Lipschitz constant $L_M>0$, such that
		\begin{align}
			\sup_{0\leq t<t+h\leq T}\left\|[{W}'(u)q](t+h)-[{W}'(u)q](t)\right\|_{L^{2}(\Omega)\times H^{1}(\Omega)}\leq&h^\alpha L_{M}\left\|q\right\|_{C\left([0,T]; H^2(\Omega)\right)}\notag\\
			+&h^\alpha TL_{M}\left\|q\right\|_{C^\alpha\left([0, T]; H^2(\Omega)\right)}\notag
		\end{align}
		holds for all $q\in C^\alpha([0, T]; B_{H^2}(\tilde{u}_0,  r))$, where  $\tilde{u}_0=u_0-\theta_1$,.
	\end{cor}
Based on these results for the solution operator $u\mapsto (v,w)=(v(u),w(u))$ to the hyperbolic problem,  we obtain an existence result for the coupled system \eqref{2ndcp1} in Section \ref{2ndsection4}. To do so, we reformulate the system \eqref{2ndcp1} as an abstract quasilinear parabolic equation involving $v(u)$ and $w(u)$:
	\bse\label{2nd1QPE}
	\be\label{2nd1QPE-1}
	\frac{\partial u}{\partial t}=\frac{1}{w(u)}\frac{\partial}{\partial x}\left([w(u)]^3u\frac{\partial u}{\partial x}\right)-\frac{v(u)}{w(u)}u,\quad (x,t)\in\Omega\times(0, T),
	\ee
	\be\label{2nd1QPE-2}
	u(x,0)=u_0(x),\quad x\in\Omega,\quad u(x,t)=\theta_1,\quad (x,t)\in\partial\Omega\times[0, T].
	\ee
	\ese
	Using a contraction mapping argument, we show that the solution of \eqref{2nd1QPE} exists as long as\\ $\left(v(u), w(u)\right)\in C\left([0, T]; B_{L^2}(v_0, r)\times B_{H^1}(w_0, r)\right)$ for small $r>0$ and $T>0$.
	
	More precisely, we set $\tilde{u}=u-\theta_1$
and consider the operator 
	\[F(\tilde{u})=\frac{1}{w(\tilde{u}+\theta_1)}\frac{\partial}{\partial x}\left([w(\tilde{u}+\theta_1)]^3(\tilde{u}+\theta_1)\frac{\partial\tilde{u}}{\partial x} \right)-\frac{v(\tilde{u}+\theta_1)}{w(\tilde{u}+\theta_1)}(\tilde{u}+\theta_1).\]
	Theorem \ref{2ndt1}, Corollary \ref{2ndc2} and Corollary \ref{2ndc3} imply H\"{o}lder estimates  for the nonlinear operator $F$, i.e.~Theorem \ref{2nd1}, which slightly simplifies the statement of Lemma \ref{2ndRHSMax} in Section 5:
	\begin{thm}\label{2nd1}
		Fix $T>0$. If $\tilde{u},\ q\in C^\alpha\left([0, T]; B_{H^2}(\tilde{u}_0, r)\right)$, with $u_0 = \tilde{u}_0 +\theta_1$, then there exist constants $L_A, L_B>0$, such that for $0\leq t<t+h\leq T$,
		\be
		\left\|\left[F(\tilde{u})\right](t+h)-\left[F(\tilde{u})\right](t)\right\|_{L^2(\Omega)}\leq \left\{[\tilde{u}+\theta_1]_{C^\alpha\left(\left[0, T\right]; H^2(\Omega)\right)}+L_U\right\}L_Ah^\alpha,
		\ee
		\begin{align}
			&\left\|\left[F'(\tilde{u})q\right](t+h)-\left[F'(\tilde{u})q\right](t)-\mathcal{P}^*\left[q(t+h)-q(t)\right]\right\|_{L^2(\Omega)}\notag\\
			\leq& h^\alpha T^\alpha L_{B}\left\|q\right\|_{C^\alpha([0, T]; H^2(\Omega))}+h^\alpha T^\alpha L_{B}\left\|\tilde{u}+\theta_1\right\|_{C^\alpha([0, T]; H^2(\Omega))}\left\|q\right\|_{C^\alpha([0, T]; H^2(\Omega))}\notag\\
			+&h^\alpha L_B\left\|q\right\|_{C([0, T]; H^2(\Omega))}+h^\alpha L_B\left\|\tilde{u}+\theta_1\right\|_{C^\alpha([0, T]; H^2(\Omega))}\left\|q\right\|_{C([0, T]; H^2(\Omega))}.
		\end{align}
Lemma \ref{2ndRHSMax} further specifies the dependence of $L_A$ and $L_B$ on the given data of the problem.
	\end{thm}
The linearisation $\mathcal{P}^*$ of $F$ around the initial condition $\tilde{u}_0$  
is shown to generate an analytic semigroup $\left\{e^{t\mathcal{P}^*}:\ t\geq 0\right\}$. 

To prove the short-time existence result for the nonlinear problem \eqref{2nd2QPE}, we now rewrite \eqref{2nd1QPE} in the form
	\be\label{2nd2QPE}
	\tilde{u}'(t)=\mathcal{P}^*\tilde{u}(t)+[F(\tilde{u})](t)-\mathcal{P}^*\tilde{u}(t),\quad t\in(0, T),\quad \tilde{u}(0)=\tilde{u}_0 ,
	\ee
and use a fixed point argument for the nonlinear mapping $\Gamma$ on a suitable space, defined by
	\be\label{2nd0}
	\Gamma(\tilde{u}(t))=e^{t\mathcal{P}^*}\tilde{u}_0+ \int_0^te^{(t-s)\mathcal{P}^*}\left\{[F(\tilde{u})](s)-\mathcal{P}^*\tilde{u}(s)\right\}ds,\quad t\in [0, T].
	\ee
%
	Combining the existence and regularity results from Section 3 with the existence of a unique strict solution of \eqref{2nd2QPE}, we conclude the proof of Theorem \ref{2ndcoupled system}.

\section{Well-Posedness Results for the Hyperbolic Equation}\label{2nd-order problem}


Let $u_0\in H^{2+\sigma}(\Omega)$, $v_0\in H_0^1(\Omega)$ and $w_0\in H^2(\Omega)$ be given functions which are compatible with  boundary conditions $u_0|_{\partial\Omega}=\theta_1$ and $w_0|_{\partial\Omega}=\theta_2$, where  $\sigma\in(0, \frac{1}{2})$. Set $r\in\left(0, \frac{\kappa}{2C}\right)$ with a constant $C=C(\Omega)>0$  and $\kappa=\inf_{x\in{\Omega}}w_0(x)$.  We introduce the state space
\be\label{2state-space}
\mathscr{X}=L^2(\Omega)\times H_0^1(\Omega),
\ee
where $\mathscr{X}$ is endowed with the norm $\left\|\cdot\right\|_{\mathscr{X}}=\left\|\cdot\right\|_{L^2(\Omega)\times H^1(\Omega)}$ and the scalar product
\[\langle \mathbf{a}, \mathbf{b}\rangle_{\mathscr{X}}= \int_\Omega a_1\cdot\overline{b}_1+\nabla a_2\cdot\nabla\overline{b}_2dx,\quad \mathbf{a}=\left(a_1, a_2\right)\in \mathscr{X}, \quad \mathbf{b}=\left(b_1, b_2\right)\in \mathscr{X}.\]
We define an operator $\Delta_D$ by
\be\label{2domain1}D(\Delta_D):=\big\{\varphi\in H_0^1(\Omega): \exists\ f\in L^2(\Omega),\ \forall\ \psi\in H_0^1(\Omega),\ \text{such that} \int_\Omega\nabla\varphi\cdot\nabla\overline{\psi}dx=\int_\Omega f\cdot\overline{\psi}dx \big\},\ee
\be\label{2domain2}	\Delta_D\varphi:=-f,\ \text{where} \ f \ \text{is given by} \ D(\Delta_D), \quad \left\|\varphi\right\|_{D(\Delta_D)}:=\left\|\varphi\right\|_{L^2(\Omega)}+\left\|\Delta_D\varphi\right\|_{L^2(\Omega)}.
\ee
From elliptic regularity theory, it follows that
\[D(\Delta_D)=\left\{\varphi\in H^{2}(\Omega): \ \varphi|_{\partial\Omega}=0\right\}=H^2(\Omega)\cap H_0^1(\Omega),\quad\left\|\varphi\right\|_{D(\Delta_D)}\simeq\|\varphi\|_{H^2(\Omega)}.\]
We also define the linear operator $\mathcal{A}$ with its domain $D(\mathcal{A})$ and the graph norm of $\mathcal{A}$ by
\bse\label{2A-1}
\be\label{2A-1-1}
\mathcal{A}=\begin{pmatrix}0\ &\Delta_D\\ 1\ &0\end{pmatrix}, \quad D(\mathcal{A})= H_0^1(\Omega)\times D(\Delta_D) ,
\ee
\be\label{2A-1-2}
\|\mathbf{a}\|_{D(\mathcal{A})}:=\|\mathbf{a}\|_{X}+\|\mathcal{A}\mathbf{a}\|_{X}\simeq\|a_1\|_{H^1(\Omega)}+\|a_2\|_{H^2(\Omega)}, \quad\mathbf{a}=\left(a_1, a_2\right)\in D(\mathcal{A}).
\ee
\ese
Let $T\in(0, \infty)$ to be specified below. We now study the initial-boundary value problem for the semilinear hyperbolic equation \eqref{2ndcp1-1-2} for $w$ for a given, fixed function   $u\in C\left([0, T]; B_{H^2}\left(u_0, r\right)\right)$,   subject to initial values 
\be\label{2nd-semilinear-hyperbolic-equation-2}
w(x,0)=w_0(x),\quad\frac{\partial w}{\partial t}(x,0)=v_0(x), \quad x\in\Omega,
\ee
and  Dirichlet boundary conditions
\be\label{2nd-semilinear-hyperbolic-equation-3}
w(x,t)=\theta_2, \quad (x,t)\in\partial\Omega\times[0, T].
\ee
We define $\tilde{w}(x,t)=w(x,t)-\theta_2$ with $\tilde{w}(t):\Omega\to \mathbb{R}$, $x\mapsto [\tilde{w}(t)](x)=\tilde{w}(x, t)$. Note that the Dirichlet boundary conditions $\tilde{w}(x,t)=0$, $(x,t)\in\partial\Omega\times[0, T]$ are incorporated in the domain of the Dirichlet realisation $\Delta_D$ of the Laplace operator  in \eqref{2domain1} and \eqref{2domain2}. 
We now rewrite \eqref{2ndcp1-1-2} with \eqref{2nd-semilinear-hyperbolic-equation-2} and \eqref{2nd-semilinear-hyperbolic-equation-3} as the equation \eqref{2nd-LWE-1} on the unknown function $\tilde{w}$:
\be\label{2nd-LWE-1}
\tilde{w}''(t)=\Delta_D\tilde{w}(t)-\frac{\beta_F}{(\tilde{w}(t)+\theta_2)^2}+\beta_p(\tilde{u}(t)+\theta_1-1),\quad t\in[0, T],\quad
\tilde{w}(0)=\tilde{w}_0,\quad \tilde{w}'(0)=\tilde{v}_0,
\ee
where $\tilde{w}'$ and $\tilde{w}''$ denote, respectively, the first and second derivatives of the unknown function $\tilde{w}$ with respect to $t\in (0, T)$;  $\tilde{u}=u-\theta_1$ is given in $ C\left([0, T]; B_{H^2}\left(\tilde{u}_0, r\right)\right)$. Observe that \[\tilde{u}_0=u_0-\theta_1\in H^{2+\sigma}(\Omega)\cap H_0^1(\Omega), \ \sigma\in\left(0, {1}/{2}\right),\ \tilde{v}_0=v_0\in H_0^1(\Omega),\ \tilde{w}_0=w_0-\theta_2\in H^2(\Omega)\cap H_0^1(\Omega).\]    
For a time-dependent state $\Phi(t)=\left(\varphi_1(t),\ \varphi_2(t)\right)$ we write
\be\label{2ini-Phi}
\Phi_0=\left(\tilde{v}_0, \tilde{w}_0\right)\in D(\mathcal{A})
\ee 
and consider the nonlinear operator
\be\label{2G-2-2}
[G(\varphi_2)](t)=-\frac{\beta_F}{(\varphi_2(t)+\theta_2)^2}+\beta_p(\theta_1-1),\ \ \mathcal{G}=\left[\mathcal{G}\left(\Phi\right)\right](t)=\left(\left[G\left({\varphi}_2\right)\right](t)+\beta_p\tilde{u}(t), 0\right).
\ee
The existence of a unique strict solution of the equation \eqref{2nd-LWE-1} is shown in Theorem \ref{2nd-mild-solution-cor}, which is proved using  
Lemma \ref{2equivalence-1} 
to Corollary \ref{2Holdercontinuity}. 
The proofs of these results, from Lemma \ref{2equivalence-1} to Theorem \ref{2nd-mild-solution-cor}, follow from adapting well-known arguments to  equation \eqref{2nd-LWE-1} and are presented in Appendix \ref{AppB-1}. These auxiliary results lead to
the well-posedness of the semilinear hyperbolic equation \eqref{2ndcp1-1-2} subject to initial values  \eqref{2nd-semilinear-hyperbolic-equation-2} and boundary conditions \eqref{2nd-semilinear-hyperbolic-equation-3}, for a given function $u$.

The first Lemma gives an equivalent, abstract formulation of the problem:
\begin{lem}\label{2equivalence-1}
	Let $\tilde{u}\in C\left([0, T]; B_{H^2}\left(\tilde{u}_0, r\right)\right)\cap C^1\left([0, T]; L^2(\Omega)\right)$ be given, and define $\mathcal{A}$ and $\mathcal{G}$ by \eqref{2A-1} and \eqref{2G-2-2} respectively. The semilinear hyperbolic equation \eqref{2nd-LWE-1} has a unique solution
	\[\tilde{w}\in C^2([0, T]; L^2(\Omega))\cap C^1([0, T]; H_0^1(\Omega))\cap C([0, T]; H^2(\Omega)\cap H_0^1(\Omega))\]
	if and only if the semilinear evolution equation
	\be\label{2nd-IEE}
		\Phi'(t)=\mathcal{A}\Phi(t)+\left[\mathcal{G}\left(\Phi\right)\right](t), \quad t\in [0, T],\quad \Phi(0)=\Phi_0,
	\ee
	has a unique solution
	\[\Phi\in C([0, T]; D(\mathcal{A}))\cap C^1([0, T]; X).\]
	In this case, $\Phi=\left(\tilde{w}', \tilde{w}\right)$.
\end{lem}
We next recall that the operator $\mathcal{A}$ is the generator of a strongly continuous semigroup.
\begin{lem}\label{2ndgenerator}
Let $\Omega$ be an open and bounded subset of $\mathds{R}$ and $\mathscr{X}$ as in \eqref{2state-space}. Then the linear operator $\mathcal{A}$ defined by \eqref{2A-1} generates a $C_0$-semigroup $\left\{T(t)\in\mathcal{B}\left(\mathscr{X}\right)\hspace*{-0.1cm}: t\in[0, \infty)\right\}$.	
\end{lem}
Standard arguments imply the existence of a unique mild solution to the abstract problem from Lemma \ref{2equivalence-1}. 
\begin{thm}\label{2nd-solu-thm}
	For $r\in\left(0, \frac{\kappa}{2C}\right)$ with $\kappa= \inf_{x\in{\Omega}}\left\{\tilde{w}_0(x)+\theta_2\right\}$ and a constant $C=C(\Omega)>0$,  there exists $T_0>0$,  such that for all $T\in(0, T_0)$ and given $\tilde{u} \in C\left([0, T]; B_{H^2}\left(\tilde{u}_0, r\right)\right)$, the semilinear evolution equation 
	\be\label{2nd-SWE-1}
	\begin{pmatrix}\tilde{v}'(t) \\ \tilde{w}'(t) \end{pmatrix}=\mathcal{A}\begin{pmatrix}\tilde{v}(t) \\ \tilde{w}(t) \end{pmatrix}+\begin{pmatrix}[G(\tilde{w})](t)+\beta_p\tilde{u}(t), \\ 0 \end{pmatrix}, \ t\in[0, T],\quad \begin{pmatrix}\tilde{v}(0) \\ \tilde{w}(0) \end{pmatrix}=\begin{pmatrix}\tilde{v}_0 \\ \tilde{w}_0\end{pmatrix},
	\ee
	has a unique mild solution $\left(\tilde{v}, \tilde{w}\right)\in  C\left([0, T]; L^2(\Omega)\times H_0^1(\Omega)\right)$ defined by the integral formulation
	\be\label{2ndmild-solu-form}
	\begin{pmatrix}\tilde{v}(t)\\ \tilde{w}(t)\end{pmatrix}=T(t)\begin{pmatrix}\tilde{v}_0\\ \tilde{w}_0\end{pmatrix}+
	 \int_0^t\left\{T(t-s)\begin{pmatrix}[G(\tilde{w})](s)+\beta_p\tilde{u}(s)\\ 0\end{pmatrix}\right\}ds.
	\ee
\end{thm}
A refined analysis establishes H\"{o}lder and Lipschitz estimates in time, as stated in the following corollaries. 
\begin{cor}\label{2Lip-mild-solu}
	Let $T\in(0, T_0)$, $r\in\left(0, \frac{\kappa}{2C}\right)$ and given $\tilde{u}\in C\left([0, T]; B_{H^2}\left(\tilde{u}_0,  r\right)\right)\cap C^1([0, T]; L^2(\Omega))$. Then the mild solution of  \eqref{2nd-SWE-1},
	$\left(\tilde{v}, \tilde{w}\right): [0, T]\to L^2(\Omega)\times H_0^1(\Omega)$,
	defined by the integral formulation \eqref{2ndmild-solu-form}, is locally Lipschitz continuous with respect to $t\in[0, T]$, i.e.~$\forall\ h\in(0, T]$,
	\be\label{2Lip-mild-solu-inquality}
	\sup_{0\leq t<t+h\leq T}\left\|\begin{pmatrix}\tilde{v}(t+h)-\tilde{v}(t)\\ \tilde{w}(t+h)-\tilde{w}(t)\end{pmatrix}\right\|_{L^2(\Omega)\times H^1(\Omega)}\leq L_Vh.
	\ee
	Here $L_V$ is a Lipschitz constant depending on the data $\beta_F$, $\beta_p$, $T_0$, $\kappa$, $\Omega$, $\left\|\tilde{u}_0\right\|_{H^{2}(\Omega)}$, $\|\tilde{w}_0\|_{H^1(\Omega)}$, \\ $\left\|\left(\tilde{v}_0, \tilde{w}_0\right)\right\|_{D(\mathcal{A})}$,
	$M_0=  \sup_{t\in[0, \infty)}\left\|T(t)\right\|_{\mathcal{B}(L^2(\Omega)\times H_0^1(\Omega))}$ and $\left\|\tilde{u}\right\|_{C^1\left([0, T_0); L^2(\Omega)\right)}$.
\end{cor}

\begin{cor}\label{2Holdercontinuity}
	If  $T\in(0, T_0)$, $r\in\left(0, \frac{\kappa}{2C}\right)$ and given
	$\tilde{u}\in C^\alpha\left([0, T]; B_{H^2}(\tilde{u}_0, r)\right)$ with $\alpha\in(0, 1)$,
	then the mild solution of \eqref{2nd-SWE-1},
	$\left(\tilde{v}, \tilde{w}\right): [0, T]\to L^2(\Omega)\times H_0^1(\Omega)$,
	defined by the integral formulation \eqref{2ndmild-solu-form}, is locally H\"{o}lder continuous with exponent $\alpha$ with respect to $t\in[0, T]$, i.e.~for all $h\in(0, T]$,
	\be\label{2Holdercontinuityformular}
	\sup_{0\leq t<t+h\leq T}\left\|\begin{pmatrix}\tilde{v}(t+h)-\tilde{v}(t)\\ \tilde{w}(t+h)-\tilde{w}(t)\end{pmatrix}\right\|_{L^2(\Omega)\times H^1(\Omega)}\leq L_Uh^\alpha.
	\ee
	Here $L_U$ is a Lipschitz constant depending on $\alpha$,  $T_0$, $\Omega$,  $\beta_p$, $\beta_F$, $\kappa$, $\left\|\tilde{u}_0\right\|_{H^2(\Omega)}$,  $\left\|\tilde{w}_0\right\|_{H^{1}(\Omega)}$,\\ $\left\|\left(\tilde{v}_0, \tilde{w}_0\right)\right\|_{D(\mathcal{A})}$ and $M_0=  \sup_{t\in[0, \infty)}\left\|T(t)\right\|_{\mathcal{B}\left(L^2(\Omega)\times H_0^1(\Omega)\right)}$.
\end{cor}
We finally conclude that the mild solution from Theorem \ref{2nd-solu-thm} is a strict solution:
\begin{thm}\label{2nd-mild-solution-cor}
	For $T\in(0, T_0)$ and given $\tilde{u}\in C\left([0, T]; B_{H^2}\left(\tilde{u}_0,  r\right)\right) \cap  C^1([0, T]; L^2(\Omega))$, the mild solution $\left(\tilde{v}, \tilde{w}\right)$ of  \eqref{2nd-SWE-1}, defined by  \eqref{2ndmild-solu-form}, is the strict solution of  \eqref{2nd-SWE-1} with
	 \[\left(\tilde{v}, \tilde{w}\right)\in C^1\left([0, T]; L^2(\Omega)\times H_0^1(\Omega)\right)\cap C\left([0, T]; H_0^1(\Omega)\times \left\{H^2(\Omega)\cap H_0^1(\Omega)\right\}\right).\]
\end{thm}


\section{Properties of the Solution Operator to the Hyperbolic Equation}\label{SolnOp}


\begin{thm}\label{2nd-cpl}
	Let $T_0$ be given by Theorem \ref{2nd-solu-thm} and $T\in (0, T_0)$. Then a solution operator $W_1$,
	\[W_1: C\left([0, T]; B_{H^2}\left(\tilde{u}_0, r\right)\right)\to Z(T),\quad \tilde{u}\mapsto  W_1(\tilde{u})=\left(\tilde{v}, \tilde{w}\right)=\left(\tilde{v}(\tilde{u}), \tilde{w}(\tilde{u})\right) ,
\]
	where	
\[ [W_1(\tilde{u})](t)=  T(t)\begin{pmatrix}\tilde{v}_0\\  \tilde{w}_0\end{pmatrix}+
		\int_0^t\left\{T(t-s)
		\begin{pmatrix}
			[G(\tilde{w})](s)+\beta_p\tilde{u}(s)\\
			0
		\end{pmatrix}
		\right\}ds,\quad t\in[0, T],\]
	satisfies  Lipschitz continuity, i.e.
	\be\label{2Lip-wrt-u}
	 \sup_{t\in[0, T]}\left\|[W_1(\tilde{u}_1)](t)-[W_1(\tilde{u}_2)](t)\right\|_{L^{2}(\Omega)\times H^{1}(\Omega)}\leq L_{W}\sup_{t\in[0, T]}\|\tilde{u}_1(t)-\tilde{u}_2(t)\|_{H^2(\Omega)}.
	\ee
	Here $r\in\left(0, \frac{\kappa}{2C}\right)$;  $Z(T)$ and $G(\tilde{w})$ are defined by  \eqref{2ini-NBD} and \eqref{2ndG-def}  respectively; $L_{W}$ is a Lipschitz constant depending on $T_0$, $M_0=  \sup_{t\in[0,\infty)}\left\|T(t)\right\|_{\mathcal{B}\left(L^2(\Omega)\times H_0^1(\Omega)\right)}$, $\kappa$, $\|w_0\|_{H^{1}(\Omega)}$, $\Omega$,  $\beta_p$ and $\beta_F$.\\	
	Furthermore, define
	\[  W_2: C\left([0, T]; B_{H^2}\left(\tilde{u}_0,  r\right)\right) \to  C\left([0,T]; L^2(\Omega)\right),\quad \tilde{u}\mapsto \frac{\tilde{v}}{\tilde{w}+\theta_2}.\]
	Then $W_2(\tilde{u})$ also depends Lipschitz-continuously on $\tilde{u}\in C\left([0, T]; B_{H^2}\left(\tilde{u}_0, r\right)\right)$, i.e.
	\be\label{2Lip-frac-W}
	\sup_{t\in[0,T]}\left\|[W_2(\tilde{u}_1)](t)-[W_2(\tilde{u}_2)](t)\right\|_{L^{2}(\Omega)}\leq L_{W_2}\sup_{t\in[0,T]}\|\tilde{u}_1(t)-\tilde{u}_2(t)\|_{H^2(\Omega)}.
	\ee
	Here $L_{W_2}$ is a Lipschitz constant depending on the above $L_{W}$ and $\left\|\tilde{v}_0\right\|_{L^2(\Omega)}$.
\end{thm}
\begin{proof}
	Let $T\in(0, T_0)$, $\tilde{u}_1$, $\tilde{u}_2\in C\left([0, T]; B_{H^2}\left(\tilde{u}_0, r\right)\right)$ be such that $W_1(\tilde{u}_1)=\left(\tilde{v}_1, \tilde{w}_1\right)\in Z(T)$ and $W_1(\tilde{u}_2)=\left(\tilde{v}_2, \tilde{w}_2\right)\in Z(T)$, then
	$[G(\tilde{w}_1)](t)-[G(\tilde{w}_2)](t)+\beta_p\left[\tilde{u}_1(t)-\tilde{u}_2(t)\right]\in H^1(\Omega)$ for all $ t\in[0, T]$,
	and it follows that
	\begin{align}
		&\|[W_1(\tilde{u}_1)](t)-[W_1(\tilde{u}_2)](t)\|_{L^2(\Omega)\times H^{1}(\Omega)}\notag\\
		=&\left\|\int_0^t  T(t-s)\begin{pmatrix}[G(\tilde{u}_1,\tilde{w}_1)](s)-[G(\tilde{u}_2,\tilde{w}_2)](s)\\ 0\end{pmatrix}ds\right\|_{L^2(\Omega)\times H^{1}(\Omega)}\notag\\
		\leq&M_0\int_0^t\left\|[G(\tilde{w}_1)](s)-[G(\tilde{w}_2)](s)+\beta_p\left[\tilde{u}_1(s)-\tilde{u}_2(s)\right]\right\|_{L^{2}(\Omega)}ds\notag\\
		\leq&M_0\int_0^t\left\|[G(\tilde{w}_1)](s)-[G(\tilde{w}_2)](s)+\beta_p\left[\tilde{u}_1(s)-\tilde{u}_2(s)\right]\right\|_{H^{1}(\Omega)}ds,\label{2W estimate-11}
	\end{align}
	where $M_0=  \sup_{t\in[0,\infty)}\left\|T(t)\right\|_{\mathcal{B}\left(L^2(\Omega)\times H_0^1(\Omega)\right)}$. Since the estimate \eqref{2ndLip-G} of $G$ from Lemma \ref{2ndLip-G-Lem} and
	\[\|\tilde{w}_1(t)-\tilde{w}_2(t)\|_{H^{1}(\Omega)}\leq\left\|\begin{pmatrix}\tilde{v}_1(t)-\tilde{v}_2(t)\\ \tilde{w}_1(t)-\tilde{w}_2(t)\end{pmatrix}\right\|_{L^{2}(\Omega)\times H^{1}(\Omega)}=\|[W_1(\tilde{u}_1)](t)-[W_1(\tilde{u}_2)](t)\|_{L^{2}(\Omega)\times H^{1}(\Omega)},	\]
	where $t\in [0, T]$,    we find
	\begin{align}
		&\left\|[G(\tilde{w}_1)](s)-[G(\tilde{w}_2)](s)+\beta_p\left[\tilde{u}_1(s)-\tilde{u}_2(s)\right]\right\|_{H^{1}(\Omega)}\notag\\
		\leq& L_G \|[W_1(\tilde{u}_1)](s)-[W_1(\tilde{u}_2)](s)\|_{L^{2}(\Omega)\times H^{1}(\Omega)}+\beta_p\left\|\tilde{u}_1(s)-\tilde{u}_2(s)\right\|_{H^2(\Omega)}\notag,\ \forall\ 0\leq s\leq t\leq T.
	\end{align}
	Thus
	\begin{align}
		&\|[W_1(\tilde{u}_1)](t)-[W_1(\tilde{u}_2)](t)\|_{L^2(\Omega)\times H^{1}(\Omega)}\notag\\
		\leq&T_0M_0\beta_p\sup_{t\in[0, T]}\left\|\tilde{u}_1(t)-\tilde{u}_2(t)\right\|_{H^2(\Omega)}+M_0L_G \int_0^t\|[W_1(\tilde{u}_1)](s)-[W_1(\tilde{u}_2)](s)\|_{L^{2}(\Omega)\times H^{1}(\Omega)}ds\notag.
	\end{align}
	According to Gronwall's inequality, we obtain
	\[\sup_{t\in[0, T]}\|\left[W_1(\tilde{u}_1)\right](t)-\left[W_1(\tilde{u}_2)\right](t)\|_{L^{2}(\Omega)\times H^{1}(\Omega)}\leq  T_0M_0\beta_pe^{M_0L_GT_0}\sup_{t\in[0, T]}\|\tilde{u}_1(t)-\tilde{u}_2(t)\|_{H^2(\Omega)}.\]
	As a result we conclude \eqref{2Lip-wrt-u} with $L_{W}=T_0M_0\beta_pe^{M_0L_GT_0}$. Because $L_G$, as a Lipschitz constant, depends on $\kappa$, $\|w_0\|_{H^{1}(\Omega)}$, $\Omega$ and the coefficient $\beta_F$,   $L_W$ depends on $T_0$, $M_0$, $\kappa$, $\|w_0\|_{H^{1}(\Omega)}$,  $\Omega$ and the coefficients $\beta_p$ and $\beta_F$,  that is
	\[L_W=L_W\left(T_0,\ M_0,\ \kappa,\ \|w_0\|_{H^{1}(\Omega)},\ \Omega,\  \beta_p,\ \beta_F\right).\]
	From the conclusion \eqref{2ndw-lower-bound} from Lemma \ref{2ndestimates}, we get that there exists a constant $C=C(\Omega)>0$, such that for all $r\in\left(0, \frac{\kappa}{2C}\right)$,
	\[ {\tilde{w}_1(t)+\theta_2}\geq \frac{\kappa}{2},\quad {\tilde{w}_2(t)+\theta_2}\geq \frac{\kappa}{2},\]
 for all $t\in[0, T]$. Then for the above constant $C=C(\Omega)$ and all $t\in [0, T]$, it follows that
	\[\left\|\frac{\tilde{v}_1(t)-\tilde{v}_2(t)}{\tilde{w}_1(t)+\theta_2}\right\|_{L^2(\Omega)}=\left[\int_\Omega\left|\frac{\tilde{v}_1(t)-\tilde{v}_2(t)}{\tilde{w}_1(t)+\theta_2}\right|^2dx\right]^{\frac{1}{2}}
	\leq \frac{2}{\kappa}\left\|\tilde{v}_1(t)-\tilde{v}_2(t)\right\|_{L^2(\Omega)},\]
	and thus 
	\begin{align}
		\left\|\frac{\tilde{v}_2(t)}{\tilde{w}_1(t)+\theta_2}-\frac{\tilde{v}_2(t)}{\tilde{w}_2(t)+\theta_2}\right\|^2_{L^2(\Omega)}&=\int_\Omega\left| \tilde{v}_2(t)\right|^2\left|\frac{1}{\tilde{w}_1(t)+\theta_2}-\frac{1}{\tilde{w}_2(t)+\theta_2}\right|^2dx\notag\\
		&=\int_\Omega\left| \tilde{v}_2(t)\right|^2\frac{\left|\tilde{w}_1(t)-\tilde{w}_2(t)\right|^2}{\left|\tilde{w}_1(t)+\theta_2\right|^2\left|\tilde{w}_2(t)+\theta_2\right|^2}dx\notag\\
		&\leq\frac{2^4}{\kappa^4}\int_\Omega\left| \tilde{v}_2(t)\right|^2\left|\tilde{w}_1(t)-\tilde{w}_2(t)\right|^2dx\notag\\
		&\leq\frac{2^4}{\kappa^4}\left\|\tilde{w}_1(t)-\tilde{w}_2(t)\right\|_{L^\infty(\Omega)}^2\int_\Omega\left| \tilde{v}_2(t)\right|^2dx\notag\\
		&\leq\frac{2^4C^2}{\kappa^4}\left\|\tilde{w}_1(t)-\tilde{w}_2(t)\right\|_{H^1(\Omega)}^2\int_\Omega\left| \tilde{v}_2(t)\right|^2dx\notag.
	\end{align}
We conclude that
	\begin{align}
		&\left\|[W_2(\tilde{u}_1)](t)-[W_2(\tilde{u}_2)](t)\right\|_{L^2(\Omega))}\notag\\
		\leq&\left\|\frac{\tilde{v}_1(t)-\tilde{v}_2(t)}{\tilde{w}_1(t)+\theta_2}\right\|_{L^2(\Omega)}+\left\|\frac{\tilde{v}_2(t)}{\tilde{w}_1(t)+\theta_2}-\frac{\tilde{v}_2(t)}{\tilde{w}_2(t)+\theta_2}\right\|_{L^2(\Omega)}\notag\\
		\leq&\frac{2}{\kappa}\left\|\tilde{v}_1(t)-\tilde{v}_2(t)\right\|_{L^2(\Omega)}+\frac{4C}{\kappa^2}\left\|\tilde{w}_1(t)-\tilde{w}_2(t)\right\|_{H^1(\Omega)}\left\| \tilde{v}_2(t)\right\|_{L^2(\Omega)}\notag\\
		\leq&\frac{2}{\kappa}\left\|\tilde{v}_1(t)-\tilde{v}_2(t)\right\|_{L^2(\Omega)}+\frac{4C}{\kappa^2}\left\|\tilde{w}_1(t)-\tilde{w}_2(t)\right\|_{H^1(\Omega)}\left(\left\| \tilde{v}_0\right\|_{L^2(\Omega)}+r\right)\notag\\
		\leq&\frac{2}{\kappa}\left\|\tilde{v}_1(t)-\tilde{v}_2(t)\right\|_{L^2(\Omega)}+\frac{4C}{\kappa^2}\left\|\tilde{w}_1(t)-\tilde{w}_2(t)\right\|_{H^1(\Omega)}\left(\left\| \tilde{v}_0\right\|_{L^2(\Omega)}+\frac{\kappa}{2C}\right)\label{2final estimate-12}.
	\end{align}
	Setting
	$L_{W_2}=L_{W}\cdot\max\left\{{2}{\kappa}^{-1},\ {4C}{\kappa^{-2}}\left(\left\| \tilde{v}_0\right\|_{L^2(\Omega)}+{\kappa}(2C)^{-1}\right)\right\}$,
	where $L_{W_2}$ depends on the above $L_{W}$ and  $\left\|\tilde{v}_0\right\|_{L^2(\Omega)}$, and using the estimates \eqref{2Lip-wrt-u} and \eqref{2final estimate-12},  \eqref{2Lip-frac-W} is obtained.
\end{proof}
\subsection*{Analysis of Fr\'{e}chet derivative}
Recall that $u=\tilde{u}+\theta_1$, $v=\tilde{v}$, $w=\tilde{w}+\theta_2$. For $T\in(0, T_0)$ a solution operator to the integral formulation \eqref{2ndmild-solu-form} is given by
\bse\label{2def-W}
\be\label{2def-W-01}
\begin{split}
	 W: C\left([0, T]; B_{H^2}(u_0,  r)\right)&\to C\left([0, T]; B_{L^2}(v_0,  r)\times  B_{H^1}(w_0,  r)\right),\\
	 u&\mapsto W(u)=\left(v, w\right)=\left(v(u), w(u)\right),
\end{split}
\ee
where
\be\label{2def-W-02}
 \left[W(u)\right](t)=\begin{pmatrix}0\\ \theta_2\end{pmatrix}+  T(t)\begin{pmatrix}v_0\\ w_0-\theta_2\end{pmatrix}+
\int_0^t\left\{T(t-s)\begin{pmatrix}[G(w-\theta_2)](s)+\beta_p(u(s)-\theta_1)\\ 0 \end{pmatrix}\right\}ds.
\ee
\ese
We recall that $W$ satisfies the Lipschitz estimate
\begin{align}
	 \sup_{t\in[0, T]}\left\|[W({u}_1)](t)-[W({u}_2)](t)\right\|_{L^{2}(\Omega)\times H^{1}(\Omega)}\leq L_{W}\sup_{t\in[0, T]}\|{u}_1(t)-{u}_2(t)\|_{H^2(\Omega)},\label{2Lip-W-I}
\end{align}
where $L_{W}$ is a Lipschitz constant depending on $T_0,\ M_0,\ \kappa,\ \|w_0\|_{H^{2}(\Omega)},\  \Omega,\  \beta_p,\ \beta_F$.

Let $h\in\mathds{R}$ be small such that  $u+hq\in C\left([0, T]; B_{H^2}(u_0, r)\right)$ for all $q\in C([0, T_0); H^2(\Omega)\cap H_0^1(\Omega))$,  then from the definition of the Fr\'{e}chet derivative ${W}{'}(u)$ of ${W}(u)$ on $u$,
\be\label{2Fre-W}
{W}{'}(u)q= \lim_{h\rightarrow 0}\frac{1}{h}\left[{W}({u}+hq)-{W}({u})\right],
\ee
${W}{'}(u)$ is an operator defined by
\bse\label{2Fre-W-1}
\be\label{2Fre-W-1-1}
{W}{'}(u): C([0,T]; H^2(\Omega)\cap H_0^1(\Omega))\to C([0,T];L^2(\Omega)\times H_0^1(\Omega)),
\ee
with
\be\label{2Fre-W-1-2}
 q\mapsto {W}{'}(u)q=\left(v'(u)q, w'(u)q\right).\ee
\ese
Inequality \eqref{2Lip-W-I} implies the Fr\'{e}chet derivative $W{'}(u)$ of $W(u)$ on $u$ uniformly exists and
\bse\label{2bound-Fre-W}
\be\label{2bound-Fre-W-01}
\sup_{t\in[0, T]}\left\|[W'(u)q](t)\right\|_{L^{2}(\Omega)\times H^{1}(\Omega)}\leq L_W\sup_{t\in[0, T]}\left\|q(t)\right\|_{H^2(\Omega)},
\ee
\be\label{2bound-Fre-W-02}
\sup_{t\in[0, T]}\left\|[W'(u)q](t)\right\|_{L^{2}(\Omega)\times H^{1}(\Omega)}\leq L_W,\quad  \forall\ q\in C\left([0, T]; B_{H^2}(0,1)\right).
\ee
\ese
\begin{cor}\label{2Lip-Frechet-W}
	Fix $T\in(0, T_0)$. For any $q\in C\left([0, T]; B_{H^2}(0,1)\right)$ and $u_1$, $u_2\in  C\left([0, T]; B_{H^2}(u_0, r)\right)$,  the Fr\'{e}chet derivative $W'(u)$ of $W(u)$ satisfies
	\be\label{2Lip-Frechet-W-I}
	\sup_{t\in[0, T]}\left\|[{W}'(u_1)q](t)-[{W}'(u_2)q](t)\right\|_{L^{2}(\Omega)\times H^{1}(\Omega)}\leq L_{F} \sup_{t\in[0, T]}\left\|{u}_1(t)-{u}_2(t)\right\|_ {H^2(\Omega)},
	\ee
	 where $L_F$ is a Lipschitz constant depending on $T_0$, $\Omega$, $\beta_p$, $\beta_F$, $M_0$, $\kappa$, $\|w_0\|_{H^{1}(\Omega)}$ and $\|v_0\|_{L^2(\Omega)}$.
\end{cor}

\begin{proof}
	According to Lemma \ref{2ndgenerator}, the linear operator $\mathcal{A}$ generates a $C_0$-semigroup
	\[\left\{T(t)=\begin{pmatrix}T_{11}(t) & T_{12}(t)\\ T_{21}(t) & T_{22}(t) \end{pmatrix}\in \mathcal{B}\left(L^2(\Omega)\times H_0^1(\Omega)\right):\quad t\in[0,\infty)\right\}.\]
	As $[G(w-\theta_2)](s)+\beta_p(u(s)-\theta_1)=-\beta_F[w(s)]^{-2}+\beta_p(u(s)-1)$, the integral formulation \eqref{2def-W-02}  implies the second component $w$ of vector function $W(u)=\left(v, w\right)=\left(v(u), w(u)\right)$ can be given by
	\begin{align}
		w(t)=&\left[w(u)\right](t)\notag\\
		=&\theta_2+T_{21}(t)v_0+T_{22}(t)\left(w_0-\theta_2\right)+ \int_0^tT_{21}(t-s)\left(\beta_p(u(s)-1)-\frac{\beta_F}{[w(u)]^2(s)}\right)ds.\notag
	\end{align}
	Using  this equality and the definitions \eqref{2Fre-W}-\eqref{2Fre-W-1} of the Fr\'{e}chet derivative $W'(u)=\left(v'(u), w'(u)\right)$, the Fr\'{e}chet derivative $w'(u)$ of $w(u)$ on $u$,  the second component of the Fr\'{e}chet derivative $W'(u)$, satisfies
	\bse\label{2Fre-w}
	\be\label{2Fre-w1}
	w'(u): C\left([0, T]; H^2(\Omega)\cap H_0^1(\Omega)\right)\to C\left([0, T]; H^2(\Omega)\cap H_0^1(\Omega)\right),
	\ee
	with
	\be\label{2Fre-w2}
	[w'(u)q](t)= \int_0^t T_{21}(t-s)\left\{\beta_p q(s)+2\beta_F\frac{[w'(u)q](s)}{[w(u)]^3(s)}\right\}ds.
	\ee
	\ese
	We shall show that there exists a positive Lipschitz constant $L_{F_2}$, which depends on $T_0$,  $M_0$, $\kappa$, $\|w_0\|_{H^1(\Omega)}$, $\Omega$, $\beta_p$, and $\beta_F$, such that
	\be\label{2Lip-Fre-w}
	\sup_{t\in[0, T]}\|[w'(u_1)q](t)-[w'(u_2)q](t)\|_{H^1(\Omega)}\leq L_{F_2}\sup_{t\in[0, T]}\|{u}_1(t)-{u}_2(t)\|_{H^2(\Omega)}.
	\ee
	We let $u_1$, $u_2\in C\left([0, T]; B_{H^2}(u_0,  r)\right)$ be such that the functions  $(v_1, w_1)=(v(u_1), w(u_1))$ and $(v_2, w_2)=(v(u_2), w(u_2))$ belong to $C\left([0, T]; B_{L^2}(v_0,  r)\times B_{H^1}(w_0,  r)\right)$.
    From the definitions \eqref{2Fre-W} and \eqref{2Fre-W-1} of Fr\'{e}chet derivative $W^{'}(u)$,
	\[W'(u_1)q=\left(v'(u_1)q, w'(u_1)q\right)\in C\left([0, T]; L^2(\Omega)\times H_0^1(\Omega)\right),\]
	\[W'(u_2)q=\left(v'(u_2)q, w'(u_2)q\right)\in C\left([0, T]; L^2(\Omega)\times H_0^1(\Omega)\right).\]
	Hence, by using 
	\[w_I(t)=[w'(u_1)q](t)\in H_0^1(\Omega),\quad w_J(t)=[w'(u_2)q](t)\in H_0^1(\Omega),\quad\forall\ t\in [0, T],\]
	one obtains
	\[\frac{[w'(u_1)q](t)}{[w(u_1)]^3(t)}-\frac{[w'(u_2)q](t)}{[w(u_2)]^3(t)}=\frac{w_I(t)}{[w_1(t)]^3}-\frac{w_J(t)}{[w_2(t)]^3}\in H_0^1(\Omega),\quad \forall\ t\in [0, T].\]
	Due to the inequality \eqref{2bound-Fre-W},
	\[\sup_{t\in[0, T]}\left\|w_I(t)\right\|_{H^1(\Omega)}=\sup_{t\in[0, T]}\left\|[w'(u_1)q](t)\right\|_{H^1(\Omega)}\leq L_W.\]
	From the Lipschitz continuity estimate \eqref{2Lip-W-I},  one can have 
	\begin{align}
		\left\|w_1(t)-w_2(t)\right\|_{H^1(\Omega)}=\left\|[w(u_1)](t)-[w(u_2)](t)\right\|_{H^1(\Omega)}
		\leq L_W\left\|u_1(t)-u_2(t)\right\|_{H^2(\Omega)}.\label{2Lip-w-II}
	\end{align}
	Using the algebraic property of $H^1(\Omega)$, i.e.~Lemma \ref{alg}, the estimates \eqref{2ndC-a}, \eqref{2ndC-d} from Lemma \ref{2ndestimates}, and  \eqref{2Lip-w-II}, it follows that
	\begin{align}
		\left\|\frac{w_I(t)}{[w_1(t)]^3}-\frac{w_J(t)}{[w_2(t)]^3}\right\|_{H^1(\Omega)}\leq&\left\|w_I(t)\right\|_{H^1(\Omega)}\cdot \left\|\frac{1}{[w_1(t)]^3}-\frac{1}{[w_2(t)]^3}\right\|_{H^1(\Omega)}\notag\\
		+&\left\|w_I(t)-w_J(t)\right\|_{H^1(\Omega)}\cdot\left\|\frac{1}{[w_2(t)]^3}\right\|_{H^1(\Omega)}\notag\\
		\leq&L_W C_3\left\|w_1(t)-w_2(t)\right\|_{H^1(\Omega)}+C_1^3\left\|w_I(t)-w_J(t)\right\|_{H^1(\Omega)}\notag\\
		\leq&L^2_WC_3\left\|u_1(t)-u_2(t)\right\|_{H^2(\Omega)}+C_1^3\left\|w_I(t)-w_J(t)\right\|_{H^1(\Omega)}\label{2w-cubic}.
	\end{align}
	Equation \eqref{2w-cubic}  and the formulation \eqref{2Fre-w} of the Fr\'{e}chet derivative $w'(u)q$ of $w(u)$ on $u$ imply
	\begin{align}
		 \|w_I(t)-w_J(t)\|_{H^1(\Omega)}=&\left\|2\beta_F\int_0^tT_{21}(t-s)\left(\frac{w_I(s)}{[w_1(s)]^3}-\frac{w_J(s)}{[w_2(s)]^3}\right)ds\right\|_{H^1(\Omega)}\notag\\
		\leq&2\beta_F \int_0^t \sup_{0\leq s\leq t}\left\|T_{21}(t-s)\right\|_{\mathcal{B}\left(L^2(\Omega), H_0^1(\Omega)\right)}\left\|\frac{w_I(s)}{[w_1(s)]^3}-\frac{w_J(s)}{[w_2(s)]^3}\right\|_{L^2(\Omega)}ds\notag\\
		\leq&2\beta_F  M_0\int_0^t\left\|\frac{w_I(s)}{[w_1(s)]^3}-\frac{w_J(s)}{[w_2(s)]^3}\right\|_{H^1(\Omega)}ds\notag\\
		\leq&2\beta_F M_0\left(L^2_{W}C_3 T_0\|{u}_1(t)-{u}_2(t)\|_{H^2(\Omega)}
		+C_1^3\int_0^t\left\|w_I(s)-w_J(s)\right\|_{H^1(\Omega)}ds\right)\notag.
	\end{align}
	 Gronwall's inequality then gives
	\[ \|w_I(t)-w_J(t)\|_{H^1(\Omega)}\leq 2\beta_F M_0L^2_{W}C_3T_0e^{2\beta_FM_0C_1^3 T_0}\|{u}_1(t)-{u}_2(t)\|_{H^2(\Omega)}.\]
	Thus we obtain the estimate \eqref{2Lip-Fre-w}  by setting
	\[L_{F_2}=2\beta_F M_0L^2_{W}C_3T_0e^{2\beta_FM_0C_1^3 T_0}.\]
	Similarly,  there exists a positive Lipschitz constant $L_{F_1}=L_{F_1}\left(L_{F_2},\ \|{v}_0\|_{L^2(\Omega)}\right)$ such that the Frech\'{e}t derivative $v{'}(u)$ of the first component $v(u)$ of $W(u)$ on $u$, defined via
	\be\label{2Fre-v}
	[v{'}(u)q](t)= \int_0^t T_{11}(t-s)\left\{\beta_p q(s)+2\beta_F\frac{[w{'}(u)q](s)}{[w(u)]^3(s)}\right\}ds,
	\ee
	is a map $t\to[v{'}(u)](t)$ from $[0, T]$ to  $C\left([0, T]; \mathcal{B}\left(H_0^1(\Omega), L^2(\Omega)\right)\right)$ and satisfies
	\be\label{2Lip-Fre-v}
	 \sup_{t\in[0, T]}\|[v'(u_1)q](t)-[v'(u_2)q](t)\|_{L^2(\Omega)}\leq L_{F_1}\sup_{t\in[0, T]}\|{u}_1(t)-{u}_2(t)\|_{H^2(\Omega)}
	\ee
	for all $ q\in C\left([0, T]; H_0^1(\Omega)\right)$. Letting $L_F=\max\left\{L_{F_1},L_{F_2}\right\}$, \eqref{2Lip-Fre-w} and \eqref{2Lip-Fre-v} imply  the assertion \eqref{2Lip-Frechet-W-I}.
\end{proof}

\begin{cor}\label{2Holder-Frechet-W-I-Cor}
	Fix $T\in (0, T_0)$, $r\in\left(0, \frac{\kappa}{2C}\right)$ with  $\kappa= \inf_{x\in{\Omega}}w_0(x)$ and positive constant $C=C(\Omega)$. 
	If $u\in C^\alpha([0, T]; B_{H^2}(u_0, r))$,
	then there exists a positive Lipschitz constant $L_M$ depending on
	$\alpha$, $M_0$, $T_0$, $\Omega$, $\|v_0\|_{L^2(\Omega)}$, $\|w_0\|_{H^1(\Omega)}$, $\kappa$, $\beta_F$, $\beta_p$,
	such that
	\be\label{2Holder-Frechet-W-I}
	\begin{split}
		\sup_{0\leq t<t+h\leq T}\left\|[{W}'(u)q](t+h)-[{W}'(u)q](t)\right\|_{L^{2}(\Omega)\times H^{1}(\Omega)}\leq&h^\alpha L_{M}\sup_{t\in [0, T]}\left\|q(t)\right\|_{H^2(\Omega)}\\
		+&h^\alpha TL_{M}\left\|q\right\|_{C^\alpha\left([0, T]; H^2(\Omega)\right)}
	\end{split}\ee
	holds for all $q\in C^\alpha([0, T]; B_{H^2}(\tilde{u}_0, r))$, where  $\tilde{u}_0=u_0-\theta_1\in H^{2+\sigma}(\Omega)\cap H_0^1(\Omega)$ and $u_0\in H^{2+\sigma}(\Omega)$ with $u|_{\partial\Omega}=\theta_1$ and $\sigma\in(0, \frac{1}{2})$.
\end{cor}
\begin{proof}
	For $T\in(0, T_0)$ 
$u\in C^\alpha([0, T]; B_{H^2}(u_0, r))$,
	 Theorem \ref{2nd-solu-thm} and Corollary \ref{2Holdercontinuity} imply that $w(u)\in C^\alpha([0, T]; B_{H^2}(w_0, r))$ is a unique mild solution of the semilinear hyperbolic equation \eqref{2ndcp1-1-2} and  $w(u)$ satisfies
	\[\left[w(u)\right](t)=\theta_2+T_{21}(t)v_0+T_{22}(t)\left(w_0-\theta_2\right)+ \int_0^tT_{21}(t-s)\left(\beta_p(u-1)-\frac{\beta_F}{[w(u)]^2(s)}\right)ds.\]
	The definitions \eqref{2Fre-W} and \eqref{2Fre-W-1} of the Fr\'{e}chet derivative $W'(u)$ imply the Fr\'{e}chet derivative $w'(u)$ of $w(u)$ on $u$ satisfies
	\be\label{2TFre-w2}
	[w'(u)q](t)= \int_0^t T_{21}(t-s)\left\{\beta_p q(s)+2\beta_F\frac{[w'(u)q](s)}{[w(u)]^3(s)}\right\}ds.
	\ee
	We aim to show that there is a constant $L_{M_1}>0$, such that
	\be\label{2H-I}
	\left\|[w'(u)q](t+h)-[w'(u)q](t)\right\|_{H^1(\Omega)}\leq L_{M_1}h^\alpha\left(\sup_{t\in[0, T]}\left\|q(t)\right\|_{H^2(\Omega)}+T\left\|q\right\|_{C^\alpha\left([0, T]; H^2(\Omega)\right)}\right)
	\ee
	holds for all $ 0\leq t<t+h\leq T,\ h\in(0, T]$. Notice that
	\be\label{2H00}\begin{split}
		[w'(u)q](t+h)-[w'(u)q](t)=& \int_0^h T_{21}(t+h-s)\left\{\beta_p q(s)+2\beta_F\frac{[w'(u)q](s)}{[w(u)]^3(s)}\right\}ds\\
		+&\int_0^t T_{21}(t-s)\beta_p[q(s+h)-q(s)]ds\\
		+&\int_0^t T_{21}(t-s)2\beta_F\left\{\frac{[w'(u)q](s+h)}{[w(u)]^3(s+h)}-\frac{[w'(u)q](s)}{[w(u)]^3(s)}\right\}ds.
	\end{split}\ee
	Having $q(t)\in H^2(\Omega)\cap H_0^1(\Omega)$, the definition \eqref{2def-W} of $W(u)$ and the definition \eqref{2Fre-W} of $W'(u)$ imply
	\[\beta_p q(t)+2\beta_F\frac{[w'(u)q](t)}{[w(u)]^3(t)}\in H_0^1(\Omega).\]
	Combining \eqref{2ndC-a} of Lemma \ref{2ndestimates} with \eqref{2bound-Fre-W} gives
	\[\sup_{t\in [0, T]}\left\|\frac{[w'(u)q](t)}{[w(u)]^3(t)}\right\|_{H^1(\Omega)}\hspace*{-0.1cm}\leq\hspace*{-0.1cm}\sup_{t\in[0, T]}\left\{\left\|\frac{1}{[w(u)]^3(t)}\right\|_{H^1(\Omega)}\hspace*{-0.1cm}\left\|[w'(u)q](t)\right\|_{H^1(\Omega)}\hspace*{-0.1cm}\right\}\leq C_1^3L_W\hspace*{-0.1cm}\sup_{t\in [0, T]}\hspace*{-0.1cm}\left\|q(t)\right\|_{H^2(\Omega)}.\]
	Therefore
	\begin{align}\label{2H01}
		&\left\| \int_0^h T_{21}(t+h-s)\left\{\beta_p q(s)+2\beta_F\frac{[w'(u)q](s)}{[w(u)]^3(s)}\right\}ds\right\|_{H^1(\Omega)}\notag\\
		\leq&hM_0\left\{\beta_p\sup_{t\in [0, T]}\|q(t)\|_{H^2(\Omega)}+\sup_{t\in [0, T]}\left\|\frac{[w'(u)q](t)}{[w(u)]^3(t)}\right\|_{H^1(\Omega)}\right\}\notag\\
		\leq&hM_0\left(\beta_p+C_1^3L_W\right)\sup_{t\in [0, T]}\left\|q(t)\right\|_{H^2(\Omega)}.
	\end{align}
	Now $q\in C^\alpha([0, T]; B_{H^2}(\tilde{u}_0, r))$ implies
	\begin{align}
		\left\| \int_0^t T_{21}(t-s)\beta_p[q(s+h)-q(s)]ds\right\|_{H^1(\Omega)}\leq& TM_0\beta_p\sup_{0\leq t<t+h\leq T}\left\|q(t+h)-q(t)\right\|_{H^1(\Omega)}\notag\\
		\leq&h^\alpha TM_0\beta_p\|q\|_{C^\alpha\left([0, T]; H^2(\Omega)\right)}\label{2H02}.
	\end{align}
	The triangle inequality and the algebraic property of $H^1(\Omega)$ (i.e.~Lemma \ref{alg}) imply
	\begin{align}
		&\left\|\frac{[w'(u)q](t+h)}{[w(u)]^3(t+h)}-\frac{[w'(u)q](t)}{[w(u)]^3(t)}\right\|_{H^1(\Omega)}\notag\\
		\leq&\left\|\frac{[w'(u)q](t+h)}{[w(u)]^3(t+h)}-\frac{[w'(u)q](t+h)}{[w(u)]^3(t)}\right\|_{H^1(\Omega)}
		+\left\|\frac{[w'(u)q](t+h)}{[w(u)]^3(t)}-\frac{[w'(u)q](t)}{[w(u)]^3(t)}\right\|_{H^1(\Omega)}\notag\\
		\leq&\left\|[w'(u)q](t+h)\right\|_{H^1(\Omega)}\left\|\frac{1}{[w(u)]^3(t+h)}-\frac{1}{[w(u)]^3(t)}\right\|_{H^1(\Omega)}\notag\\
		+&\left\|[w'(u)q](t+h)-[w'(u)q](t)\right\|_{H^1(\Omega)}\left\|\frac{1}{[w(u)]^3(t)}\right\|_{H^1(\Omega)}.\label{2H02-1}
	\end{align}
	Since $w(u)\in C^\alpha\left([0, T]; B_{H^1}(w_0, r)\right)$ is a mild solution of the semilinear hyperbolic equation \eqref{2ndcp1-1-2}, then from estimate \eqref{2ndC-d} of Lemma \ref{2ndestimates} and the estimate \eqref{2Holdercontinuityformular} of Corollary \ref{2Holdercontinuity}, it follows that
	\begin{align}
		\left\|\frac{1}{[w(u)]^3(t+h)}-\frac{1}{[w(u)]^3(t)}\right\|_{H^1(\Omega)}\leq&C_3 \left\|[w(u)](t+h)-[w(u)](t)\right\|_{H^1(\Omega)}
		\leq C_3L_Uh^\alpha.\label{2H02-2}
	\end{align}
	Hence, estimates  \eqref{2ndC-a}, \eqref{2bound-Fre-W}, \eqref{2H02-1} and \eqref{2H02-2} imply
	\begin{align}
		\left\|\frac{[w'(u)q](t+h)}{[w(u)]^3(t+h)}-\frac{[w'(u)q](t)}{[w(u)]^3(t)}\right\|_{H^1(\Omega)}
		\leq& L_W\sup_{t\in [0, T]}\left\|q(t)\right\|_{H^1(\Omega)}C_3L_Uh^\alpha\notag\\
		+&\left\|[w'(u)q](t+h)-[w'(u)q](t)\right\|_{H^1(\Omega)}C_1^3\notag.
	\end{align}
	Therefore, 
	\begin{align}
		&\left\| \int_0^t T_{21}(t-s)2\beta_F\left\{\frac{[w'(u)q](s+h)}{[w(u)]^3(s+h)}-\frac{[w'(u)q](s)}{[w(u_0)]^3(s)}\right\}ds\right\|_{H^1(\Omega)}\notag\\
		\leq&2\beta_FM_0\hspace*{-0.1cm}\left(\hspace*{-0.1cm}T_0L_W\hspace*{-0.1cm}\sup_{t\in [0, T]}\left\|q(t)\right\|_{H^2(\Omega)}C_3L_Uh^\alpha
		+C_1^3\hspace*{-0.1cm}\int_0^t \left\|[w'(u)q](s+h)-[w'(u)q](s)\right\|_{H^1(\Omega)}ds\hspace*{-0.1cm}\right).\label{2H03}
	\end{align}
	Consequently, \eqref{2H00}, \eqref{2H01}, \eqref{2H02} and \eqref{2H03} imply
	\begin{align}
		\left\|[w'(u)q](t+h)-[w'(u)q](t)\right\|_{H^1(\Omega)}
		\leq& h^\alpha T_0^{1-\alpha}M_0\left[\beta_p+C_1^3L_W\right]\sup_{t\in [0, T]}\left\|q(t)\right\|_{H^2(\Omega)}\notag\\
		+&h^\alpha TM_0\beta_p\|q\|_{C^\alpha\left([0, T]; H^2(\Omega)\right)}\notag\\
		+&h^\alpha2\beta_FM_0T_0L_WC_3L_U\sup_{t\in [0, T]}\left\|q(t)\right\|_{H^2(\Omega)}\notag\\
		+&2\beta_FM_0C_1^3 \int_0^t \left\|[w'(u)q](s+h)-[w'(u)q](s)\right\|_{H^1(\Omega)}ds\notag.
	\end{align}
	Set
	$R_1=M_0T_0^{1-\alpha}\left[\beta_p+C_1^3L_W\right]$, $R_2=2\beta_FM_0T_0L_WC_3L_U$, $R_3=M_0\beta_p$ and  $R_4=2\beta_FM_0C_1^3$.
	Gronwall's inequality implies $\forall\ 0\leq t<t+h\leq T$,
	\begin{align}
		\left\|[w'(u)q](t+h)-[w'(u)q](t)\right\|_{H^1(\Omega)}\leq &h^\alpha e^{R_4T_0}(R_1+R_2)\sup_{t\in [0, T]}\left\|q(t)\right\|_{H^2(\Omega)}\notag\\
		+&h^\alpha T e^{R_4T_0} R_3\|q\|_{C^\alpha\left([0, T]; H^2(\Omega)\right)}.\notag
	\end{align}
	Eqn.~\eqref{2H-I} holds by setting $L_{M_1}= (R_1+R_2+R_3)e^{R_4T_0}$, where $L_{M_1}$ depends on $\alpha$, $M_0$, $T_0$, $\Omega$, $\|w_0\|_{H^1(\Omega)}$, $\kappa$, $\beta_F$, $\beta_p$.
	
	Similarly, there exists a Lipschitz constant $L_{M_2}>0$ depending on $L_{M_1}$ and $\|{v}_0\|_{L^2(\Omega)}$, such that the Frech\'{e}t derivative $v{'}(u)$ of the first component $v(u)$ of $W(u)$, defined via
	\[[v{'}(u)q](t)= \int_0^t T_{11}(t-s)\left\{\beta_p q(s)+2\beta_F\frac{[w{'}(u)q](s)}{[w(u)]^3(s)}\right\}ds,\]
	satisfies
	\be\label{2H05}
	\|[v'(u)q](t+h)-[v'(u)q](t)\|_{L^2(\Omega)}\leq L_{M_2}h^\alpha\left(\sup_{t\in[0, T]}\left\|q(t)\right\|_{H^2(\Omega)}+T\left\|q\right\|_{C^\alpha\left([0, T]; H^2(\Omega)\right)}\right).
	\ee
	Setting $L_M=\max\left\{L_{M_1},L_{M_2}\right\}$,  the assertion \eqref{2Holder-Frechet-W-I} follows from \eqref{2H-I} and \eqref{2H05}.
\end{proof}


\section{Well-posedness of the Coupled System}\label{2ndsection4}


	In this section we prove the main result of this article, Theorem \ref{2ndcoupled system}. First,  recall the formulation \eqref{2ndcp3} of the coupled system \eqref{2ndcp1} in the open bounded subset $\Omega \subset \mathds{R}$:
	\bse
	\be
	\frac{\partial u}{\partial t}=\frac{1}{w}\frac{\partial}{\partial x}\left(w^3u\frac{\partial u}{\partial x}\right)-\frac{v}{w}u,\quad x\in\Omega,\ t\geq 0;
	\ee
	\be
	\frac{\partial v}{\partial t}=\frac{\partial^2w}{\partial x^2}-\frac{\beta_F}{w^2}+\beta_p(u-1),\quad x\in\Omega,\ t\geq 0;
	\ee
	\be
	\frac{\partial w}{\partial t}=v,\quad x\in\Omega,\ t\geq 0.
	\ee
	\ese
	Here, the initial values are denoted by $u(x,0)=u_0(x) \in H^{2+\sigma}(\Omega)$ ($\sigma\in(0, \frac{1}{2})$), $v(x,0)=v_0(x)\in H_0^1(\Omega)$, $w(x,0)=w_0(x) \in H^2(\Omega)$. We assume that $u_0\geq\epsilon_1$ and $w_0 \geq \kappa$, for given constants $\epsilon_1, \kappa>0$. The boundary values are given by $u|_{\partial\Omega}=\theta_1>0$,  $w|_{\partial\Omega}=\theta_2>0$. 
 Note that $\tilde{u}_0=u_0-\theta_1\in H^{2+\sigma}(\Omega)\cap H_0^1(\Omega)$.  

\subsection*{Abstract Formulation}

We study the existence and uniqueness of a strict solution for the initial-boundary value problem for this parabolic-hyperbolic coupled system by writing it in the following, equivalent abstract  quasilinear parabolic equation with coefficients involving $v(u)$ and $w(u)$:
\bse\label{2ndQPE}
\be\label{2ndQPE-1}
\frac{\partial u}{\partial t}=\frac{1}{w(u)}\frac{\partial}{\partial x}\left([w(u)]^3u\frac{\partial u}{\partial x}\right)-\frac{v(u)}{w(u)}u,\quad (x,t)\in\Omega\times(0, T),
\ee
\be\label{2ndQPE-2}
u(x,0)=u_0(x),\quad x\in\Omega,\quad u(x,t)=\theta_1,\quad (x,t)\in\partial\Omega\times[0, T].
\ee
\ese
Here $u=u(x,t)$ is an unknown function, $v(u)=[v(u)](x,t)$ and $w(u)=[w(u)](x,t)$ are  given as  functions depending on $u$ by the integral formulation 
\[\begin{pmatrix}v(t)\\ w(t)\end{pmatrix}=\begin{pmatrix}0\\ \theta_2\end{pmatrix}+ T(t)\begin{pmatrix}v_0\\ w_0-\theta_2\end{pmatrix}+\int_0^t\left\{T(t-s)\begin{pmatrix}-\beta_F[w(s)]^{-2}+\beta_p\left(u(s)-1\right)\\ 0 \end{pmatrix}\right\}ds,\]
where $\left\{T(t):\ t\geq 0\right\}$ is the strongly continuous semigroup from Lemma \ref{2ndgenerator}.

Let $\tilde{u}=u-\theta_1$ be restricted in $C\left([0, T]; H^2(\Omega)\cap H_0^1(\Omega)\right)$ temporarily, then from the existence of the mild solution of the semilinear evolution equation \eqref{2nd-SWE-1}, i.e.~Theorem \ref{2nd-solu-thm}, the functions $v$ and $w$  satisfy the definition \eqref{2def-W} of solution operator $W$ of $u$ and the mapping property:
\[v:\  C\left([0, T]; H^2(\Omega)\cap H_0^1(\Omega)\right)\to C\left([0, T]; L^2(\Omega)\right),\qquad  \tilde{u}\mapsto v(\tilde{u}+\theta_1);\]
\[w:\ C\left([0, T]; H^2(\Omega)\cap H_0^1(\Omega)\right)\to C\left([0, T]; H^1(\Omega)\right),\qquad \tilde{u}\mapsto w(\tilde{u}+\theta_1).\] 
Hence 
\bse\label{2ndF-def}
\be\label{2ndF-def-1}
\tilde{u}\mapsto F(\tilde{u}):\quad C\left([0, T]; H^2(\Omega)\cap H_0^1(\Omega)\right)\to C\left([0, T]; L^2(\Omega)\right),
\ee
where
\be\label{2ndF-def-2}
F(\tilde{u})=\frac{1}{w(\tilde{u}+\theta_1)}\frac{\partial}{\partial x}\left(\left[w(\tilde{u}+\theta_1)\right]^3(\tilde{u}+\theta_1)\frac{\partial\tilde{u}}{\partial x}\right)-\frac{v(\tilde{u}+\theta_1)}{w(\tilde{u}+\theta_1)}(\tilde{u}+\theta_1).
\ee
\ese
The linearisation of $F(\tilde{u})$ is defined by
\be\label{2ndlinearisation}
q\mapsto F'(\tilde{u}_0)q:\quad C\left([0, T]; H^2(\Omega)\cap H_0^1(\Omega)\right)\to C\left([0, T]; L^2(\Omega)\right).
\ee
Here,  $F'(\tilde{u}_0)q$ is the Fr\'{e}chet derivative of $F(\tilde{u})$ on $\tilde{u}$ at $\tilde{u}_0$, $F'(\tilde{u}_0)q$ at $t$ is given as:
\be\begin{split}
	\left[F'(\tilde{u}_0)q\right](t)=&\frac{1}{[w(u_0)](t)}\frac{\partial}{\partial x}\left\{[w(u_0)]^3(t)u_0\frac{\partial q(t)}{\partial x}+[w(u_0)]^3(t)q(t)\frac{\partial u_0}{\partial x}\right\}\\
	+&\frac{1}{[w(u_0)](t)}\frac{\partial}{\partial x}\left\{3[w(u_0)]^2(t)[w'(u_0)q](t)u_0\frac{\partial u_0}{\partial x}\right\}\\
	-&\frac{[w'(u_0)q](t)}{[w(u_0)]^2(t)}\frac{\partial}{\partial x}\left([w(u_0)]^3(t)u_0\frac{\partial u_0}{\partial x}\right)-\frac{[v(u_0)](t)}{[w(u_0)](t)}q(t)\\
	-&\frac{[w(u_0)](t)[v'(u_0)q](t)-[v(u_0)](t)[w'(u_0)q](t)}{[w(u_0)]^2(t)}u_0, \label{2ndFre-F}
\end{split}\ee
where the functions $v(u_0)$ and $w(u_0)$ satisfy the definition \eqref{2def-W} of solution operator $W$ with $u=u_0$. Equivalently, $(\tilde{v}, \tilde{w})=(v(u_0), w(u_0)- \theta_2)$ is a unique mild solution of the semilinear evolution equation \eqref{2nd-SWE-1} with $\tilde{u}=u_0-\theta_1$, and $([v(u_0)](0),[w(u_0)](0))=(v_0, w_0)$. 


\subsection*{Well-posedness of Linearised Formulation}

Essentially, we aim to define an operator $\mathcal{P}^*$ by the expression
\be\begin{split}
	\mathcal{P}^*q(t)=&\frac{1}{w_0}\frac{\partial}{\partial x}\left\{w_0^3u_0\frac{\partial q(t)}{\partial x}+w_0^3q(t)\frac{\partial u_0}{\partial x}\right\}+\frac{1}{w_0}\frac{\partial}{\partial x}\left\{3w_0^2[w'(u_0)q](0)u_0\frac{\partial u_0}{\partial x} \right\}\\
	-&\frac{[w'(u_0)q](0)}{w_0^2}\frac{\partial}{\partial x}\left(w_0^3u_0\frac{\partial u_0}{\partial x}\right)-\frac{v_0}{w_0}q(t)-\frac{w_0[v'(u_0)q](0)-v_0[w'(u_0)q](0)}{w_0^2}u_0. \label{2ndlinearopt-0}
\end{split}\ee
Note that the definition of the Fr\'{e}chet derivative $w'(u)$ of $w(u)$ on $u$ at $u=u_0$ and $t=0$ is by
\[[w'(u_0)q](0)= \lim_{h\rightarrow 0}\frac{1}{h}\left\{[w(u_0+hq)](0)-[w(u_0)](0)\right\}.\]
Here $h\in\mathds{R}$ is small such that $u_0+hq\in C\left([0, T]; B_{H^2}(u_0,  r)\right)$ for all $q\in C\left([0, T]; H^2(\Omega)\cap H_0^1(\Omega)\right)$. Since $(\tilde{v}_h, \tilde{w}_h)=(v(u_0+hq), w(u_0+hq)-\theta_2)$ is a unique mild solution of the semilinear evolution equation \eqref{2nd-SWE-1} with $\tilde{u}=u_0+hq-\theta_1$, then it follows that $([v(u_0+hq)](0),  [w(u_0+hq)](0))=(v_0, w_0)$. As $([v(u_0)](0),[w(u_0)](0))=(v_0, w_0)$, we have $[w'(u_0)q](0)=0$, analogously $[v'(u_0)q](0)=0$, and then  \eqref{2ndlinearopt-0} is simplified to
\[\mathcal{P}^*q(t)=\frac{1}{w_0}\frac{\partial}{\partial x}\left\{w_0^3u_0\frac{\partial q(t)}{\partial x}+w_0^3q(t)\frac{\partial u_0}{\partial x}\right\}-\frac{v_0}{w_0}q(t).\]
We therefore define $\mathcal{P}^*$ as the linear operator 
\be\label{2ndlinearopt}
\mathcal{P}^*: D\left(\mathcal{P}^*\right)\subseteq H^2(\Omega)\cap H_0^1(\Omega)\to L^2(\Omega),\ \mathcal{P}^*\psi=\frac{1}{w_0}\frac{\partial}{\partial x}\left\{w_0^3u_0\frac{\partial\psi}{\partial x}+\left(w_0^3\frac{\partial u_0}{\partial x}\right)\psi\right\}-\frac{v_0}{w_0}\psi,
\ee  
defined for smooth functions satsifying homogeneous Dirichlet  boundary conditions. It is a Dirichlet realisation of the differential expression in \eqref{2ndQPE-1}.
Using $\mathcal{P}^*$, we rewrite \eqref{2ndQPE} as an equation for the unknown function $\tilde{u}$:
\be\label{2ndNPE}
\tilde{u}'(t)=\mathcal{P}^*\tilde{u}(t)+[F(\tilde{u})](t)-\mathcal{P}^*\tilde{u}(t), \quad t\in[0, T],\quad \tilde{u}(0)=\tilde{u}_{0}.
\ee
The next lemma gives an elliptic  estimate for the quadratic form associated to $\mathcal{P}^*$, with a standard proof.
\begin{lem}\label{2ndlinearisation-elliptic}
	There exist  positive constants $K$ and $K_{o}$  depending on $u_0$, $w_0$, such that for all $t\in [0, T]$,  the following elliptic estimate is satisfied:
	\be\label{2ndelliptic-est}
	\left|\int_{\Omega}\frac{q(t)}{w_0}\left[w_0^3u_0\frac{\partial q(t)}{\partial x}\right]dx\right|\geq K\int_{\Omega}\left|\frac{\partial q(t)}{\partial x}\right|^2dx-  K_{o}\int_{\Omega}|q(t)|^2dx,\qquad\forall \ q(t)\in D\left(\mathcal{P}^*\right)
	\ee
\end{lem}
\begin{proof}
	For $t\in [0, T]$ and $q(t)\in D\left(\mathcal{P}^*\right)$, the highest order derivative term  of $\mathcal{P}^*q(t)$ is 
	\[\mathcal{P}^*_hq(t)=\frac{1}{w_0}\frac{\partial}{\partial x}\left\{w_0^3u_0\frac{\partial q(t)}{\partial x}\right\}\]
    Integrating by parts, we obtain
	\begin{align}
		 \int_{\Omega}\frac{q(t)}{w_0}\frac{\partial}{\partial x}\left\{w_0^3u_0\frac{\partial q(t)}{\partial x}\right\}dx=&\left\{w_0^2u_0q(t)\frac{\partial q(t)}{\partial x}\right\}_{\partial\Omega}-\int_{\Omega}\left\{\frac{\partial}{\partial x}\left(\frac{q(t)}{w_0}\right)\right\}\left\{w_0^3 u_0\frac{\partial q(t)}{\partial x}\right\}dx.\label{2nddiv1}
	\end{align}
	As $q(t)\in D\left(\mathcal{P}^*\right)\subseteq H^2(\Omega)\cap H_0^1(\Omega)$,  $[q(t)](x)=q(x,t)=0$ for all $(x,t)\in\partial\Omega\times[0, T]$ and hence $w_0^2u_0q(t)\frac{\partial q(t)}{\partial x}=0$ on $\partial \Omega$.
	Using that $u_0(x)\geq\epsilon_1>0$ for a given constant $\epsilon_1$ and that $\kappa= \inf_{x\in{\Omega}}w_0(x)$ from \eqref{2nddiv1}, we find
	\begin{align}
		\left|\int_{\Omega}\frac{q(t)}{w_0}\frac{\partial}{\partial x}\left\{w_0^3u_0\frac{\partial q(t)}{\partial x}\right\}dx\right|
		=&\left|\int_{\Omega}\left\{\frac{\partial}{\partial x}\left(\frac{q(t)}{w_0}\right)\right\}\left\{w_0^3 u_0\frac{\partial q(t)}{\partial x}\right\}dx\right|\notag\\
		\geq&\left|\int_{\Omega}u_0 w_0^2\left|\frac{\partial q(t)}{\partial x}\right|^2dx\right|-\left|\int_{\Omega}  u_0w_0\frac{\partial w_0}{\partial x}\left\{q(t)\frac{\partial q(t)}{\partial x}\right\}dx\right|\notag\\
		\geq& {\epsilon_1\kappa^2}\int_{\Omega}\left|\frac{\partial q(t)}{\partial x}\right|^2dx-\left|\int_{\Omega}  u_0w_0\frac{\partial w_0}{\partial x}\left\{q(t)\frac{\partial q(t)}{\partial x}\right\}dx\right|.\notag
	\end{align}
	As  $w_0\in H^2(\Omega)$ and $u_0\in H^{2+\sigma}(\Omega)$ with $\sigma\in(0, \frac{1}{2})$,  writing $C=C(\Omega)$ a positive constant, we have 
	\begin{align}
		\left|\int_{\Omega}u_0w_0\frac{\partial w_0}{\partial x}\left\{q(t)\frac{\partial q(t)}{\partial x}\right\}dx\right|\leq&\left\|u_0w_0\frac{\partial w_0}{\partial x}\right\|_{L^\infty(\Omega)}\left|\int_{\Omega}q(t)\frac{\partial q(t)}{\partial x}dx\right|\notag\\
		\leq&C\left\|u_0w_0\frac{\partial w_0}{\partial x}\right\|_{H^1(\Omega)}\left|\int_{\Omega}q(t)\frac{\partial q(t)}{\partial x}dx\right|\notag\\
		\leq&C\left\|u_0\right\|_{H^1(\Omega)}\left\| w_0\right\|_{H^1(\Omega)}\left\|\frac{\partial w_0}{\partial x}\right\|_{H^1(\Omega)}\left|\int_{\Omega}q(t)\frac{\partial q(t)}{\partial x}dx\right|\notag\\
		\leq&C\left\|u_0\right\|_{H^1(\Omega)}\left\|w_0\right\|^2_{H^2(\Omega)}\left|\int_{\Omega}q(t)\frac{\partial q(t)}{\partial x}dx\right|\notag
	\end{align}
	With $K_2=C\left\|u_0\right\|_{H^1(\Omega)}\left\|w_0\right\|^2_{H^2(\Omega)}$ and Young's inequality, it follows that
	\begin{align}
		\left|\int_{\Omega}\frac{q(t)}{w_0}\frac{\partial}{\partial x}\left\{w_0^3u_0\frac{\partial q(t)}{\partial x}\right\}dx\right|\geq&(\epsilon_1\kappa^2-\varepsilon^2K_2)\int_{\Omega}\left|\frac{\partial q(t)}{\partial x}\right|^2dx-  \frac{K_2}{4\varepsilon^2}\int_{\Omega}|q(t)|^2dx.\label{2ndF}
	\end{align}
	The assertion \eqref{2ndelliptic-est} follows for $\varepsilon$ sufficiently small.
\end{proof}
\begin{cor}\label{2ndgeneratoroflinearisation}
	$\mathcal{P}^*$, defined by \eqref{2ndlinearopt}, is a sectorial operator which generates an analytic semigroup $\left\{e^{t\mathcal{P}^*}:\ t\geq0 \right\}$ on $H_0^1(\Omega)$.
\end{cor}
\begin{proof}
	Using the Corollary 12.19 and Corollary 12.21 from \cite{GG}, we conclude that the operator $\mathcal{P}^*$ defined in Lemma \ref{2ndlinearisation-elliptic}  satisfies the elliptic estimate \eqref{2ndelliptic-est} and the  estimate \eqref{2ndsectorial-1} in the resolvent set
	\be\label{2ndsectorial-1}
	\rho(\mathcal{P}^*)\supset S_{\Theta,\omega}=\left\{\lambda\in \mathds{C}: \lambda\neq\omega, |\arg(\lambda-\omega)|<\Theta, \omega\in\mathds{R}, \Theta\in \left(\frac{\pi}{2}, \pi\right) \right\}.
	\ee
	The estimate for the resolvent operator $(\lambda-\mathcal{P}^*)^{-1}$ follows (see e.g.~Proposition 1.22, Proposition 1.51 and Theorem 1.52 in \cite{OE}):
	\be\label{2ndsectorial-2}
	\|(\lambda-\mathcal{P}^*)^{-1}\|_{\mathcal{B}\left(L^2(\Omega), H_0^1(\Omega)\right)}\leq   \frac{M}{|\lambda-\omega|}
	\ee
	for $\omega\in \mathds{R}$, $M>0$ and $\lambda\in  S_{\Theta,\omega}$. We conclude that $\mathcal{P}^*$ is a sectorial operator which generates an analytic semigroup $\{e^{t\mathcal{P}^*}: t\geq 0\}$ on $H_0^1(\Omega)$. 
\end{proof}
If the domain $D(\mathcal{P}^*)$ of $\mathcal{P}^*$ is endowed with the graph norm of $\mathcal{P}^*$,
$\|g\|_{D(\mathcal{P}^*)}=\|g\|_{L^2(\Omega)}+\|\mathcal{P}^*g\|_{L^2(\Omega)}$,  then there exists a constant $\gamma_0\geq 1$, such that
\be\label{2ndgraphnorm}
 {\gamma_0}^{-1}\left(\|g\|_{L^2(\Omega)}+\|\mathcal{P}^*g\|_{L^2(\Omega)}\right)\leq \|g\|_{H^2(\Omega)}\leq \gamma_0\left(\|g\|_{L^2(\Omega)}+\|\mathcal{P}^*g\|_{L^2(\Omega)}\right).
\ee
In fact,  as $\left\{H^2(\Omega)\cap H_0^1(\Omega)\right\}\hookrightarrow L^2(\Omega)$, $\mathcal{P}^*\in\mathcal{B}\left(H^2(\Omega)\cap H_0^1(\Omega), L^2(\Omega)\right)$, $\mathcal{B}\left(H^2(\Omega)\cap H_0^1(\Omega), L^2(\Omega)\right)$ is a set of  linear operators from $H^2(\Omega)\cap H_0^1(\Omega)$ to $L^2(\Omega)$, there exists a constant $c_0>0$ such that
\[\|g\|_{L^2(\Omega)}+\|\mathcal{P}^*g\|_{L^2(\Omega)}\leq c_0\|g\|_{H^2(\Omega)},\quad \forall\ g\in H^2(\Omega)\cap H_0^1(\Omega),\ i.e.\ \left\{H^2(\Omega)\cap H_0^1(\Omega)\right\}\hookrightarrow D(\mathcal{P}^*).\]
The properties \eqref{2ndsectorial-1} and \eqref{2ndsectorial-2} imply that $\mathcal{P}^*$ extends to a closed operator, which we again denote by $\mathcal{P}^*$, and then $D(\mathcal{P}^*)$ is a complete Banach space. We conclude  $D\left(\mathcal{P}^*\right)= H^2(\Omega)\cap H_0^1(\Omega)$, which is assertion \eqref{2ndgraphnorm}.  

Since $H^2(\Omega)\cap H_0^1(\Omega)$ is dense in $L^2(\Omega)$, hence $\mathcal{P}^*$ is densely defined in $L^2(\Omega)$ and $\overline{D\left(\mathcal{P}^*\right)}=L^2(\Omega)$. If $t>0$ and $\varphi\in L^2(\Omega)$ then $e^{t\mathcal{P}^*}\varphi\in D\left(\left(\mathcal{P}{^*}\right)^k\right)$ for each $k\in\mathds{N}$. Moreover, there exist constants $\overline{M}_0$, $\overline{M}_1$, $\overline{M}_2>0$ (depending on $\Theta$ in \eqref{2ndsectorial-1} and $M$ in \eqref{2ndsectorial-2}), such that
\be\label{2ndanalyticsemigroupbound}
\left\|t^k\left(\mathcal{P}{^*}\right)^ke^{t\mathcal{P}^*}\right\|_{\mathcal{B}(L^2(\Omega))}\leq \overline{M}_k,\quad s>0,\quad k=0,1,2,\quad t\in [0, T_0).
\ee
\begin{thm}\label{2ndlinear-parabolic-equation}
	Let $\mathcal{P}^*:D(\mathcal{P}^*)\to L^2(\Omega)$ be a sectorial operator and generate an analytic semigroup $e^{t\mathcal{P}^*}$, $ D\left(\mathcal{P}^*\right)\cong H^2(\Omega)\cap H_0^1(\Omega)$ and $\overline{D\left(\mathcal{P}^*\right)}=L^2(\Omega)$. Given $\alpha\in(0, 1)$, $T\in(0, T_0)$, 
	\[\tilde{u}_0\in D(\mathcal{P}^*),\quad\mathcal{F}(0)+\mathcal{P}^*\tilde{u}_0\in \overline{D\left(\mathcal{P}^*\right)},\quad \mathcal{F}\in C^\alpha\left([0, T]; L^2(\Omega)\right),\]
	the function
	\be\label{2ndLPE-mild-sol}
	\varphi(t)=e^{t\mathcal{P}^*}\tilde{u}_0+\int_{0}^te^{(t-s)\mathcal{P}^*}\mathcal{F}(s)ds
	\ee
	is the unique solution in $C^1\left([0, T]; L^2(\Omega)\right)\cap C\left([0, T]; D(\mathcal{P}^*)\right)$ of the problem
	\be\label{2nd4th-LPE}
	\varphi'(t)=\mathcal{P}^* \varphi(t)+\mathcal{F}(t),\quad t\in[0, T],\quad \varphi(0)=\tilde{u}_0.
	\ee
	Moreover, the following maximal regularity property holds:
	\[\mathcal{F}\in C^\alpha([0, T]; L^2(\Omega)),\quad \mathcal{P}^*\tilde{u}_0+\mathcal{F}(0)\in D_{\mathcal{P}^*}(\alpha, \infty)\to\]
	\[\varphi\in C^{\alpha+1}([0, T]; L^2(\Omega))\cap C^\alpha\left([0, T]; H^2(\Omega)\cap H_0^1(\Omega)\right),\quad  \varphi'(t)\in D_{\mathcal{P}^*}(\alpha, \infty),\ \forall\ t\in[0, T],\]
	and there exists a continuous and  increasing function $I:\mathds{R}_{+}\rightarrow \mathds{R}_{+}$ (depending on $\overline{M}_0$, $\overline{M}_1$, $\overline{M}_2$ and $\alpha$) such that
	\be\label{2ndlinear-solu-est}
	\|\varphi\|_{C^\alpha\left([0, T]; D(\mathcal{P}^*)\right)}\leq I(T)\left[\|\mathcal{F}\|_{C^\alpha([0, T]; L^2(\Omega))}+\left\|\mathcal{P}^*\tilde{u}_0+\mathcal{F}(0)\right\|_{D_{\mathcal{P}^*}(\alpha, \infty)}+\|\tilde{u}_0\|_{L^2(\Omega)}\right].
	\ee
\end{thm}
Theorem \ref{2ndlinear-parabolic-equation} is a maximal regularity result for linear autonomous evolution equations of parabolic type and its proof is identical to the proof of Theorem 4.5 in  \cite{SE}. 

\subsection*{Estimates of Linearisation Error}
	We are going to use Theorem \ref{2ndlinear-parabolic-equation} to prove the existence of a strict solution to the coupled system, which is Theorem \ref{2ndcp-sys}. Before proving Theorem \ref{2ndcp-sys},  we require certain estimates for the error of the linearisation $F$, which are given in Lemma \ref{2ndRHSMax}. 

\begin{lem}\label{2ndRHSMax}
	Let $F(\tilde{u})$ and $\mathcal{P}^*$ be defined by \eqref{2ndF-def} and \eqref{2ndlinearopt} respectively and fix $T\in(0, T_0)$. If  $\tilde{u}$,  $q\in C^\alpha([0, T]; B_{H^2}(\tilde{u}_0, r))$, with
	$u_0=\tilde{u}_0+\theta_1$, then there exist  constants $L_A=L_A\left(u_0, v_0, w_0, \Omega\right)>0$ and
	$L_B=L_B\left(u_0, v_0, w_0, \Omega, \alpha, T_0, L_U, L_W, L_M\right)>0$, such that for $ 0\leq t<t+h\leq T$,
	\be\label{2ndMax-I}
	\left\|\left[F(\tilde{u})\right](t+h)-\left[F(\tilde{u})\right](t)\right\|_{L^2(\Omega)}\leq \left\{[\tilde{u}+\theta_1]_{C^\alpha\left(\left[0, T\right]; H^2(\Omega)\right)}+L_U\right\}L_Ah^\alpha,
	\ee
	and
	\be\begin{split}
		&\left\|\left[F'(\tilde{u})q\right](t+h)-\left[F'(\tilde{u})q\right](t)-\mathcal{P}^*\left[q(t+h)-q(t)\right]\right\|_{L^2(\Omega)}\\
		\leq& h^\alpha T^\alpha L_{B}\left\|q\right\|_{C^\alpha([0, T]; H^2(\Omega))}+h^\alpha T^\alpha L_{B}\left\|\tilde{u}+\theta_1\right\|_{C^\alpha([0, T]; H^2(\Omega))}\left\|q\right\|_{C^\alpha([0, T]; H^2(\Omega))}\\
		+&h^\alpha L_B\sup_{t\in[0, T]}\left\|q(t)\right\|_{H^2(\Omega)}+h^\alpha L_B\left\|\tilde{u}+\theta_1\right\|_{C^\alpha([0, T]; H^2(\Omega))}\sup_{t\in[0, T]}\left\|q(t)\right\|_{H^2(\Omega)}.\label{2ndMax-II}
	\end{split}\ee
	Here  $L_U$, $L_W$ and $L_M$ are given by Corollary \ref{2Holdercontinuity}, Theorem \ref{2nd-cpl} and Corollary \ref{2Holder-Frechet-W-I-Cor} respectively.
\end{lem}
\begin{proof}
	Let $T\in(0, T_0)$. 	
	According to Theorem \ref{2nd-solu-thm} and Corollary \ref{2Holdercontinuity}, the semilinear evolution equation \eqref{2nd-SWE-1} exists a unique mild solution $(\tilde{v}, \tilde{w})\in Z(T)\cap C^\alpha\left(\left[0, T\right]; L^2(\Omega)\times H_0^1(\Omega)\right)$  provided $\tilde{u}\in C^\alpha\left([0, T]; B_{H^2}(\tilde{u}_0, r)\right)$ for all $r\in\left(0, \frac{\kappa}{2C}\right)$. Here $\kappa=\inf_{x\in{\Omega}}\tilde{w}_0+\theta_2$ and $C=C(\Omega)$ is a constant depending on $\Omega$. 
	Recall that 
	$u_0=\tilde{u}_0+\theta_1$, $v_0=\tilde{v}_0$, $w_0=\tilde{w}_0+\theta_2$, $u=\tilde{u}+\theta_1$, $v=\tilde{v}$, $w=\tilde{w}+\theta_2$. It follows that the solution operators $u\mapsto v$ and $u\mapsto w$ have the following properties:
	\be\label{2ndB-Holder-3}
	\begin{split}
		u\mapsto v:  C^\alpha\left([0, T]; B_{H^2}(u_0, r)\right)\to C^\alpha\left([0, T]; B_{L^2}(v_0, r)\right),\\ 
u\mapsto w:  C^\alpha\left([0, T]; B_{H^2}(u_0, r)\right)\to C^\alpha\left([0, T]; B_{H^1}(w_0, r)\right) , \\
		\left\|v(t+h)-v(t)\right\|_{L^2(\Omega)}\leq L_Uh^\alpha,\quad \left\|w(t+h)-w(t)\right\|_{H^1(\Omega)}\leq L_Uh^\alpha.
	\end{split}	
	\ee
	Hence
	\[F\left(\tilde{u}\right)=\frac{1}{w}\frac{\partial}{\partial x}\left(w^3(\tilde{u}+\theta_1)\frac{\partial\tilde{u}}{\partial x}\right)-\frac{v}{w}(\tilde{u}+\theta_1)\in C^\alpha\left([0, T]; L^2(\Omega)\right). \]
	We next prove assertion \eqref{2ndMax-I} of Lemma \ref{2ndRHSMax}.
	
	Let $h\in(0, T]$ be such that $0\leq t<t+h\leq T$. As $\left\|u(s)\right\|_{H^2(\Omega)}\leq\tilde{C_1}$, ,  $\left\|v(s)\right\|_{L^2(\Omega)}\leq\tilde{C_2}$ and 	$\left\|w(s)\right\|_{H^1(\Omega)}\leq\tilde{C}$ for all $s\in[0, T]$, where $\tilde{C}=\left\|w_0\right\|_{H^1(\Omega)}+{\kappa}/(2C)$, $\tilde{C_1}=\left\|u_0\right\|_{H^2(\Omega)}+{\kappa}/(2C)$, $\tilde{C_2}=\left\|v_0\right\|_{L^2(\Omega)}+{\kappa}/(2C)$.	
	Because  $H^1(\Omega)$ is an algebra,  estimate \eqref{2ndC-a} of Lemma \ref{2ndestimates}, estimate \eqref{2Holdercontinuityformular} from Corollary \ref{2Holdercontinuity} and estimate in \eqref{2ndB-Holder-3},  we obtain
	\begin{align}
		\left\|[w(t+h)]^{-1}-[w(t)]^{-1}\right\|_{H^1(\Omega)}\leq C_1^2L_Uh^\alpha,\quad \left\|[w(t+h)]^3-[w(t)]^3\right\|_{H^1(\Omega)}\leq&3\tilde{C}^2L_Uh^\alpha . \label{2ndB8}
	\end{align}
	Similarly, for $u\in C^\alpha\left([0, T]; B_{H^2}(u_0, r)\right)$, we get
	\begin{align}
		\left\|[u(t+h)]^2-[u(t)]^2\right\|_{H^2(\Omega)}\leq2\tilde{C}_1\left[u\right]_{C^\alpha([0, T]; H^2(\Omega))}h^\alpha.\label{2ndB8-1}
	\end{align}
	The arguments of the proof of Lemma \ref{2ndLip-nonlinearity} give that \eqref{2ndMax-I} of Lemma \ref{2ndRHSMax} holds by
	\begin{equation}
		\left\|\left[F(\tilde{u})\right](t+h)-\left[F(\tilde{u})\right](t)\right\|_{L^2(\Omega)}
		\leq L_AL_Uh^\alpha+L_A[u]_{C^\alpha\left(\left[0, T\right]; H^2(\Omega)\right)}h^\alpha\label{2ndB15}.
	\end{equation}
	Here $L_A$ is a constant depending on $C$, $C_1$, $\tilde{C}$, $\tilde{C}_1$ and $\tilde{C}_2$.	
	
	We next prove the assertion \eqref{2ndMax-II} of Lemma \ref{2ndRHSMax}.		
	For $q\in C^\alpha([0, T]; B_{H^2}(\tilde{u}_0, r))$ and $t\in [0, T]$, we note $w(t)=[w(u)](t)$, $v(t)=[v(u)](t)$. From the definition \eqref{2ndFre-F} of the Frech\'{e}t derivative of $F(\tilde{u})$ on $\tilde{u}$ at $t$ and the definition \eqref{2ndlinearopt} of $\mathcal{P}^*q(t)$, we have, for $h\in (0, T]$ such that $t+h\in (0, T]$,
	\be\begin{split}
		&\left[F'(\tilde{u})q\right](t+h)-\left[F'(\tilde{u})q\right](t)-\mathcal{P}^*\left(q(t+h)-q(t)\right)\\
		=&\frac{1}{w(t+h)}\frac{\partial}{\partial x}\left\{[w(t+h)]^3u(t+h)\frac{\partial q(t+h)}{\partial x}+[w(t+h)]^3\frac{\partial u(t+h)}{\partial x} q(t+h)\right\}\\
		-&\frac{1}{w(t)}\frac{\partial}{\partial x}\left\{[w(t)]^3u(t)\frac{\partial q(t)}{\partial x}+[w(t)]^3\frac{\partial u(t)}{\partial x}q(t)\right\}\\
		+&\frac{3}{2w(t+h)}\frac{\partial}{\partial x}\left\{\frac{\partial[u(t+h)]^2}{\partial x}[w(t+h)]^2[w'(u)q](t+h)\right\}\\
		-&\frac{3}{2w(t)}\frac{\partial}{\partial x}\left\{\frac{\partial[u(t)]^2}{\partial x}[w(t)]^2[w'(u)q](t)\right\}\\
		-&\frac{[w'(u)q](t+h)}{2[w(t+h)]^2}\frac{\partial}{\partial x}\left([w(t+h)]^3\frac{\partial[u(t+h)]^2}{\partial x}\right)-\frac{v(t+h)}{w(t+h)}q(t+h)\\
		+&\frac{[w'(u)q](t)}{2[w(t)]^2}\frac{\partial}{\partial x}\left([w(t)]^3\frac{\partial[u(t)]^2}{\partial x}\right)+\frac{v(t)}{w(t)}q(t)\\
		-&\frac{w(t+h)[v'(u)q](t+h)-v(t+h)[w'(u)q](t+h)}{[w(t+h)]^2}u(t+h)\\
		+&\frac{w(t)[v'(u)q](t)-v(t)[w'(u)q](t)}{[w(t)]^2}u(t)
		-\frac{1}{w_0}\frac{\partial}{\partial x}\left\{w_0^3u_0\frac{\partial q(t+h)}{\partial x}+w_0^3\frac{\partial u_0}{\partial x} q(t+h)\right\}\\
		-&\frac{v_0}{w_0}q(t+h)+\frac{1}{w_0}\frac{\partial}{\partial x}\left\{w_0^3u_0\frac{\partial q(t)}{\partial x}+w_0^3\frac{\partial u_0}{\partial x} q(t)\right\}+\frac{v_0}{w_0}q(t). \label{2ndFre-F-4}
	\end{split}\ee
	Observe that
	\begin{align}
		\left\|[w(t+h)]^3u(t+h)-[w(t)]^3u(t)\right\|_{H^1(\Omega)}
		\leq&\left\|[w(t+h)]^3-[w(t)]^3\right\|_{H^1(\Omega)}\left\|u(t+h)\right\|_{H^1(\Omega)}\notag\\
		+&\left\|[w(t)]^3\right\|_{H^1(\Omega)}\left\|u(t+h)-u(t)\right\|_{H^1(\Omega)}\notag\\
		\leq&h^\alpha3L_U\tilde{C}^2\tilde{C}_1+h^\alpha\tilde{C}^3\left\|u\right\|_{C^\alpha([0, T]; H^2(\Omega))}\label{2ndw3u},
	\end{align}
	with this, the algebraic properties of $H^1(\Omega)$, i.e.~Lemma \ref{alg}, inequalities \eqref{2ndalg-1-2}, \eqref{alg-1-3} of Lemma \ref{2ndalg-1} and the assertion \eqref{2ndC-a} of    Lemma \ref{2ndestimates} imply that
	\begin{align}
		&\left\|\left(\frac{1}{w(t+h)}-\frac{1}{w(t)}\right)\frac{\partial}{\partial x}\left\{[w(t+h)]^3u(t+h)\frac{\partial q(t+h)}{\partial x} \right\}\right\|_{L^2(\Omega)}\notag\\
		\leq&C\left\|[w(t+h)]^{-1}-[w(t)]^{-1}\right\|_{H^1(\Omega)}\left\|w(t+h)^3\right\|_{H^1(\Omega)}\left\|u(t+h)\right\|_{H^2(\Omega)}\left\|q(t+h)\right\|_{H^2(\Omega)}\notag\\
		\leq&h^\alpha CC_1^2L_U\tilde{C}^3\tilde{C}_1\sup_{t\in[0, T]}\left\|q(t)\right\|_{H^2(\Omega)}\label{2ndB20},
	\end{align}
	\begin{align}
		&\left\|\frac{1}{w(t)}\frac{\partial}{\partial x}\left\{\left([w(t+h)]^3u(t+h)-[w(t)]^3u(t)\right)\frac{\partial q(t+h)}{\partial x} \right\}\right\|_{L^2(\Omega)}\notag\\
		\leq&C\left\|[w(t)]^{-1}\right\|_{H^1(\Omega)}\left\|[w(t+h)]^3u(t+h)-[w(t)]^3u(t)\right\|_{H^1(\Omega)}\left\|q(t+h)\right\|_{H^1(\Omega)}\notag\\
		\leq&h^\alpha3CC_1L_U\tilde{C}^2\tilde{C}_1\sup_{t\in[0, T]}\left\|q(t)\right\|_{H^2(\Omega)}+h^\alpha CC_1\tilde{C}^3\left\|u\right\|_{C^\alpha\left([0, T]; H^2(\Omega)\right)}\sup_{t\in[0, T]}\left\|q(t)\right\|_{H^2(\Omega)}\label{2ndBb20}.
	\end{align}
	Hence, we deduce the estimates
	\begin{align}
		&\left\|\frac{1}{w(t+h)}\frac{\partial}{\partial x}\left\{[w(t+h)]^3u(t+h)\frac{\partial q(t+h)}{\partial x}\right\}-\frac{1}{w(t)}\frac{\partial}{\partial x}\left\{[w(t)]^3u(t)\frac{\partial q(t+h)}{\partial x}\right\}\right\|_{L^2(\Omega)}\notag\\
		\leq&\left\|\left(\frac{1}{w(t+h)}-\frac{1}{w(t)}\right)\frac{\partial}{\partial x}\left\{[w(t+h)]^3u(t+h)\frac{\partial q(t+h)}{\partial x} \right\}\right\|_{L^2(\Omega)}\notag\\
		+&\left\|\frac{1}{w(t)}\frac{\partial}{\partial x}\left\{\left([w(t+h)]^3u(t+h)-[w(t)]^3u(t)\right)\frac{\partial q(t+h)}{\partial x}\right\}\right\|_{L^2(\Omega)}\notag\\
		\leq&h^\alpha C\left[C_1^2L_U\tilde{C}^3\tilde{C}_1+3C_1L_U\tilde{C}^2\tilde{C}_1+C_1\tilde{C}^3\left\|u\right\|_{C^\alpha\left([0, T]; H^2(\Omega)\right)}\right]\sup_{t\in[0, T]}\left\|q(t)\right\|_{H^2(\Omega)}\label{2ndB22},
	\end{align} and
	\begin{align}
		&\left\|\frac{1}{w(t)}\frac{\partial}{\partial x}\left([w(t)]^3u(t)\frac{\partial}{\partial x}\left[q(t+h)-q(t)\right]\right)-\frac{1}{w_0}\frac{\partial}{\partial x}\left(w_0^3u_0\frac{\partial}{\partial x}\left[q(t+h)-q(t)\right]\right)\right\|_{L^2(\Omega)}\notag\\
		\leq&\left\|\frac{1}{w(t)}\frac{\partial}{\partial x}\left(\left\{[w(t)]^3u(t)-w_0^3u_0\right\}\frac{\partial}{\partial x}\left[q(t+h)-q(t)\right]\right)\right\|_{L^2(\Omega)}\notag\\
		+&\left\|\left(\frac{1}{w(t)}-\frac{1}{w_0}\right)\frac{\partial}{\partial x}\left(w_0^3u_0\frac{\partial}{\partial x}\left[q(t+h)-q(t)\right]\right)\right\|_{L^2(\Omega)}\notag\\
		\leq&C\left\|[w(t)]^{-1}\right\|_{H^1(\Omega)}\left\|[w(t)]^3u(t)-w_0^3u_0\right\|_{H^1(\Omega)}\left\| q(t+h)- q(t)\right\|_{H^2(\Omega)}\notag\\
		+&C\left\|[w(t)]^{-1}-[w_0]^{-1}\right\|_{H^1(\Omega)}\left\|w_0\right\|_{H^1(\Omega)}^3\left\|u_0\right\|_{H^2(\Omega)}\left\| q(t+h)- q(t)\right\|_{H^2(\Omega)}\notag\\
		\leq&h^\alpha T^\alpha CC_1\left(3L_U\tilde{C}^2\tilde{C}_1+\tilde{C}^3\left\|u\right\|_{C^\alpha([0, T]; H^2(\Omega))}\right)\left\|q\right\|_{C^\alpha([0, T];H^2(\Omega))}\notag\\
		+&h^\alpha T^\alpha C_1L_U\tilde{C}^3\tilde{C}_1\left\|q\right\|_{C^\alpha([0, T];H^2(\Omega))}\label{2ndB27}.
	\end{align}
Therefore, the triangle inequality, \eqref{2ndB22} and \eqref{2ndB27} imply, 	with a constant $V_1$ which depends on $C$, $C_1$, $\tilde{C}$, $\tilde{C}_1$, and $L_U$: 
	\begin{align}
		&\bigg\|\frac{1}{w(t+h)}\frac{\partial}{\partial x}\left\{[w(t+h)]^3u(t+h)\frac{\partial q(t+h)}{\partial x}\right\}-\frac{1}{w(t)}\frac{\partial}{\partial x}\left\{[w(t)]^3u(t)\frac{\partial q(t)}{\partial x}\right\}\notag\\
		&-\frac{1}{w_0}\frac{\partial}{\partial x}\left\{w_0^3u_0\frac{\partial q(t+h)}{\partial x}\right\}+\frac{1}{w_0}\frac{\partial}{\partial x}\left\{w_0^3u_0\frac{\partial q(t)}{\partial x}\right\}\bigg\|_{L^2(\Omega)}\notag\\
		\leq&h^\alpha V_1\sup_{t\in[0, T]}\left\|q(t)\right\|_{H^2(\Omega)}+h^\alpha V_1\left\|u\right\|_{C^\alpha\left([0, T];H^2(\Omega)\right)}\sup_{t\in[0, T]}\left\|q(t)\right\|_{H^2(\Omega)}\notag\\
		+&h^\alpha T^\alpha V_1\left\|u\right\|_{C^\alpha([0,T];H^2(\Omega))}\left\|q\right\|_{C^\alpha([0,T];H^2(\Omega))}+h^\alpha T^\alpha V_1\left\|q\right\|_{C^\alpha([0, T];H^2(\Omega))}\label{2ndB0}.
	\end{align}
	Similarly,
	\begin{align}
		&\bigg\|\frac{1}{w(t+h)}\frac{\partial}{\partial x}\left\{[w(t+h)]^3q(t+h)\frac{\partial u(t+h)}{\partial x}\right\}-\frac{1}{w(t)}\frac{\partial}{\partial x}\left\{[w(t)]^3q(t)\frac{\partial u(t)}{\partial x}\right\}\notag\\
		&-\frac{1}{w_0}\frac{\partial}{\partial x}\left\{w_0^3q(t+h)\frac{\partial u_0}{\partial x}\right\}+\frac{1}{w_0}\frac{\partial}{\partial x}\left\{w_0^3q(t)\frac{\partial u_0}{\partial x}\right\}\bigg\|_{L^2(\Omega)}\notag\\
		\leq&h^\alpha V_1\sup_{t\in[0, T]}\left\|q(t)\right\|_{H^2(\Omega)}+h^\alpha V_1\left\|u\right\|_{C^\alpha\left([0, T];H^2(\Omega)\right)}\sup_{t\in[0, T]}\left\|q(t)\right\|_{H^2(\Omega)}\notag\\
		+&h^\alpha T^\alpha V_1\left\|q\right\|_{C^\alpha([0, T]; H^2(\Omega))}+h^\alpha T^\alpha V_1\left\|u\right\|_{C^\alpha\left([0, T]; H^2(\Omega)\right)}\left\|q\right\|_{C^\alpha([0, T]; H^2(\Omega))}.\label{2ndB30-1}
	\end{align}
	Set $\tilde{C}_3=C\left[C_1L_U+\tilde{C}_2C_1^2L_U\right]$, $\tilde{C}_4=CC_1\left[L_U+\left\|v_0\right\|_{L^2(\Omega)}\left\|w_0^{-1}\right\|_{H^2(\Omega)}\right]$. The triangle inequality, algebraic properties of Sobolev spaces, i.e.~\eqref{2ndalg-1-1} of Lemma \ref{2ndalg-1}, \eqref{2ndC-a} of    Lemma \ref{2ndestimates}, and the assertion \eqref{2Holdercontinuityformular} of Corollary \ref{2Holdercontinuity} imply the estimate
	\begin{align}
		\left\|-\frac{v(t+h)}{w(t+h)}q(t+h)+\frac{v(t)}{w(t)}q(t)-\frac{v_0}{w_0}q(t+h)+\frac{v_0}{w_0}q(t)\right\|_{L^2(\Omega)}
		\leq& h^\alpha \tilde{C}_3\sup_{t\in[0, T]}\left\|q(t)\right\|_{H^2(\Omega)}\notag\\
		+&h^\alpha T^\alpha \tilde{C}_4\left\|q\right\|_{C([0, T]; H^2(\Omega))}\notag.
	\end{align}
	We now combine \eqref{2bound-Fre-W}, \eqref{2Holder-Frechet-W-I} in Corollary \ref{2Holder-Frechet-W-I-Cor}, \eqref{2Holdercontinuityformular} in Corollary \ref{2Holdercontinuity} with the above arguments for estimate \eqref{2ndB0}. With a constant $V_2$ which depends on $C$, $C_1$ $C_2$, $\tilde{C}$, $\tilde{C}_1$, $\tilde{C}_2$, $L_U$, $L_W$, $L_M$ and $T_0^{1-\alpha}$,  we deduce
	\begin{align}
		&\frac{3}{2}\bigg\|\frac{1}{w(t+h)}\frac{\partial}{\partial x}\left\{[w(t+h)]^2[w'(u)q](t+h)\frac{\partial[u(t+h)]^2}{\partial x}\right\}\notag\\
		&-\frac{1}{w(t)}\frac{\partial}{\partial x}\left\{[w(t)]^2[w'(u)q](t)\frac{\partial[u(t)]^2}{\partial x}\right\}\bigg\|_{L^2(\Omega)}\notag\\
		\leq&h^\alpha T^\alpha V_2\left\|q\right\|_{C^\alpha([0, T]; H^2(\Omega))}
		+h^\alpha V_2\left(1+\left\|u\right\|_{C^\alpha([0, T]; H^2(\Omega))}\right)\sup_{t\in[0, T]}\left\|q(t)\right\|_{H^2(\Omega)}\notag,
	\end{align}
	\begin{align}
		&\left\|\frac{[w'(u)q](t+h)}{2[w(t+h)]^2}\frac{\partial}{\partial x}\left\{[w(t+h)]^3\frac{\partial[u(t+h)]^2}{\partial x}\right\}-\frac{[w'(u)q](t)}{2[w(t)]^2}\frac{\partial}{\partial x}\left\{[w(t)]^3\frac{\partial[u(t)]^2}{\partial x}\right\}\right\|_{L^2(\Omega)}\notag\\
		\leq&  h^\alpha V_2\left(1+\|u\|_{C^\alpha\left(\left[0, T\right]; H^2(\Omega)\right)}\right)\sup_{t\in[0, T]}\|q(t)\|_{H^2(\Omega)}
		+h^\alpha T^\alpha V_2\left\|q\right\|_{C^\alpha([0, T]; H^2(\Omega))}\notag,
	\end{align}
	\begin{align}
		&\bigg\|-\frac{w(t+h)[v'(u)q](t+h)-v(t+h)[w'(u)q](t+h)}{[w(t+h)]^2}u(t+h)\notag\\
		&+\frac{w(t)[v'(u)q](t)-v(t)[w'(u)q](t)}{[w(t)]^2}u(t)\bigg\|_{L^2(\Omega)}\notag\\
		\leq&h^\alpha V_2\left(1+\left\|u\right\|_{C^\alpha([0, T]; H^2(\Omega))}\right)\sup_{t\in[0, T]}\left\|q(t)\right\|_{H^2(\Omega)}+h^\alpha T^\alpha V_2\left\|q\right\|_{C^\alpha([0, T]; H^2(\Omega))}\notag.
	\end{align}
	Consequently, by setting $L_B=V_1+V_2+\tilde{C}_3+\tilde{C}_4$, we obtain
	\begin{align}
		&\left\|\left[F'(\tilde{u})q\right](t+h)-\left[F'(\tilde{u})q\right](t)-\mathcal{P}^*\left[q(t+h)-q(t)\right]\right\|_{L^2(\Omega)}\notag\\
		\leq&h^\alpha L_B\sup_{t\in[0, T]}\left\|q(t)\right\|_{H^2(\Omega)}+h^\alpha L_B\left\|u\right\|_{C^\alpha\left([0, T];H^2(\Omega)\right)}\sup_{t\in[0, T]}\left\|q(t)\right\|_{H^2(\Omega)}\notag\\
		+&h^\alpha T^\alpha L_B\left\|q\right\|_{C^\alpha([0, T];H^2(\Omega))}+h^\alpha T^\alpha L_B\left\|u\right\|_{C^\alpha([0, T]; H)}\left\|q\right\|_{C^\alpha([0, T]; H^2(\Omega))}\notag.
	\end{align}
	Hereby, \eqref{2ndMax-II} is proved and this concludes the proof of Lemma \ref{2ndRHSMax}.
\end{proof}

\subsection*{Proof of Theorem \ref{2ndcoupled system}}

\begin{thm}\label{2ndcp-sys}
	Assume that the initial value $u_0\in\left\{\psi\in H^{2+\sigma}(\Omega):\ \psi(x)=\theta_1,\ x\in\partial\Omega\right\}$ is given for $\sigma\in\left(0, \frac{1}{2}\right)$ such that the compatibility condition
	\[\frac{1}{w_0}\frac{\partial}{\partial x}\left(w_0^3u_0\frac{\partial u_0}{\partial x}\right)\in H^\sigma(\Omega)\subseteq D_{\mathcal{P}^*}(\alpha, \infty)\]
	holds for $\alpha\in\left(0,  \frac{\sigma}{2}\right)$ and $\tilde{u}_0=u_0-\theta_1\in H^{2+\sigma}(\Omega)\cap H_0^1(\Omega)$. 
	
	Then there exists $T_1>0$, such that the nonlinear problem \eqref{2ndNPE} has a unique strict solution $\tilde{u}\in C^\alpha\left([0, T_1); H^2(\Omega)\cap H_0^1(\Omega)\right)\cap C^{\alpha+1}\left([0, T_1); L^2(\Omega)\right)$ and $\tilde{u}'(t)\in D_{\mathcal{P}^*}(\alpha, \infty)$ for all $ t\in[0, T_1)$. \end{thm}
\begin{proof}
	Let $\sigma\in\left(0, \frac{1}{2}\right)$, $\alpha\in\left(0, \frac{\sigma}{2}\right)$. We divide the proof into three parts.
	
	\noindent\textbf{H\"{o}lder Continuity.}
	Let us first state  H\"{o}lder continuity results for the solution operators
	$\tilde{u}\to v(\tilde{u}+\theta_1)$ and $\tilde{u}\to w(\tilde{u}+\theta_1)$ in the Section \ref{SolnOp} which will be used below.	
	Take $T\in(0, T_0)$ to be specified later.
	The estimate \eqref{2Holdercontinuityformular} in Corollary \ref{2Holdercontinuity} implies that the solution operators $\tilde{u}\to v(\tilde{u}+\theta_1)$ and $\tilde{u}\to w(\tilde{u}+\theta_1)$  satisfy the mapping properties 
	\[\tilde{u}\mapsto v(\tilde{u}+\theta_1):\quad C^\alpha\left([0, T]; B_{H^2}(\tilde{u}_0, r)\right) \to C^\alpha\left([0, T]; B_{L^2}(v_0, r)\right),\]
	\[\tilde{u}\mapsto w(\tilde{u}+\theta_1):\quad C^\alpha\left([0, T]; B_{H^2}(\tilde{u}_0, r)\right) \to C^\alpha\left([0, T]; B_{H^1}(w_0, r)\right).\]
	Thus, from  inequality \eqref{2ndMax-I} in Lemma \ref{2ndRHSMax},  the nonlinearity $F(\tilde{u})$ from \eqref{2ndF-def} satisfies
	\[F(\tilde{u})\in C^\alpha\left([0, T]; L^2(\Omega)\right).\]
	Due to Theorem \ref{2nd-cpl} and its following discussion of the Fr\'{e}chet derivative, together with the estimate \eqref{2Lip-Frechet-W-I} of Corollary \ref{2Lip-Frechet-W}, we obtain that the Fr\'{e}chet derivative $\left(v'(\tilde{u}+\theta_1)q, w'(\tilde{u}+\theta_1)q\right)$ of the function $\left(v(\tilde{u}+\theta_1), w(\tilde{u}+\theta_1)\right)$ on $\tilde{u}\in C\left([0, T]; H^2(\Omega)\cap H_0^1(\Omega)\right)$ exists in $C\left([0, T]; L^2(\Omega)\times H_0^1(\Omega)\right)$ and depends Lipschitz continuously on $\tilde{u}\in C\left([0, T]; H^2(\Omega)\cap H_0^1(\Omega)\right)$ for $q\in C\left([0, T]; H^2(\Omega)\cap H_0^1(\Omega)\right)$.
	
	 If $\tilde{u}\in C^\alpha\left([0, T]; B_{H^2}(\tilde{u}_0, r)\right)$,  then inequality \eqref{2Holder-Frechet-W-I} in Corollary \ref{2Holder-Frechet-W-I-Cor} implies that  \[\left(v'(\tilde{u}+\theta_1)q, w'(\tilde{u}+\theta_1)q\right)\in C^\alpha\left([0, T]; L^2(\Omega)\times H_0^1(\Omega)\right)\] for all $ q\in C^\alpha\left([0, T]; H^2(\Omega)\cap H_0^1(\Omega)\right)$. Thus, following  \eqref{2ndMax-II} in Lemma \ref{2ndRHSMax},  the Fr\'{e}chet derivative $F'(\tilde{u})q$ of $F(\tilde{u})$ and $\mathcal{P}^*q$ (defined by \eqref{2ndFre-F} and \eqref{2ndlinearopt}, respectively) satisfy
	\[
	F'(\tilde{u})q-\mathcal{P}^*q\in C^\alpha\left([0, T]; L^2(\Omega)\right),\quad \forall\ q\in C^\alpha\left([0, T]; H^2(\Omega)\cap H_0^1(\Omega)\right).
	\]
Now	$v_0\in H_0^1(\Omega)$, $w_0\in H^2(\Omega)$ with $w_0|_{\partial\Omega}=\theta_2$ and the compatibility assumption of Theorem \ref{2ndcp-sys} imply
	\[[F(\tilde{u})](0)=\frac{1}{w_0}\frac{\partial}{\partial x}\left[w_0^3u_0\frac{\partial u_0}{\partial x}\right]-\frac{v_0}{w_0}u_0\in H^\sigma(\Omega)\subseteq D_{\mathcal{P}^*}(\alpha, \infty),\quad \tilde{u}_0\in D(\mathcal{P}^*).\]
	From the definition \eqref{2ndlinearopt} of $\mathcal{P}^*$ and equation \eqref{2ndgraphnorm}, we conclude
	\[D(\mathcal{P}^*)= H^2(\Omega)\cap H_0^1(\Omega),\ \overline{D(\mathcal{P}^*)}=L^2(\Omega),\ C^\alpha\left([0, T]; H^2(\Omega)\cap H_0^1(\Omega)\right)= C^\alpha\left([0, T]; D(\mathcal{P}^*)\right).\]
	\noindent\textbf{Equivalence.}
	We now study the nonlinear problem
	\be\label{2nd2nd-NLP}
	\tilde{u}'(t)=\mathcal{P}^*\tilde{u}(t)+[F(\tilde{u})](t)-\mathcal{P}^*\tilde{u}(t), \quad t\in[0, T],\quad \tilde{u}(0)=\tilde{u}_{0}.
	\ee
	in its integral formulation
	\be\label{2ndintegralform}
	\tilde{u}(t)=e^{t\mathcal{P}^*}\tilde{u}_0+ \int_0^te^{(t-s)\mathcal{P}^*}\left\{[F(\tilde{u})](s)-\mathcal{P}^*\tilde{u}(s)\right\}ds,\quad t\in [0, T].
	\ee
	We will prove that if $\tilde{u}\in C^\alpha\left([0, T]; H^2(\Omega)\cap H_0^1(\Omega)\right)$ satisfies \eqref{2ndintegralform}, $\tilde{u}(t)\in B_{H^2}(\tilde{u}_0, r)$ for all $ t\in [0, T]$, $r\in\left(0, \frac{\kappa}{2C}\right)$ for   $\kappa= \inf_{x\in{\Omega}}w_0(x)$ and $C=C(\Omega)>0$ is a constant, then
	$\tilde{u}\in C^{\alpha+1}\left([0, T]; L^2(\Omega)\right)$, $\tilde{u}'(t)\in D_{\mathcal{P}^*}(\alpha, \infty)$ for all $ t\in[0, T]$, and $\tilde{u}$ satisfies the equation \eqref{2nd2nd-NLP}. 

To prove this assertion, set
	\be\label{2ndpertubation}
	[\mathcal{F}(\tilde{u})](t)=[F(\tilde{u})](t)-\mathcal{P}^*\tilde{u}(t),\quad\forall\ t\in[0, T],
	\ee
for $\tilde{u}\in C^\alpha\left([0, T]; B_{H^2}(\tilde{u}_0, r)\right)$.
	We will show that
	\be\label{2ndHolderRHS-2nd}
	\mathcal{F}(\tilde{u})\in C^\alpha\left([0, T]; L^2(\Omega)\right).
	\ee
	In fact, let $0\leq t<t+h\leq T$. Using \eqref{2ndMax-I} in Lemma \ref{2ndRHSMax}, we observe
	\begin{align}
		\left\|[\mathcal{F}(\tilde{u})](t+h)-[\mathcal{F}(\tilde{u})](t)\right\|_{L^2(\Omega)}\leq&\left\|[F(\tilde{u})](t+h)-[F(\tilde{u})](t)\right\|_{L^2(\Omega)}
		+\left\|\mathcal{P}^*\tilde{u}(t+h)-\mathcal{P}^*\tilde{u}(t)\right\|_{L^2(\Omega)}\notag\\
		\leq&\left(L_A+\left\|\mathcal{P}^*\right\|_{\mathcal{B}\left(H^2(\Omega), L^2(\Omega)\right)}\right)\left\|\tilde{u}(t+h)-\tilde{u}(t)\right\|_{H^2(\Omega)},\notag
	\end{align}
	because $\tilde{u}\in C^\alpha\left([0, T]; H^2(\Omega)\cap H_0^1(\Omega)\right)$, we get \eqref{2ndHolderRHS-2nd}.
	
	In addition, as $\tilde{u}(0)=\tilde{u}_0\in\left\{H^2(\Omega)\cap H_0^1(\Omega)\right\} = D(\mathcal{P}^*)$ and $[F(\tilde{u})](0)\in D_{\mathcal{P}^*}(\alpha, \infty)$, we get
	\[\mathcal{P}^*\tilde{u}_0+[\mathcal{F}(\tilde{u})](0)=[F(\tilde{u})](0)\in D_{\mathcal{P}^*}(\alpha, \infty).\]
	Therefore, using Theorem \ref{2ndlinear-parabolic-equation} and Theorem 1.2 of \cite{LS1}, we obtain that if $\tilde{u}\in C^\alpha\left([0, T]; B_{H^2}(\tilde{u}_0, r)\right)$ is a solution of \eqref{2ndintegralform}, then there exist
	$\tilde{u}'\in C^{\alpha}\left([0, T]; L^2(\Omega)\right)$, $\tilde{u}'(t)\in D_{\mathcal{P}^*}(\alpha, \infty)$ for all $ t\in[0, T]$,
	and $\tilde{u}$ satisfies \eqref{2nd2nd-NLP}.
	
	Conversely, let $\tilde{u}\in C^\alpha\left([0, T]; B_{H^2}(\tilde{u}_0, r)\right)\cap C^{\alpha+1}\left([0, T]; L^2(\Omega)\right)$ satisfy \eqref{2nd2nd-NLP}, i.e.
	\[\tilde{u}'(t)=\mathcal{P}^*\tilde{u}(t)+[\mathcal{F}(\tilde{u})](t), \quad t\in[0, T],\quad \tilde{u}(0)=\tilde{u}_{0}.\]
	As we have proved that $\mathcal{F}(\tilde{u})\in C^\alpha\left([0, T]; L^2(\Omega)\right)$, we can apply  Theorem \ref{2ndlinear-parabolic-equation} again and obtain that $\tilde{u}$ is a solution of the integral equation \eqref{2ndintegralform}.
	
	In conclusion, it is sufficient to solve \eqref{2ndintegralform} in the space $C^\alpha\left([0, T]; B_{{H^2}}(\tilde{u}_0, r)\right)$. To do so, we let $T\in(0, T_0)$ to be specified later and find a fixed point for the mapping 
	\be\label{2ndintegralform2}
	\Gamma: Y\to Y,\quad [\Gamma\tilde{u}](t)=e^{t\mathcal{P}^*}\tilde{u}_0+ \int_0^te^{(t-s)\mathcal{P}^*}\left\{[F(\tilde{u})](s)-\mathcal{P}^*\tilde{u}(s)\right\}ds,\quad t\in [0, T],
	\ee
on the space
	\[Y=\left\{\tilde{u}\in C^\alpha\left([0, T]; H^2(\Omega)\cap H_0^1(\Omega)\right):\ \tilde{u}(0)=\tilde{u}_0,\ \left\|\tilde{u}(\cdot)-\tilde{u}_0\right\|_{C^\alpha([0, T]; H^2(\Omega))}\leq r\right\},\ \forall\ r\in\left(0, \frac{\kappa}{2C}\right).\]
	\noindent\textbf{Contraction Mapping.} The space $Y$ is a metric space in the metric defined from the norm of $C^\alpha\left([0, T]; H^2(\Omega)\right)$. We will show that $\Gamma$ is a contractive mapping of $Y$ into itself provided $T$ is sufficiently small.
	
	Following the preceding results and proceeding as in the proof of Theorem 4.3.1 in Lunardi's work \cite{LA2}, we obtain that $\Gamma u\in C^\alpha\left([0, T]; H^2(\Omega)\cap H_0^1(\Omega)\right)$ if $\tilde{u}\in Y$, as $Y\subset C^\alpha\left([0, T]; B_{H^2}(\tilde{u}_0, r)\right)$. We will show that when $T$ is sufficiently small we have
	\be\label{2ndcontractivemap}
	\left\|\Gamma\tilde{u}_1-\Gamma\tilde{u}_2\right\|_{C^\alpha\left([0, T]; H^2(\Omega)\right)}\leq\frac{1}{2}\left\|\tilde{u}_1-\tilde{u}_2\right\|_{C^\alpha\left([0, T]; H^2(\Omega)\right)},\quad\forall\ \tilde{u}_1,\ \tilde{u}_2\in Y.
	\ee
	From \eqref{2ndpertubation} and \eqref{2ndintegralform2}, we get
	\[[\Gamma\tilde{u}_1](t)-[\Gamma\tilde{u}_2](t)= \int_0^te^{(t-s)\mathcal{P}^*}\left\{[\mathcal{F}(\tilde{u}_1)](s)-[\mathcal{F}(\tilde{u}_2)](s)\right\}ds,\quad t\in[0, T],\]
	hence, by using \eqref{2ndgraphnorm} and applying \eqref{2ndlinear-solu-est} from Theorem \ref{2ndlinear-parabolic-equation}, we obtain
	\begin{align}
		\left\|\Gamma\tilde{u}_1-\Gamma\tilde{u}_2\right\|_{C^\alpha\left([0, T]; H^2(\Omega)\right)}\leq& \gamma_0I(T)\left\|\mathcal{F}(\tilde{u}_1)-\mathcal{F}(\tilde{u}_2)\right\|_{C^\alpha\left([0, T]; L^2(\Omega)\right)}\notag\\
		\leq&\gamma_0I(T_0)\left\|\mathcal{F}(\tilde{u}_1)-\mathcal{F}(\tilde{u}_2)\right\|_{C^\alpha\left([0, T]; L^2(\Omega)\right)}.\label{2ndGammacontraction}
	\end{align}
	Here $I(\cdot)$ is a continuous and increasing function given by Theorem \ref{2ndlinear-parabolic-equation} when applied to $\mathcal{P}^*$ which is defined by \eqref{2ndlinearopt} and satisfies Lemma \ref{2ndlinearisation-elliptic} and Corollary \ref{2ndgeneratoroflinearisation}. As $\tilde{u}_1(t)$ and $\tilde{u}_2(t)$ belong to $B_{H^2}(\tilde{u}_0, r)$ for $t\in[0, T]$, we can use inequality \eqref{2ndnonlinearity-Lip-0} in Lemma \ref{2ndLip-nonlinearity} to estimate the right hand side, obtaining for $t\in [0, T]$,
	\begin{align}
		\left\|[\mathcal{F}(\tilde{u}_1)](t)-[\mathcal{F}(\tilde{u}_2)](t)\right\|_{L^2(\Omega)}\leq&\left\|[F(\tilde{u}_1)](t)-[F(\tilde{u}_2)](t)\right\|_{L^2(\Omega)}+\left\|\mathcal{P}^*\tilde{u}_1(t)-\mathcal{P}^*\tilde{u}_2(t)\right\|_{L^2(\Omega)}\notag\\
		\leq&\left(L_e+\left\|\mathcal{P}^*\right\|_{\mathcal{B}\left(H^2(\Omega), L^2(\Omega)\right)}\right)\left\|\tilde{u}_1(t)-\tilde{u}_2(t)\right\|_{H^2(\Omega)}.\notag
	\end{align}
	As $\tilde{u}_1(0)=\tilde{u}_2(0)=\tilde{u}_0\in \left\{H^2(\Omega)\cap H_0^1(\Omega)\right\}= D(\mathcal{P}^*)$, then we have
	\be\label{2ndu1}
	\sup_{t\in[0, T]}\left\|\tilde{u}_1(t)-\tilde{u}_2(t)\right\|_{H^2(\Omega)}\leq T^\alpha\left\|\tilde{u}_1-\tilde{u}_2\right\|_{C^\alpha\left([0, T]; H^2(\Omega)\right)}
	\ee
	and so
	\be\label{2ndF1}
	\left\|\mathcal{F}(\tilde{u}_1)-\mathcal{F}(\tilde{u}_2)\right\|_{C\left([0, T]; L^2(\Omega)\right)}\leq \left[L_e+\left\|\mathcal{P}^*\right\|_{\mathcal{B}\left(H^2(\Omega), L^2(\Omega)\right)}\right]T^\alpha\left\|\tilde{u}_1-\tilde{u}_2\right\|_{C^\alpha\left([0, T]; H^2(\Omega)\right)}
	\ee
	On the other hand, for $0\leq t<t+h\leq T$, by using\eqref{2ndMax-II} in Lemma \ref{2ndRHSMax} and \eqref{2ndu1}, we get
	\begin{align}
		&\left\|[\mathcal{F}(\tilde{u}_1)](t+h)-[\mathcal{F}(\tilde{u}_2)](t+h)-[\mathcal{F}(\tilde{u}_1)](t)+[\mathcal{F}(\tilde{u}_2)](t)\right\|_{L^2(\Omega)}\notag\\
		=&\bigg\|\int_0^1[F'(\gamma\tilde{u}_1+(1-\gamma)\tilde{u}_2)(\tilde{u}_1-\tilde{u}_2)](t+h)-[F'(\gamma\tilde{u}_1+(1-\gamma)\tilde{u}_2)(\tilde{u}_1-\tilde{u}_2)](t)\notag\\
		-&\mathcal{P}^*\left[\tilde{u}_1(t+h)-\tilde{u}_2(t+h)-\tilde{u}_1(t)+\tilde{u}_2(t)\right]d\gamma\bigg\|_{L^2(\Omega)}\notag\\
		\leq&2h^\alpha T^\alpha L_{B}\left\|\tilde{u}_1-\tilde{u}_2\right\|_{C^\alpha\left([0, T]; H^2(\Omega)\right)}\notag\\
		+&2h^\alpha T^\alpha L_{B}\left\|\gamma\tilde{u}_1+(1-\gamma)\tilde{u}_2+\theta_1\right\|_{C^\alpha\left([0, T]; H^2(\Omega)\right)}\left\|\tilde{u}_1-\tilde{u}_2\right\|_{C^\alpha\left([0, T]; H^2(\Omega)\right)}\notag\\
		\leq&2h^\alpha T^\alpha L_{B}\left(1+\|u_0\|_{H^2(\Omega)}+\kappa(2C)^{-1}\right)\left\|\tilde{u}_1-\tilde{u}_2\right\|_{C^\alpha\left([0, T]; H^2(\Omega)\right)}\label{2ndcontraction-2},
	\end{align} and
		\begin{align}
		&[\mathcal{F}(\tilde{u}_1)-\mathcal{F}(\tilde{u}_2)]_{C^\alpha\left([0, T]; L^2(\Omega)\right)}\notag\\
		=&\sup_{0\leq t<t+h\leq T}\frac{1}{h^\alpha}\left\{{\|[\mathcal{F}(\tilde{u}_1)](t+h)-[\mathcal{F}(\tilde{u}_2)](t+h)-[\mathcal{F}(\tilde{u}_1)](t)+[\mathcal{F}(\tilde{u}_2)](t)\|_{L^2(\Omega)}}\right\}\notag\\
		\leq&2T^\alpha L_B\left(1+\|u_0\|_{H^2(\Omega)}+\kappa(2C)^{-1}\right)\left\|\tilde{u}_1-\tilde{u}_2\right\|_{C^\alpha\left([0, T]; H^2(\Omega)\right)}.\label{2ndcontraction-5}
	\end{align}
	Thus we can deduce from \eqref{2ndGammacontraction}, \eqref{2ndF1} and \eqref{2ndcontraction-5}:
	\begin{align}
		&\left\|\Gamma\tilde{u}_1-\Gamma\tilde{u}_2\right\|_{C^\alpha\left([0, T]; H^2(\Omega)\right)}\notag\\
		\leq&\gamma_0I(T_0)\left\|\mathcal{F}(\tilde{u}_1)-\mathcal{F}(\tilde{u}_2)\right\|_{C^\alpha\left([0, T]; L^2(\Omega)\right)}\notag\\
		\leq&\gamma_0I(T_0)\left(\left\|\mathcal{F}(\tilde{u}_1)-\mathcal{F}(\tilde{u}_2)\right\|_{C\left([0, T]; L^2(\Omega)\right)}+[\mathcal{F}(\tilde{u}_1)-\mathcal{F}(\tilde{u}_2)]_{C^\alpha([0, T]; L^2(\Omega))}\right)\notag\\
		\leq&\gamma_0I(T_0)\left[L_e+\left\|\mathcal{P}^*\right\|_{\mathcal{B}(H^2(\Omega), L^2(\Omega))}+2L_B\left(1+\|u_0\|_{H^2(\Omega)}+\kappa(2C)^{-1}\right)\right]T^\alpha\notag\\
		\cdot&\left\|\tilde{u}_1-\tilde{u}_2\right\|_{C^\alpha([0, T]; H^2(\Omega))}.\label{2ndfinalcontraction}
	\end{align}
	Set
	\be\label{2ndtimedef-1}
	T_0^*:= {\left[2\gamma_0I(T_0)\left(L_e+\left\|\mathcal{P}^*\right\|_{\mathcal{B}(H^2(\Omega), L^2(\Omega))}+2L_B\left(1+\|u_0\|_{H^2(\Omega)}+\kappa(2C)^{-1}\right)\right)\right]^{-\frac{1}{\alpha}}}.
	\ee
	If $0<T\leq \min\{T_0,\ T_0^*\}$, then $\Gamma$ satisfies the contraction property \eqref{2ndcontractivemap} by using \eqref{2ndfinalcontraction}.
	
	To prove that $\Gamma(Y)\subseteq Y$, it remains to check that
	\be\label{2ndmapinto}
	\left\|\Gamma\tilde{u}-\tilde{u}_0\right\|_{C^\alpha\left([0, T]; H^2(\Omega)\right)}\leq r, \quad \forall\ \tilde{u}\in Y.
	\ee
	Let us observe that if $0<T<\min\{T_0,\ T_0^*\}$, then from \eqref{2ndcontractivemap}, we get
	\begin{align}
		\left\|\Gamma\tilde{u}-\tilde{u}_0\right\|_{C^\alpha([0, T]; H^2(\Omega))}\leq&\left\|\Gamma\tilde{u}-\Gamma\tilde{u}_0\right\|_{C^\alpha([0, T]; H^2(\Omega))}+\left\|\Gamma\tilde{u}_0-\tilde{u}_0\right\|_{C^\alpha([0, T]; H^2(\Omega))}\notag\\
		\leq&\frac{1}{2}\left\|\tilde{u}-\tilde{u}_0\right\|_{C^\alpha([0, T]; H^2(\Omega))}+\left\|\Gamma\tilde{u}_0-\tilde{u}_0\right\|_{C^\alpha([0, T]; H^2(\Omega))}\notag\\
		\leq&\frac{r}{2}+\left\|\Gamma\tilde{u}_0-\tilde{u}_0\right\|_{C^\alpha([0, T]; H^2(\Omega))}\notag.
	\end{align}
	Now $\Gamma\tilde{u}_0-\tilde{u}_0\in C^\alpha\left([0, T]; H^2(\Omega)\cap H_0^1(\Omega)\right)$ and it vanishes at $t=0$, so there exists a $\delta^*=\delta^*(r)>0$ such that, if $0<T\leq\delta^*$, then
	\[\left\|\Gamma\tilde{u}_0-\tilde{u}_0\right\|_{C^\alpha\left([0, T]; H^2(\Omega)\right)}\leq \frac{r}{2},\]
	and consequently \eqref{2ndmapinto} is true by choosing $0<T\leq\min\{T_0,\ T_0^*,\ \delta^*\}$. Because $r$ is controlled by $C=C(\Omega)$ and $\kappa$, we also note $\delta^*=\delta^*(\kappa,\Omega)$.
	
	Summing up, set
	\be\label{2ndtimedef}
	T_1:=\min\left\{ T_0,\ T_0^*,\ \delta^*\right\},
	\ee
	$\Gamma$, defined by \eqref{2ndintegralform2},  is a contractive mapping of $Y$ into itself provided
	\[0<T< T_1.\]
	Hence, $\Gamma$ has a unique fixed point $\tilde{u}$ in $Y$, $\tilde{u}\in C^\alpha\left([0, T_1); H^2(\Omega)\cap H_0^1(\Omega)\right)$ is a unique solution of the integral equation \eqref{2ndintegralform}, and $\tilde{u}\in C^\alpha\left([0, T_1); H^2(\Omega)\cap H_0^1(\Omega)\right)\cap C^{\alpha+1}\left([0, T_1); L^2(\Omega)\right)$ is a unique strict solution of the differential equation \eqref{2nd2nd-NLP}, due to the preceding results in \textbf{Equivalence}, Theorem \ref{2ndlinear-parabolic-equation},  Theorem 1.2 of \cite{LS1} and Theorem 4.5 of \cite{SE}.
	
	Recall that $[F(\tilde{u})](0)\in D_{\mathcal{P}^*}(\alpha, \infty)$,  $\tilde{u}_0\in \left\{H^2(\Omega)\cap H_0^1(\Omega)\right\}= D(\mathcal{P}^*)$, $F(\tilde{u})\in C^{\alpha}\left([0, T_1); L^2(\Omega)\right)$ with $L^2(\Omega)=\overline{D(\mathcal{P}^*)}$, as $\mathcal{P}^*\tilde{u}_0+[\mathcal{F}(\tilde{u})](0)=[F(\tilde{u})](0)=\tilde{u}'(0)$,  Theorem 1.2 of \cite{LS1} as well as Theorem \ref{2ndlinear-parabolic-equation} state in particular if $\tilde{u}'(t)\in D_{\mathcal{P}^*}(\alpha, \infty)$ for $t=0$, then the same is true for $t>0$, i.e.~$\tilde{u}'(t)\in D_{\mathcal{P}^*}(\alpha, \infty)$ for all $ t\in[0, T_1)$, provided \[\mathcal{F}(\tilde{u})=F(\tilde{u})-\mathcal{P}^*\tilde{u}\in C^\alpha\left([0, T_1); L^2(\Omega)\right).\]
	Meanwhile, $[F(\tilde{u})](t)\in D_{\mathcal{P}^*}(\alpha, \infty)$ because $\tilde{u}'(t)=[F(\tilde{u})](t)$ and $\tilde{u}'(t)\in D_{\mathcal{P}^*}(\alpha, \infty)$ for all $t\in[0, T_1)$.
	
	Because of Theorem \ref{2nd-mild-solution-cor} and $u=\tilde{u}+\theta_1$, the initial-boundary value problem  \eqref{2ndcp1} has a unique strict solution $(u, w)$ with the regularity
	\[u\in C^\alpha\left([0, T_1); H^2(\Omega)\right)\cap C^{\alpha+1}\left([0, T_1); L^2(\Omega)\right),\]
	\[{w}\in C\left([0, T_1); H^2(\Omega)\right)\cap C^{1}\left([0, T_1); H^1(\Omega)\right)\cap C^{2}\left([0, T_1); L^2(\Omega)\right).\]
	This concludes the proof of Theorem \ref{2ndcp-sys}.
\end{proof}

The well-posedness of the coupled system, Theorem \ref{2ndcoupled system}, is a direct consequence of Theorem \ref{2ndcp-sys}.

\appendix
\renewcommand{\theequation}{\thesection.\arabic{equation}}

\section{Auxiliary Estimates}\label{AppA}

In this appendix,  we formulate and prove some estimates which are repeatedly used in the proof of the main results. They follow from well-known  properties of the Sobolev spaces $H^k(\Omega)$ with $k>0$. In particular, we recall the following well-known algebra property, see \cite{TT} for a proof:
\begin{lem}\label{alg}
	$H^k(\Omega)$ is an algebra when $k>\frac{n}{2}$. In particular, $H^1(\Omega)$ is an algebra if $\Omega\subset\mathds{R}$ and $H^2(\Omega)$ is an algebra if $\Omega\subset \mathds{R}^n$, $n=1,\ 2$.
\end{lem}
We formulate a direct consequence of Lemma \ref{alg}, which is used throughout this article. 
\begin{lem}\label{2ndalg-1}
	Let $\Omega\subset\mathds{R}$ be an open and bounded subset and $C=C(\Omega)$ be a positive constant. If $f_1\in H^1(\Omega)$ and $f_2\in L^2(\Omega)$, then
	\be\label{2ndalg-1-1}
	\left\|f_1f_2\right\|_{L^2(\Omega)}\leq C\left\|f_1\right\|_{H^1(\Omega)}\left\|f_2\right\|_{L^2(\Omega)}.
	\ee
	If $g_1$, $g_2\in H^1(\Omega)$, $g_3\in H^2(\Omega)$, then
	\be\label{2ndalg-1-2}
	\left\|g_1\frac{\partial}{\partial x}\left(g_2\frac{\partial g_3}{\partial x}\right)\right\|_{L^2(\Omega)}\leq C\left\|g_1\right\|_{H^1(\Omega)}\left\|g_2\right\|_{H^1(\Omega)}\left\|g_3\right\|_{H^2(\Omega)}.
	\ee
	If $g_1$, $g_2\in H^1(\Omega)$, $g_3$, $g_4\in H^2(\Omega)$, then
	\be\label{alg-1-3}
	\left\|g_1\frac{\partial}{\partial x}\left(g_2g_3\frac{\partial g_4}{\partial x}\right)\right\|_{L^2(\Omega)}\leq C\left\|g_1\right\|_{H^1(\Omega)}\left\|g_2\right\|_{H^1(\Omega)}\left\|g_3\right\|_{H^2(\Omega)}\left\|g_4\right\|_{H^2(\Omega)}.
	\ee
\end{lem}
\begin{proof}
	\begin{align}
		\left\|f_1f_2\right\|_{L^2(\Omega)}&\leq\left\|f_1\right\|_{L^\infty(\Omega)}\left\|f_2\right\|_{L^2(\Omega)}\leq C\left\|f_1\right\|_{H^1(\Omega)}\left\|f_2\right\|_{L^2(\Omega)}\notag.
	\end{align}
	\begin{align}
		\left\|g_1\frac{\partial}{\partial x}\left(g_2\frac{\partial g_3}{\partial x}\right)\right\|_{L^2(\Omega)}&\leq C\left\|g_1\right\|_{H^1(\Omega)}\left\|g_2\frac{\partial g_3}{\partial x}\right\|_{H^1(\Omega)}
		\leq C\left\|g_1\right\|_{H^1(\Omega)}\left\|g_2\right\|_{H^1(\Omega)}\left\|\frac{\partial g_3}{\partial x}\right\|_{H^1(\Omega)}\notag\\
		&\leq C\left\|g_1\right\|_{H^1(\Omega)}\left\|g_2\right\|_{H^1(\Omega)}\left\|g_3\right\|_{H^2(\Omega)}\notag.
	\end{align}
	\begin{align}
		\left\|g_1\frac{\partial}{\partial x}\left(g_2g_3\frac{\partial g_4}{\partial x}\right)\right\|_{L^2(\Omega)}\leq& C\left\|g_1\right\|_{H^1(\Omega)}\left\|g_2g_3\right\|_{H^1(\Omega)}\left\|g_4\right\|_{H^2(\Omega)}\notag\\
		\leq& C\left\|g_1\right\|_{H^1(\Omega)}\left\|g_2\right\|_{H^1(\Omega)}\left\|g_3\right\|_{H^2(\Omega)}\left\|g_4\right\|_{H^2(\Omega)}\notag.
	\end{align}
\end{proof}
We now formulate and prove three technical estimates used in the main body of this article, which repeatedly use these algebra properties.
\begin{lem}\label{2ndestimates}
	Let $T\in(0, \infty)$, $\kappa=  \inf_{x\in{\Omega}}w_0(x)>0$ and $w_0\in H^2(\Omega)$. Then there exists a constant $C=C(\Omega)>0$, such that for all
	\be\label{2ndr-def-I}
	r\in\left(0, \frac{\kappa}{2C}\right),
	\ee
	the function $w\in C\left([0, T]; B_{H^1}\left(w_0, r\right)\right)$ has the lower bound:
	\be\label{2ndw-lower-bound}
	w(t)\geq\frac{\kappa}{2},\quad \forall\ t\in[0, T].
	\ee
	Moreover, for all $w_1$, $w_2\in C\left([0, T]; B_{H^1}(w_0, r)\right)$, there exist  positive constants $C_k$, $k=1,\ 2,\ 3$, which depend on $\Omega$, $\kappa$ and $\left\|w_0\right\|_{H^1(\Omega)}$, such that
	\be\label{2ndC-a}
	\sup_{t\in[0, T]}\left\|\frac{1}{[w_1(t)]^k}\right\|_{H^1(\Omega)}\leq {C_1^k},\quad  k=1,\ 2,\ 3.
	\ee
	\be\label{2ndC-d}
	\sup_{t\in[0, T]}\left\|\frac{1}{[w_1(t)]^k}-\frac{1}{[w_2(t)]^k}\right\|_{H^1(\Omega)}
	\leq C_k\sup_{t\in[0, T]}\left\|w_1(t)-w_2(t)\right\|_{H^1(\Omega)},\quad k=2,\ 3.
	\ee
\end{lem}
\begin{proof}
We first prove assertion \eqref{2ndw-lower-bound}.

Since $ w\in C\left([0, T]; B_{H^2}\left({w}_0, r\right)\right)$,  then $\|w(t)-w_0\|_{H^{1}(\Omega)}\leq r$ holds for all $t\in [0, T]$.

The triangle inequality and the Sobolev embedding theorem imply there exists a positive constant $C=C(\Omega)$, such that for all $r\in\left(0, \frac{\kappa}{2C}\right)$, it follows that
\bse\label{2ndw-lower-bound1}
\be\label{2ndw-lower-bound12}
w(t)=w_0+w(t)-w_0\geq \kappa-\left\|w(t)-w_0\right\|_{L^{\infty}(\Omega)}\geq \kappa-C\left\|w(t)-w_0\right\|_{H^{1}(\Omega)}\geq\kappa-Cr\geq \frac{\kappa}{2},
\ee
\be\label{2ndw-lower-bound11}
\left\|w(t)\right\|_{H^1(\Omega)}\leq\tilde{C},
\ee 
\ese
for all $t\in [0, T]$. Here, $\tilde{C}=\frac{\kappa}{2C}+\left\|w_0\right\|_{H^1(\Omega)}$.  Hence we have proved assertion \eqref{2ndw-lower-bound}.

Since \eqref{2ndw-lower-bound1} and $r\in\left(0, \frac{\kappa}{2C}\right)$,  we find
\begin{equation*}\begin{split}
\sup_{t\in[0, T]}\left\|\frac{1}{w(t)}\right\|^2_{H^1(\Omega)}=& \sup_{t\in[0, T]}\int_\Omega\left|\frac{1}{w(t)}\right|^2+\left| \frac{\partial}{\partial x}\left[\frac{1}{w(t)}\right]\right|^2dx
\leq\frac{4C}{\kappa^2}+\sup_{t\in[0, T]}\int_\Omega\frac{1}{\left|w(t)\right|^4}\left| \frac{\partial  w(t)}{\partial x} \right|^2dx\\
\leq&\frac{4C}{\kappa^2}+\frac{16}{\kappa^4}\sup_{t\in[0, T]}\int_\Omega\left|\frac{\partial  w(t)}{\partial x}\right|^2dx
\leq\frac{4C}{\kappa^2}+\frac{16}{\kappa^4}\sup_{t\in[0, T]}\left\|w(t)\right\|_{H^1(\Omega)}^2
\leq\frac{4C}{\kappa^2}+\frac{16}{\kappa^4}\tilde{C}^2.
\end{split}\end{equation*}
We set $C_1^2=\frac{4C}{\kappa^2}+\frac{16}{\kappa^4}\tilde{C}^2$ such that
$C_1$ is a positive constant depending on $\Omega$, $\kappa$,  and $\left\|w_0\right\|_{H^1(\Omega)}$. Using the algebra property of $H^1(\Omega)$, see Lemma \ref{alg}, the assertion \eqref{2ndC-a} holds for $t\in [0, T]$.
We can now proceed to show  \eqref{2ndC-d}. For  $w_1$, $w_2\in C\left([0, T]; B_{H^2}(\tilde{w}_0, r)\right)$, \eqref{2ndC-a}, the algebraic property of $H^1(\Omega)$ from Lemma \ref{alg} and triangle inequality imply
\[\sup_{t\in[0, T]}\left\| \frac{\left[w_1(t)+w_2(t)\right]}{[w_1(t)]^{2}[w_2(t)]^{2}}\right\|_{H^1(\Omega)}\leq2 C^3_1,\
\sup_{t\in[0, T]}\left\|\frac{[w_1(t)]^2+[w_2(t)]^2+w_1(t) w_2(t)}{[w_1(t)]^3 [w_2(t)]^3}\right\|_{H^1(\Omega)}\leq3C_1^4.\]
Hence, setting $C_2=2 C^3_1$ and $C_3=3C_1^4$ implies \eqref{2ndC-d}.

This concludes the proof of Lemma \ref{2ndestimates}.
\end{proof}

\begin{lem}\label{2ndLip-G-Lem}
	Let $\tilde{w}_0=w_0-\theta_2\in H^2(\Omega)\cap H_0^1(\Omega)$ and $r\in\left(0, \frac{\kappa}{2C}\right)$.
	Then the nonlinear operator $G$, 
	\bse\label{2ndG-def}
	\be\label{2ndG-def1}
	G:  C\left([0, T]; B_{H^1}\left(\tilde{w}_0, r\right)\right)\longrightarrow  C([0, T]; H^1(\Omega)),\quad \tilde{w}\longmapsto G(\tilde{w})
	\ee
	\be\label{2ndG-def2}
	[G(\tilde{w})](t)=G(\tilde{w}(t))=-\frac{\beta_F}{(\tilde{w}(t)+\theta_2)^2}+\beta_p(\theta_1-1),
	\ee
	\ese
	satisfies
	\be\label{2ndHoldercontinuous}
	\sup_{0\leq t<t+h\leq T}\left\|[G(\tilde{w})](t+h)-[G(\tilde{w})](t)\right\|_{H^{1}(\Omega)}\leq L_G\sup_{0\leq t<t+h\leq T}\left\|\tilde{w}(t+h)-\tilde{w}(t)\right\|_{H^{1}(\Omega)},
	\ee
	for all $h\in(0, T]$, $\tilde{w}\in C\left([0, T]; B_{H^1}(\tilde{w}_0, r)\right)$. Here $L_G=L_G\left(\Omega,\ \kappa,\ \|w_0\|_{H^{1}(\Omega)},\ \beta_F\right)$ is a constant.	
	Moreover, $G$ is locally Lipschitz continuous with respect to  $\tilde{w}\in C\left([0, T]; B_{H^1}\left(\tilde{w}_0, r\right)\right)$, i.e.
	\be\label{2ndLip-G}
	\sup_{t\in[0, T]}\left\|[G(\tilde{w}_1)](t)-[G(\tilde{w}_2)](t)\right\|_{H^{1}(\Omega)}\leq L_G\sup_{t\in[0, T]}\left\|\tilde{w}_1(t)-\tilde{w}_2(t)\right\|_{H^{1}(\Omega)}
	\ee
	 for all $\tilde{w}_1, \tilde{w}_2\in C\left([0, T]; B_{H^1}\left(\tilde{w}_0,  r\right)\right)$. In particular,
	\be\label{2ndLip-G-1}
	\sup_{t\in[0, T]}\left\|[G(\tilde{w}_1)](t)-G(\tilde{w}_0)\right\|_{H^{1}(\Omega)}\leq L_G r.
	\ee
	Furthermore, the Fr\'{e}chet derivative $G'(\tilde{w})$ of $G(\tilde{w})$ on  $\tilde{w}\in C\left([0, T]; B_{H^1}\left(\tilde{w}_0, r\right)\right)$, defined by
	\bse\label{2ndFre-G-def-w}
	\be\label{2ndFre-G-def-w1}
	G'\left(\tilde{w}\right):\  C\left([0, T]; H_0^1(\Omega)\right)\longrightarrow C\left([0, T]; H_0^1(\Omega)\right),\quad q \longmapsto  G'\left(\tilde{w}\right)q
	\ee
	\be\label{2ndFre-G-def-w2}
	[G'\left(\tilde{w}\right)q](t)=[G'(\tilde{w}(t))]q(t)=\frac{2\beta_F}{\left(\tilde{w}(t)+\theta_2\right)^3}q(t),
	\ee
	\ese
	satisfies
	\be\label{2ndLip-G-2}
	\sup_{t\in[0, T]}\left\|\left[G'\left(\tilde{w}\right)q\right](t)\right\|_{H^1(\Omega)}\leq L_G\sup_{t\in[0, T]}\left\|q(t)\right\|_{H^1(\Omega)},
	\ee
	and $G'(\tilde{w}(t)): H_0^1(\Omega)\longrightarrow H_0^1(\Omega)$ satisfies
	\be\label{2nduniformly-continuous-Fre-G}
	\lim_{h\rightarrow0}\sup_{\begin{smallmatrix}0\leq t\leq t+ h\leq T\\ 0\leq\tau\leq 1\end{smallmatrix}}\left\|G'(\tilde{w}(t)+\tau[\tilde{w}(t+ h)-\tilde{w}(t)])-G'(\tilde{w}(t))\right\|_{\mathcal{B}\left(H_0^1(\Omega)\right)}=0.
	\ee
\end{lem}
\begin{proof}
We recall that
$r\in\left(0, \frac{\kappa}{2C}\right)$, $\kappa=\inf_{x\in{\Omega}}w_0>0$, $w_0=\tilde{w}_0+\theta_2$,  $\tilde{w}\in C\left([0, T]; B_{H^1}(\tilde{w}_0, r)\right)$, $w=\tilde{w}+\theta_2$.  For small $h\in(0, T)$ such that $t+h\in(0, T]$, $\left\|\tilde{w}(t+h)-\tilde{w}(t)\right\|_{H^1(\Omega)}\leq r$, Thus, \eqref{2ndHoldercontinuous} and \eqref{2ndLip-G} of  Lemma \ref{2ndLip-G-Lem} are valid by using \eqref{2ndC-a}, \eqref{2ndC-d}, respectively, in    Lemma \ref{2ndestimates} and setting $L_G=\beta_FC_2$.

In particular, for $\tilde{w}_1\in C\left([0, T]; B_{H^1}\left(\tilde{w}_0, r\right)\right)$, set $\tilde{w}_2(t)=\tilde{w}_0$, $ t\in[0, T]$,  then $[G(\tilde{w}_2)](t)=G(\tilde{w}_0)$. Hence \eqref{2ndLip-G-1} of   Lemma \ref{2ndLip-G-Lem} is valid because of  \eqref{2ndLip-G} in   Lemma \ref{2ndLip-G-Lem}.

We consider  $\tilde{w}_2=\tilde{w}\in C\left([0, T]; B_{H^1}\left(\tilde{w}_0, r\right)\right)$ and  choose small $\lambda\in\mathds{R}$ such that $\tilde{w}_1=\tilde{w}+\lambda q$ is in  $C\left([0, T]; B_{H^1}\left(\tilde{w}_0, r\right)\right) $ for all $q\in C\left([0, T]; H_0^1(\Omega)\right)$.
Then the Fr\'{e}chet derivative $G'\left(\tilde{w}\right)q$ of $G$ with respect to $\tilde{w}\in C\left([0, T]; B_{H^1}\left(\tilde{w}_0, r\right)\right) $ exists as a linear operator given by
\[G'(\tilde{w}): C\left([0, T]; H_0^1(\Omega)\right)\to C\left([0, T]; H_0^1(\Omega)\right),\ G'\left(\tilde{w}\right)q= \lim_{\lambda\rightarrow0}\frac{1}{\lambda}\left[G\left(\tilde{w}+\lambda q\right)-G\left(\tilde{w}\right)\right]= \frac{2\beta_F}{\left(\tilde{w}+\theta_2\right)^3}q,\]
\[[G'\left(\tilde{w}\right)q](t)=[G'\left(\tilde{w}(t)\right)]q(t)= \frac{2\beta_F}{\left(\tilde{w}(t)+\theta_2\right)^3}q(t).\]
According to  \eqref{2ndLip-G} of  Lemma \ref{2ndLip-G-Lem}, the inequality \eqref{2ndLip-G-2} of  Lemma \ref{2ndLip-G-Lem} holds by the following computation:
\begin{align}
\sup_{t\in[0, T]}\left\|\left[G'\left(\tilde{w}\right)q\right](t)\right\|_{H^1(\Omega)}=&\sup_{t\in[0, T]}\left\| \lim_{\lambda\rightarrow0}\frac{[G\left(\tilde{w}+\lambda q\right)](t)-[G\left(\tilde{w}\right)](t)}{\lambda}\right\|_{H^1(\Omega)}\notag\\
=& \lim_{\lambda\rightarrow0}\frac{1}{\lambda}\sup_{t\in[0, T]}\left\|[G\left(\tilde{w}_1\right)](t)-[G\left(\tilde{w}_2\right)](t)\right\|_{H^1(\Omega)}\notag\\
\leq& \lim_{\lambda\rightarrow0}\frac{1}{\lambda}L_G\sup_{t\in[0, T]}\left\| \tilde{w}_1(t)-\tilde{w}_2(t)\right\|_{H^{1}(\Omega)}
=L_G\sup_{t\in[0, T]}\left\|q(t)\right\|_{H^{1}(\Omega)}\notag.
\end{align}
For all $t\in[0, T]$, we choose small $h\in(0, T)$ and $\tau\in[0, 1]$ such that $t+h\in(0, T]$, then we obtain
\[\left\|\tilde{w}(t)+\tau\left[\tilde{w}(t+h)-\tilde{w}(t)\right]-\tilde{w}_0\right\|_{H^1(\Omega)}\leq r,\]
then for $\psi\in H_0^1(\Omega)$ with $\|\psi\|_{H^1(\Omega)}\leq1$, we find $G'(\tilde{w}(t)): H_0^1(\Omega)\longrightarrow H_0^1(\Omega)$ with $\psi\longmapsto [G(\tilde{w}(t))]\psi $. By the algebraic properties of $H^1(\Omega)$, together with \eqref{2ndC-a} and \eqref{2ndC-d} in    Lemma \ref{2ndestimates}, we have
\begin{align}
&\left\|G'\left(\tilde{w}(t)+\tau\left[\tilde{w}(t+h)-\tilde{w}(t)\right]\right)-G'(\tilde{w}(t))\right\|_{\mathcal{B}\left(H_0^1(\Omega)\right)}\notag\\
=&\left\|[G'\left(\tilde{w}(t)+\tau\left[\tilde{w}(t+h)-\tilde{w}(t)\right]\right)]\psi-[G'(\tilde{w}(t))]\psi\right\|_{H^1(\Omega)}\notag\\
\leq&2\beta_F\sup_{\begin{smallmatrix}0\leq t<t+h\leq T\\ 0\leq \tau\leq 1\end{smallmatrix}}\left\| \frac{1}{(\tilde{w}(t)+\tau\left[\tilde{w}(t+h)-\tilde{w}(t)\right]+\theta_2)^3}-\frac{1}{[\tilde{w}(t)+\theta_2]^3}\right\|_{H^1(\Omega)}\|\psi\|_{H^1(\Omega)}\notag\\
\leq&2\beta_FC_3\sup_{0\leq t<t+h\leq T}\|\tilde{w}(t+h)-\tilde{w}(t)\|_{H^1(\Omega)}\label{2nd002}.
\end{align}
Since $\tilde{w}\in C([0, T]; H_0^1(\Omega))$, $\tilde{w}$ is uniformly continuous with respect to $t\in[0, T]$, hence the assertion \eqref{2nduniformly-continuous-Fre-G} of  Lemma \ref{2ndLip-G-Lem} is proved by
\[\lim_{h\rightarrow 0^+}\sup_{0\leq t<t+h\leq T}\|\tilde{w}(t+\tau h)-\tilde{w}(t)\|_{H^1(\Omega)}=0.\]
This concludes the proof of  Lemma \ref{2ndLip-G-Lem}.
\end{proof}

\begin{lem}\label{2ndLip-nonlinearity}
	Let $u_0\in H^{2+\sigma}(\Omega)$ with $\sigma\in (0, 1/2)$
	and assume the operators $u\longmapsto v(u)$, $u\longmapsto w(u)$ given by, respectively,
	\[u\longmapsto v(u): C\left([0, T]; B_{H^2}\left(u_0, r\right)\right)\longrightarrow C\left([0, T]; B_{L^2}\left(v_0, r\right)\right),\]
	\[u\longmapsto w(u): C\left([0, T]; B_{H^2}\left(u_0, r\right)\right)\longrightarrow C\left([0, T]; B_{H^1}\left(w_0, r\right)\right),\]
	satisfy, for all $u_1$, $u_2\in  C\left([0, T]; B_{H^2}\left(u_0, r\right)\right)$,
	\[
	\sup_{t\in[0, T]}\|[v(u_1)](t)-[v(u_2)](t)\|_{L^2(\Omega)}\leq L_W \sup_{t\in[0, T]}\|{u}_1(t)-{u}_2(t)\|_{H^2(\Omega)},
	\]
	\[
	\sup_{t\in[0, T]}\|[w(u_1)](t)-[w(u_2)](t)\|_{H^1(\Omega)}\leq L_W \sup_{t\in[0, T]}\|{u}_1(t)-{u}_2(t)\|_{H^2(\Omega)}.
	\]
	Then $f(u)$, defined by
	\[u\longmapsto f(u):\ C([0, T]; H^2(\Omega))\longrightarrow C([0, T]; L^2(\Omega)),\quad  f(u)=\frac{1}{w(u)}\frac{\partial}{\partial x}\left([w(u)]^3u\frac{\partial u}{\partial x}\right)-\frac{v(u)}{w(u)}u,\]
	is Lipschitz continuous in $u$,
	\begin{align}
		\sup_{t\in[0, T]}\|[f(u_1)](t)-[f(u_2)](t)\|_{L^2(\Omega)}\leq L_e \sup_{t\in[0, T]}\|{u}_1(t)-{u}_2(t)\|_{H^2(\Omega)}.\label{2ndnonlinearity-Lip-0}
	\end{align}
	Here $L_W$ and $L_e$ are  Lipschitz constants, and $L_e$ depends on $L_W$, $\Omega$, $\kappa$, $\left\|u_0\right\|_{H^2(\Omega)}$, $\left\|v_0\right\|_{L^2(\Omega)}$ and $\left\|w_0\right\|_{H^1(\Omega)}$.
\end{lem}
\begin{proof}
	Let $u_1,\ u_2\in C\left([0, T]; B_{H^2}(u_0, r)\right)$,  the definitions of the operators $u\longmapsto v(u)$ and $u\longmapsto w(u)$ imply that $(v_1, w_1)=(v(u_1), w(u_1))$, $(v_2, w_2)=(v(u_2), w(u_2))\in C\left([0, T]; B_{L^2}(v_0, r)\times B_{H^1}(w_0, r)\right)$.
	Set
	$\tilde{C}=\left\|w_0\right\|_{H^1(\Omega)}+\frac{\kappa}{2C}$, $\tilde{C_1}=\left\|u_0\right\|_{H^2(\Omega)}+\frac{\kappa}{2C}$, $\tilde{C_2}=\left\|v_0\right\|_{L^2(\Omega)}+\frac{\kappa}{2C}$.
	Thus for all  $ t\in[0, T]$, we get
	$\left\|u_1(t)\right\|_{H^2(\Omega)}\leq\tilde{C_1}$, $\left\|u_2(t)\right\|_{H^2(\Omega)}\leq\tilde{C_1}$, $\left\|v_1(t)\right\|_{L^2(\Omega)}\leq\tilde{C_2}$, $\left\|v_2(t)\right\|_{L^2(\Omega)}\leq\tilde{C_2}$,
	$\left\|w_1(t)\right\|_{H^1(\Omega)}\leq\tilde{C}$, $\left\|w_2(t)\right\|_{H^1(\Omega)}\leq\tilde{C}$.
	
	Because of estimate $\left\|w_1(t)-w_2(t)\right\|_{H^1(\Omega)}\leq L_W\left\|u_1(t)-u_2(t)\right\|_{H^2(\Omega)}$ for all $t\in[0, T]$ and $H^1(\Omega)$ is an algebra for $\Omega\subset\mathds{R}$,  \eqref{2ndC-a} and \eqref{2ndC-d} in    Lemma \ref{2ndestimates}, we obtain similar bounds for the functions $[w_1(t)]^{-1}-[w_2(t)]^{-1}$ and $[w_1(t)]^3-[w_2(t)]^3$:
	\be\label{2ndAppB01}
	\left\|[w_1(t)]^{-1}-[w_2(t)]^{-1}\right\|_{H^1(\Omega)}\leq{C_1^2L_W}\left\|u_1(t)-u_2(t)\right\|_{H^2(\Omega)},
	\ee
	\be\label{2ndAppB02}
	\left\|[w_1(t)]^3-[w_2(t)]^3\right\|_{H^1(\Omega)}\leq3\tilde{C}^2L_W\left\|u_1(t)-u_2(t)\right\|_{H^2(\Omega)}.
	\ee
	Similarly, the algebraic property of $H^2(\Omega)$, i.e. Lemma \ref{alg}, implies
	\begin{align}
		\left\|[u_1(t)]^2-[u_2(t)]^2\right\|_{H^2(\Omega)}\leq&2\tilde{C}_1\left\|u_1(t)-u_2(t)\right\|_{H^2(\Omega)} . \label{2ndAppB1}
	\end{align}
	The algebraic properties of $H^1(\Omega)$, i.e. Lemma \ref{alg}, \eqref{2ndalg-1-2} of Lemma \ref{2ndalg-1} imply
	\begin{align}
		& \frac{1}{2}\left\|\left[\frac{1}{w_1(t)}-\frac{1}{w_2(t)}\right]\frac{\partial}{\partial x}\left\{[w_1(t)]^3\frac{\partial[u_1(t)]^2}{\partial x} \right\}\right\|_{L^2(\Omega)}\notag\\
		\leq&C \left\|[w_1(t)]^{-1}-[w_2(t)]^{-1}\right\|_{H^2(\Omega)}\left\|[w_1(t)]^3\right\|_{H^2(\Omega)}\left\|[u_1(t)]^2\right\|_{H^2(\Omega)}\notag\\
		\leq&  CC_1^2\tilde{C}^3\tilde{C}_1^2L_W\left\|u_1(t)-u_2(t)\right\|_{H^2(\Omega)}\notag.
	\end{align}
	Similarly, setting $\widehat{C}_1=3CC_1(\tilde{C}\tilde{C}_1)^2L_W$ and $\widehat{C}_2=2CC_1\tilde{C}^3\tilde{C}_1$, it follows that 
	\[ \frac{1}{2}\left\|\frac{1}{w_2(t)}\frac{\partial}{\partial x}\left\{\left([w_1(t)]^3-[w_2(t)]^3\right)\frac{\partial[u_1(t)]^2}{\partial x}\right\}\right\|_{L^2(\Omega)}\leq\widehat{C}_1\left\|u_1(t)-u_2(t)\right\|_{H^2(\Omega)},\]
	\[\frac{1}{2}\left\|\frac{1}{w_2(t)}\frac{\partial}{\partial x}\left\{[w_2(t)]^3\frac{\partial}{\partial x}\left([u_1(t)]^2-[u_2(t)]^2\right)\right\}\right\|_{L^2(\Omega)}
	\leq\widehat{C}_2\left\|u_1(t)-u_2(t)\right\|_{H^2(\Omega)}.\]
	Because of \eqref{2ndAppB01}, $\left\|v_1(t)-v_2(t)\right\|_{H^1(\Omega)}\leq L_W\left\|u_1(t)-u_2(t)\right\|_{H^2(\Omega)}$, $H^1(\Omega)$ being an algebra and \eqref{2ndalg-1-1} in Lemma \ref{2ndalg-1}, setting $\widehat{C}_3=C\tilde{C}_2C_1+C\tilde{C}_1C_1L_W+C\tilde{C}_1\tilde{C}_2C_1^2L_W$, we obtain 
	\begin{align}
		\left\| \frac{v_1(t)}{w_1(t)}u_1(t)- \frac{v_2(t)}{w_2(t)}u_2(t)\right\|_{L^2(\Omega)}
		\leq&\left\|v_1(t)\right\|_{L^2(\Omega)}\left\|[w_1(t)]^{-1}\right\|_{H^1(\Omega)}\left\|u_1(t)-u_2(t)\right\|_{H^2(\Omega)}\notag\\
		+&\left\|u_2(t)\right\|_{H^2(\Omega)}\left\|[w_1(t)]^{-1}\right\|_{H^1(\Omega)}\left\|v_1(t)-v_2(t)\right\|_{L^2(\Omega)}\notag\\
		+&\left\|u_2(t)\right\|_{H^2(\Omega)}\left\|v_2(t)\right\|_{L^2(\Omega)}\left\|[w_1(t)]^{-1}-[w_2(t)]^{-1}\right\|_{H^1(\Omega)}\notag\\
		\leq&\widehat{C}_3\left\|u_1(t)-u_2(t)\right\|_{H^2(\Omega)}\notag.
	\end{align}
	Consequently, \eqref{2ndnonlinearity-Lip-0} holds by setting $L_e=CC_1^2\tilde{C}^3\tilde{C}_1^2L_W+\widehat{C}_1+\widehat{C}_2+\widehat{C}_3$ and computing
	
	\begin{align}
		\|[f(u_1)](t)-[f(u_2)](t)\|_{L^2(\Omega)}\leq& \frac{1}{2}\left\|\left[\frac{1}{w_1(t)}-\frac{1}{w_2(t)}\right]\frac{\partial}{\partial x}\left\{[w_1(t)]^3\frac{\partial[u_1(t)]^2}{\partial x}\right\}\right\|_{L^2(\Omega)}\notag\\
		+& \frac{1}{2}\left\|\frac{1}{w_2(t)}\frac{\partial}{\partial x}\left\{\left([w_1(t)]^3-[w_2(t)]^3\right)\frac{\partial[u_1(t)]^2}{\partial x}\right\}\right\|_{L^2(\Omega)}\notag\\
		+& \frac{1}{2}\left\|\frac{1}{w_2(t)}\frac{\partial}{\partial x}\left\{[w_2(t)]^3\frac{\partial}{\partial x}\left([u_1(t)]^2-[u_2(t)]^2\right)\right\}\right\|_{L^2(\Omega)}\notag\\
		+&\left\| \frac{v_1(t)}{w_1(t)}u_1(t)-\frac{v_2(t)}{w_2(t)}u_2(t)\right\|_{L^2(\Omega)}
		\leq L_e\left\|u_1(t)-u_2(t)\right\|_{H^2(\Omega)}.\notag
	\end{align}
	This concludes the proof of    Lemma  \ref{2ndLip-nonlinearity}.
\end{proof}

\section{Proofs of Results in Section \ref{2nd-order problem}}\label{AppB-1}
\subsection{Proof of Lemma \ref{2equivalence-1}}
\begin{proof}
	The proof follows closely the proof of Lemma 6.10 in the reference \cite{schnaubelt}. We here describe
	the necessary adaptations.
		
	Let
	$\tilde{w}\in C^2([0, T]; L^2(\Omega))\cap C^1([0, T]; H_0^1(\Omega))\cap C([0, T]; H^2(\Omega)\cap H_0^1(\Omega))$
	solve the semilinear hyperbolic equation \eqref{2nd-LWE-1}. Then
	$\Phi:=\left(\tilde{w}', \tilde{w}\right)\in C^1([0, T]; \mathscr{X})$, $\Phi(t)=\left(\tilde{w}'(t), \tilde{w}(t)\right)\in D(\mathcal{A})$ for all $ t\in [0, T]$, $\Phi\in C([0, T]; D(\mathcal{A}))$ and
	\begin{align}
		\Phi'(t)=\begin{pmatrix}\tilde{w}^{''}(t)\\ \tilde{w}'(t)\end{pmatrix}
		=\begin{pmatrix}\Delta_D\tilde{w}(t)- \frac{\beta_F}{(\tilde{w}(t)+\theta_2)^2}+\beta_p(\tilde{u}(t)+\theta_1-1)\\ \tilde{w}'(t)\end{pmatrix}
		=\mathcal{A}\Phi(t)+[\mathcal{G}(\Phi)](t)\notag
	\end{align}
	 for all $t\in [0, T]$. Moreover, $\Phi(0)=\left(\tilde{w}'(0), \tilde{w}(0)\right)=\left(\tilde{v}_0, \tilde{w}_0\right)=\Phi_0.$
	
	Therefore, $\Phi\in C([0, T]; D(\mathcal{A}))\cap C^1([0, T]; \mathscr{X})$ solves the equation \eqref{2nd-IEE}.
	
	Conversely, let $\Phi=\left(\varphi_1, \varphi_2\right)\in C^1([0, T]; \mathscr{X})\cap C([0, T]; D(\mathcal{A}))$ solve the semilinear  evolution equation \eqref{2nd-IEE}. We set $\tilde{w}:=\varphi_2$, obtaining $\tilde{w}\in C^1([0, T]; H_0^1(\Omega)),\ \tilde{w}(t)\in H^2(\Omega)\cap H_0^1(\Omega)$, $\forall\ t\in [0, T]$, and $\tilde{w}\in C([0, T]; H^2(\Omega)\cap H_0^1(\Omega))$. It further follows, 
	\begin{align}
		\begin{pmatrix}\varphi'_1(t)\\ \tilde{w}'(t)\end{pmatrix}=\mathcal{A}\begin{pmatrix}\varphi_1(t)\\ \tilde{w}(t)\end{pmatrix}+[\mathcal{G}(\Phi)](t)
		=\begin{pmatrix}\Delta_D\tilde{w}(t)- \frac{\beta_F}{(\tilde{w}(t)+\theta_2)^2}+\beta_p(\tilde{u}(t)+\theta_1-1)\\ \varphi_1(t)\end{pmatrix},\quad \forall\ t\in [0, T].\notag
	\end{align}
	As a result, $\tilde{w}'=\varphi_1\in C^1([0, T]; L^2(\Omega))$ and $\Phi=\left(\tilde{w}', \tilde{w}\right)$, so that $\tilde{w}\in C^2([0, T]; L^2(\Omega))$ as well as
	 $\left(\tilde{w}'(0), \tilde{w}(0)\right)=\left(\tilde{v}_0, \tilde{w}_0\right)$. So $\tilde{w}\in C^2([0, T]; L^2(\Omega))\cap C^1([0, T]; H_0^1(\Omega))\cap C([0, T]; H^2(\Omega)\cap H_0^1(\Omega))$ solves the semilinear hyperbolic equation \eqref{2nd-LWE-1}.
	
	This equivalence also yields that  solutions to the semilinear evolution equation \eqref{2nd-IEE} are unique if and only if  solutions to the semilinear equation \eqref{2nd-LWE-1} are unique.
\end{proof}

\subsection{Proof of Lemma \ref{2ndgenerator}}
\begin{proof}
	We aim to show that $\mathcal{A}$, defined by \eqref{2A-1}, is skew adjoint on the Hilbert space $\mathscr{X}$ defined by \eqref{2state-space}, and thus generates a strongly continuous semigroup ($C_0$-semigroup) on $\mathscr{X}$ by using Stone's Lemma (see  Theorem 3.24, Section 3, Chapter II, \cite{EK}).
	
	The following proof is analogous to the discussion from Section 5 (Example 5.13) of  reference \cite{schnaubelt}.
	
	From the definition \eqref{2A-1} of $\mathcal{A}$, $\mathcal{A}$ is densely defined in $\mathscr{X}$, i.e. $\overline{D(\mathcal{A})}=\mathscr{X}$,  then $\mathcal{A}$ is skew symmetric (i.e. $\mathbf{i}\mathcal{A}$ is symmetric) for any two $(\phi_1, \phi_2)\in D(\mathcal{A})$ and $(\psi_1, \psi_2 )\in D(\mathcal{A})$ by the following computations:
	\begin{align}
		\left\langle\mathcal{A}\begin{pmatrix}\phi_1\\ \phi_2\end{pmatrix}, \begin{pmatrix}\psi_1\\ \psi_2 \end{pmatrix}\right\rangle_{\mathscr{X}}&=\left\langle\begin{pmatrix}\Delta_D\phi_2\\ \phi_1 \end{pmatrix}, \begin{pmatrix}\psi_1\\ \psi_2 \end{pmatrix}\right\rangle_{\mathscr{X}}
		= \int_\Omega(\Delta_D\phi_2)\cdot\overline{\psi_1}+\nabla\phi_1\cdot\nabla\overline{\psi_2}dx\notag\\
		&= \int_\Omega -\nabla\phi_2\cdot\nabla\overline{\psi_1}+\nabla\phi_1\cdot\nabla\overline{\psi_2}dx
		=- \left(\int_\Omega \nabla\phi_2\cdot\nabla\overline{\psi_1}-\nabla\phi_1\cdot\nabla\overline{\psi_2}dx\right)\notag\\
		&=- \left(\int_\Omega \nabla\phi_2\cdot\nabla\overline{\psi_1}+\phi_1\cdot(\overline{\Delta_D\psi_2})dx\right)
		=-\left\langle\begin{pmatrix}\phi_1\\ \phi_2 \end{pmatrix}, \begin{pmatrix}\Delta_D\psi_2\\ \psi_1 \end{pmatrix}\right\rangle_{\mathscr{X}}\notag\\
		&=-\left\langle\begin{pmatrix}\phi_1\\ \phi_2 \end{pmatrix}, \mathcal{A}\begin{pmatrix}\psi_1\\ \psi_2 \end{pmatrix}\right\rangle_{\mathscr{X}} . \notag
	\end{align}
	Furthermore, $\text{Re}\left\langle\mathcal{A}\begin{pmatrix}\psi\\ \phi\end{pmatrix}, \begin{pmatrix}\psi\\ \phi \end{pmatrix}\right\rangle_{\mathscr{X}}=0$ for all $(\psi, \phi )\in D(\mathcal{A})$, so $\mathcal{A}$ is dissipative. By using the Lax-Milgram Theorem (Theorem 1, Section 6.2, \cite{EL}), we have the inverse $\Delta_D^{-1}$ of $\Delta_D$ exists, thus we define an operator
	\[\mathcal{R}=\begin{pmatrix}0 &1\\ \Delta_D^{-1} &0\end{pmatrix}.\]
	Then
	\[\mathcal{R}\mathscr{X}\subset D(\mathcal{A}),\ \mathcal{A}\mathcal{R}=I, \ \mathcal{R}\mathcal{A}\begin{pmatrix}\psi\\ \phi \end{pmatrix}=\begin{pmatrix}\psi\\ \phi \end{pmatrix},\ \forall\ \left(\psi, \phi \right)\in D(\mathcal{A}).\]
	Therefore, $\mathbf{i}\mathcal{A}$ is invertible and the resolvent set $\rho(\mathbf{i}\mathcal{A})$ of $\mathbf{i}\mathcal{A}$ satisfies $\rho(\mathbf{i}\mathcal{A})\cap \mathds{R}\neq\emptyset$, so the spectrum $\sigma(\mathbf{i}\mathcal{A})\subseteq\mathds{R}$, consequently, $\mathbf{i}\mathcal{A}$ is selfadjoint, as a result, $\mathcal{A}$ is skew adjoint. According to Stone's Lemma, we have the linear operator $\mathcal{A}$ generates a $C_0$-semigroup
	$\left\{T(t)\in \mathcal{B}(\mathscr{X}):\ t\in[0, \infty)\right\}$.
\end{proof}

\subsection{Proof of Theorem \ref{2nd-solu-thm}}
\begin{proof}
	We let $T\in(0, \infty)$ be taken to be specified below and define a complete metric space  $Z(T)$ for the metric induced by the norm $ \sup_{t\in[0, T]}\|(\tilde{v}(t),\tilde{w}(t))\|_{L^2(\Omega)\times H^{1}(\Omega)}$ as follows:
	\begin{align*}
		Z(T):=\left\{\hspace*{-0.1cm}\begin{pmatrix}\tilde{v}\\ \tilde{w}\end{pmatrix}\hspace*{-0.1cm}\in C([0,T]; L^2(\Omega)\times H_0^1(\Omega)):
		\begin{pmatrix}\tilde{v}(0)\\ \tilde{w}(0)\end{pmatrix}=\begin{pmatrix}\tilde{v}_0\\ \tilde{w}_0\end{pmatrix}, 
		\sup_{t\in[0, T]}\left\|\begin{pmatrix}\tilde{v}(t)-\tilde{v}_0\\ \tilde{w}(t)-\tilde{w}_0\end{pmatrix}\right\|_{L^2\times H^{1}(\Omega)}\hspace*{-0.1cm}\leq r\right\}.
	\end{align*}
	Because $\mathcal{A}$ generates a strongly continuous semigroup $\left\{T(t)\in\mathcal{B}\left(L^2(\Omega)\times H_0^1(\Omega)\right): t\in[0, \infty)\right\}$ and $[G\left(\tilde{w}\right)](t)=-\beta_F[\tilde{w}(t)+\theta_2]^{-2}+\beta_p(\theta_1-1)$, we introduce a nonlinear operator $\Phi$ on $Z(T)$ by
	\[\left[\Phi(\tilde{v},\tilde{w})\right](t)
	:=T(t)\begin{pmatrix}\tilde{v}_0\\ \tilde{w}_0\end{pmatrix}+
	 \int_0^t\left\{T(t-s)\begin{pmatrix}[G(\tilde{w})](s)+\beta_p\tilde{u}(s)\\ 0\end{pmatrix}\right\}ds,\quad \forall\ t\in [0,T].\]
	We observe that
	\[\left[\Phi(\tilde{v},\tilde{w})\right](0)=T(0)\begin{pmatrix}\tilde{v}_0\\ \tilde{w}_0\end{pmatrix}=\begin{pmatrix}\tilde{v}_0\\ \tilde{w}_0\end{pmatrix}\in D(\mathcal{A}).\]
	According to  Lemma 1.3 of Chapter II in reference \cite{EK},
	\[T(t)\begin{pmatrix}\tilde{v}_0\\ \tilde{w}_0\end{pmatrix}\in D(\mathcal{A}),\quad \forall\ t\in[0, T].\]
	Since $(\tilde{v},\tilde{w})\in Z(T)$, $\tilde{u}\in C\left([0, T]; B_{H^2}\left(\tilde{u}_0, r\right)\right)$ such that
	$G(\tilde{w})+\beta_p\tilde{u}\in C([0, T]; H^1(\Omega))$, hence
	\[\begin{pmatrix}[G(\tilde{w})](t)+\beta_p\tilde{u}(t)\\ 0\end{pmatrix},\quad \int_0^t\left\{T(t-s)\begin{pmatrix}[G(\tilde{w})](s)+\beta_p\tilde{u}(s)\\ 0\end{pmatrix}\right\}ds\in L^2(\Omega)\times H_0^1(\Omega).\]
	Therefore, $\Phi$ is a nonlinear operator which maps $Z(T)$ into $C\left([0,T]; L^2(\Omega)\times H_0^1(\Omega)\right)$:
	\[\Phi: \ Z(T)\longrightarrow\ C\left([0,T]; L^2(\Omega)\times H_0^1(\Omega)\right). \]
	We next show that there exists a unique mild solution $(\tilde{v},\tilde{w})\in Z(T)$ of the semilinear evolution equation \eqref{2nd-SWE-1} which is a fixed point of $\Phi$ on $Z(T)$.
	
	We write $M_0= \sup_{t\in [0, \infty )}\|T(t)\|_{\mathcal{B}\left(L^2(\Omega)\times H_0^1(\Omega)\right)}$ an operator norm of $\left\{T(t)\right\}_{0\leq t<\infty}$ on the space $L^2(\Omega)\times H_0^1(\Omega)$. For given $\tilde{u}\in C\left([0, T]; B_{H^2}(\tilde{u}_0, r)\right)$, if $(\tilde{v}_1,\tilde{w}_1),\ (\tilde{v}_2,\tilde{w}_2)\in Z(T)$, then 
	\[[G(\tilde{w}_1)](t)-[G(\tilde{w}_2)](t)\in H^1(\Omega), \quad \forall\ t\in[0, T].\] 
	By using the estimate \eqref{2ndLip-G} of    Lemma \ref{2ndestimates},  we obtain
	\begin{align}
		&\sup_{t\in[0, T]}\left\|[\Phi(\tilde{v}_1,\tilde{w}_1)](t)-[\Phi(\tilde{v}_2,\tilde{w}_2)](t)\right\|_{L^2(\Omega)\times H^{1}(\Omega)}\notag\\
		=&\sup_{t\in[0, T]}\left\|\int_0^tT(t-s)\begin{pmatrix}[G(\tilde{w}_1)](s)+\beta_p\tilde{u}(s)-[G(\tilde{w}_2)](s)-\beta_p\tilde{u}(s)\\ 0\end{pmatrix}ds\right\|_{L^{2}(\Omega)\times H^{1}(\Omega)}\notag\\
		\leq&TM_0\sup_{t\in[0, T]}\left\|[G(\tilde{w}_1)](t)-[G(\tilde{w}_2)](t)\right\|_{L^{2}(\Omega)}
		\leq TM_0\sup_{t\in[0, T]}\left\|[G(\tilde{w}_1)](t)-[G(\tilde{w}_2)](t)\right\|_{H^{1}(\Omega)}\notag\\
		\leq&TM_0L_G\sup_{t\in[0, T]}\left\| \tilde{w}_1(t)-\tilde{w}_2(t)\right\|_{H^{1}(\Omega)}
		\leq TM_0L_G\sup_{t\in[0, T]}\left\|\begin{pmatrix}\tilde{v}_1(t)-\tilde{v}_2(t)\\ \tilde{w}_1(t)-\tilde{w}_2(t)\end{pmatrix}\right\|_{L^{2}(\Omega)\times H^{1}(\Omega)}\label{2difference-1}.
	\end{align}
	Because $\{T(t)\in\mathcal{B}\left(L^2(\Omega)\times H_0^1(\Omega)\right):\ t\geq 0\}$ is a strongly continuous semigroup,  according to the definition of strong continuity, for $(\tilde{v}_0, \tilde{w}_0)\in D(\mathcal{A})$ and given constant $r\in\left(0, \frac{\kappa}{2C}\right)$, there exists $\delta_o=\delta_o(r)>0$, such that if $0<t\leq\delta_o$, then
	\be\label{2semigp-continuity}
	0<\left\|T(t)\begin{pmatrix}\tilde{v}_0\\ \tilde{w}_0\end{pmatrix}-\begin{pmatrix}\tilde{v}_0\\ \tilde{w}_0\end{pmatrix}\right\|_{L^2(\Omega)\times H^{1}(\Omega)}\leq \frac{r}{2} .
	\ee
	Since $r\in\left(0, \frac{\kappa}{2C}\right)$ and $C=C(\Omega)$ is a constant, $\delta_o$ depends on $\kappa$ and $\Omega$, i.e. $\delta_o=\delta_o(\kappa, \Omega)$.
	
	 $\tilde{u}\in C\left([0, T]; B_{H^2}\left(\tilde{u}_0, r\right)\right)$ and $(\tilde{v}_1,\tilde{w}_1)\in Z(T)$ imply that  $\tilde{v}_1(t)\in L^2(\Omega)$, $\tilde{w}_1(t)\in H_0^1(\Omega)$ and $\tilde{u}(t)\in H^2(\Omega)\cap H_0^1(\Omega)$ with $\left\|\tilde{u}(t)-\tilde{u}_0\right\|_{H^2(\Omega)}\leq r$,
	thus $[G(\tilde{w}_1)](t)+\beta_p\tilde{u}(t)\in H^1(\Omega)$ for all $ t\in[0, T]$.
	
	Because $G_0=G(\tilde{w}_0)+\beta_p\tilde{u}_0\in H^1(\Omega)$, the inequality \eqref{2ndLip-G-1} of    Lemma \ref{2ndestimates} implies
	\begin{align}
		&\sup_{t\in[0, T]}\left\|\left[\Phi(\tilde{v}_1,\tilde{w}_1)\right](t)-\begin{pmatrix}\tilde{v}_0\\ \tilde{w}_0\end{pmatrix}\right\|_{L^2(\Omega)\times H^{1}(\Omega)}\notag\\
		=&\sup_{t\in[0, T]}\left\|T(t)\begin{pmatrix}\tilde{v}_0\\ \tilde{w}_0\end{pmatrix}-\begin{pmatrix}\tilde{v}_0\\ \tilde{w}_0\end{pmatrix}+ \int_0^t\left\{T(t-s)\begin{pmatrix}[G(\tilde{w}_1)](s)+\beta_p\tilde{u}(s)\\ 0\end{pmatrix}\right\}ds\right\|_{L^2(\Omega)\times H^{1}(\Omega)}\notag\\
		\leq&\sup_{t\in[0, T]}\left\|T(t)\begin{pmatrix}\tilde{v}_0\\ \tilde{w}_0\end{pmatrix}-\begin{pmatrix}\tilde{v}_0\\ \tilde{w}_0\end{pmatrix}\right\|_{L^2(\Omega)\times H^{1}(\Omega)}+  TM_0\left\|G_0\right\|_{ L^{2}(\Omega)}\notag\\
		+&  TM_0\sup_{t\in[0, T]}\left\|[G(\tilde{w}_1)](t)-G(\tilde{w}_0)\right\|_{L^2(\Omega)}+TM_0\sup_{t\in[0, T]}\left\|\tilde{u}(t)-\tilde{u}_0\right\|_{L^2(\Omega)}\notag\\
		\leq&\sup_{t\in[0, T]}\left\|T(t)\begin{pmatrix}\tilde{v}_0\\ \tilde{w}_0\end{pmatrix}-\begin{pmatrix}\tilde{v}_0\\ \tilde{w}_0\end{pmatrix}\right\|_{L^2(\Omega)\times H^{1}(\Omega)}+  TM_0\left\|G_0\right\|_{ H^{1}(\Omega)}\notag\\
		+&  TM_0\sup_{t\in[0, T]}\left\|[G(\tilde{w}_1)](t)-G(\tilde{w}_0)\right\|_{H^1(\Omega)}+TM_0\sup_{t\in[0, T]}\left\|\tilde{u}(t)-\tilde{u}_0\right\|_{H^1(\Omega)}\notag\\
		\leq&\sup_{t\in[0, T]} \left\|T(t)\begin{pmatrix}\tilde{v}_0\\ \tilde{w}_0\end{pmatrix}-\begin{pmatrix}\tilde{v}_0\\ \tilde{w}_0\end{pmatrix}\right\|_{L^2(\Omega)\times H^{1}(\Omega)}+TM_0\left(\left\|G_0\right\|_{ H^{1}(\Omega)}+ \left(L_G+1\right) r\right)\label{2estimate of phi-1}.
	\end{align}
	For fixed small $r$ given in $\left(0, \kappa/(2C)\right)$,  there exists a number $T_0>0$,
	\be\label{2T-0}
	T_0=\min\left\{\delta_o,\ \frac{1}{2M_0L_G},\
	 \frac{\kappa}{2M_0}\left[\left(L_G+1\right)\kappa+2C\left\|G_0\right\|_{ H^{1}(\Omega)}\right]^{-1}\right\},
	\ee
	such that for every $T\in (0,T_0)$, it follows that
	\[\sup_{t\in[0, T]}\left\|\left[\Phi(\tilde{v}_1,\tilde{w}_1)\right](t)-\begin{pmatrix}\tilde{v}_0\\ \tilde{w}_0\end{pmatrix}\right\|_{L^2(\Omega)\times H^{1}(\Omega)}\leq r,\]
	\[\sup_{t\in[0, T]}\|[\Phi(\tilde{v}_1,\tilde{w}_1)](t)-[\Phi(\tilde{v}_2,\tilde{w}_2)](t)\|_{L^2(\Omega)\times H^{1}(\Omega)}
	\leq \frac{1}{2}\sup_{t\in[0, T]}\left\|\begin{pmatrix}\tilde{v}_1(t)-\tilde{v}_2(t)\\ \tilde{w}_1(t)-\tilde{w}_2(t)\end{pmatrix}\right\|_{L^{2}(\Omega)\times H^{1}(\Omega)}.\]
	Thus $\Phi(\tilde{v}_1,\tilde{w}_1)\in z(T)$ for $(\tilde{v}_1,\tilde{w}_1)\in Z(T)$,  $\Phi(\tilde{v}, \tilde{w})$ is Lipschitz continuous on the bounded set $Z(T)$ with Lipschitz constant smaller than or equal to $1/2$, and $\Phi(\tilde{v}, \tilde{w})$ is a contraction map which maps $Z(T)$ into itself. 
	
	According to the Banach fixed point theorem, for each $T\in(0, T_0)$, there exists a unique fixed point $(\tilde{v}_T, \tilde{w}_T)\in Z(T)$, such that $(\tilde{v}_T,\tilde{w}_T)=\Phi(\tilde{v}_T,\tilde{w}_T)$ for given $\tilde{u}\in C\left([0, T]; B_{H^2}\left(\tilde{u}_0, r\right)\right)$.
	
	Hence, $(\tilde{v}_T, \tilde{w}_T)\in Z(T)$ is the unique mild solution of the semilinear evolution equation \eqref{2nd-SWE-1} on $[0, T]$, and $(\tilde{v}_T,\tilde{w}_T)$ satisfies the integral formulation \eqref{2ndmild-solu-form}. Due to the uniqueness of the fixed point, we set $(\tilde{v}, \tilde{w})=(\tilde{v}_{T},\tilde{w}_{T})$ and note that $(\tilde{v}_T, \tilde{w}_T)$ is the restriction $(\tilde{v}|_{[0, T]}, \tilde{w}|_{[0, T]})\in Z(T)$ of $(\tilde{v}, \tilde{w})$. As a result, Theorem~\ref{2nd-solu-thm}
is proved.  
\end{proof}

\subsection{Proof of Corollary \ref{2Lip-mild-solu}}
\begin{proof}
	Take $0\leq t<t+h\leq T$. Equation \eqref{2ndmild-solu-form} leads to
	\begin{align}
		\begin{pmatrix}\tilde{v}(t+h)-\tilde{v}(t)\\ \tilde{w}(t+h)-\tilde{w}(t)\end{pmatrix}=&T(t)\left[T(h)\begin{pmatrix}\tilde{v}_0\\ \tilde{w}_0\end{pmatrix}-\begin{pmatrix}\tilde{v}_0\\ \tilde{w}_0\end{pmatrix}\right]+ \int_0^hT(t+h-s)\begin{pmatrix}[G(\tilde{w})](s)+\beta_p\tilde{u}(s)\\ 0\end{pmatrix}ds\notag\\
		+& \int_0^tT(t-s)\begin{pmatrix}[G(\tilde{w})](s+h)-[G(\tilde{w})](s)+\beta_p[\tilde{u}(s+h)-\tilde{u}(s)]\\ 0\end{pmatrix}ds\notag\\
		=&  \int_0^hT(t+s)\mathcal{A}\begin{pmatrix}\tilde{v}_0\\ \tilde{w}_0\end{pmatrix}ds+ \int_0^hT(t+h-s)\begin{pmatrix}[G(\tilde{w})](s)+\beta_p\tilde{u}(s)\\ 0\end{pmatrix}ds\notag\\
		+& \int_0^tT(t-s)\begin{pmatrix}[G(\tilde{w})](s+h)-[G(\tilde{w})](s)+\beta_p[\tilde{u}(s+h)-\tilde{u}(s)]\\ 0\end{pmatrix}ds.\label{2diff-mild-solution}
	\end{align}
	Recall that $\mathscr{X}=L^2(\Omega)\times H_0^1(\Omega)$ and observe that
	\be\label{2Lip-est-1}
	M_0=\sup_{t\in[0, \infty)}\left\|T(t)\right\|_{\mathcal{B}(\mathscr{X})}, \qquad \left\|\mathcal{A}\begin{pmatrix}\tilde{v}_0\\ \tilde{w}_0\end{pmatrix}\right\|_{L^2(\Omega)\times H^1(\Omega)}\leq \left\|\left(\tilde{v}_0, \tilde{w}_0\right)\right\|_{D(\mathcal{A})}.
	\ee
	Fixing $T\in (0, T_0)$ and $\tilde{u}\in C\left([0, T]; B_{H^2}\left(\tilde{u}_0,  r\right)\right)$, the semilinear evolution equation \eqref{2nd-SWE-1} has a unique mild solution $\left(\tilde{v}, \tilde{w}\right)\in Z(T)$, and by using \eqref{2ndLip-G-1} in   Lemma \ref{2ndLip-G-Lem}, we have
	\begin{align}
		\sup_{t\in[0, T]}\left\|\begin{pmatrix}[G(\tilde{w})](t)+\beta_p\tilde{u}(t)\\ 0\end{pmatrix}\right\|_{L^2(\Omega)\times H^1(\Omega)}
		&\leq \sup_{t\in[0, T]}\left\|[G(\tilde{w})](t)+\beta_p\tilde{u}(t)\right\|_{H^1(\Omega)}\notag\\
		&\leq \sup_{t\in[0, T]}\left\|[G(\tilde{w})](t)-G(\tilde{w}_0)\right\|_{H^1(\Omega)}\notag\\
		&+\beta_p\sup_{t\in[0, T]}\left\|\tilde{u}(t)-\tilde{u}_0\right\|_{H^1(\Omega)}+\left\|G_0\right\|_{H^{1}(\Omega)}\notag\\
		&\leq \left(L_G+1\right) r+\left\|G_0\right\|_{H^{1}(\Omega)}\notag\\
		&\leq  \frac{\kappa \left(L_G+1\right)}{2C}+\left\|G_0\right\|_{H^{1}(\Omega)}\label{2ndLip-est-2}.
	\end{align}
	Here $L_G$ is given by   Lemma \ref{2ndLip-G-Lem} and $G_0=G(\tilde{w}_0)+\beta_p\tilde{u}_0\in H^1(\Omega)$.
	
	Moreover, $\tilde{u}(s+h)$, $\tilde{u}(s)\in H^2(\Omega)\cap H_0^1(\Omega)$, $\tilde{w}(s+h)$, $\tilde{w}(s)\in H_0^1(\Omega)$, $\forall\ 0\leq s< s+h \leq t\leq T$, therefore,
	$[G(\tilde{w})](s+h)+\beta_p\tilde{u}(s+h)-[G(\tilde{w})](s)-\beta_p\tilde{u}(s)\in H^1(\Omega).$
	
	Since $\tilde{u}\in C^1([0, T_0); L^2(\Omega))$, then, for all $ T\in(0, T_0)$, we find
	\begin{align}
		\sup_{t\in[0, T]} \int_0^t\left\|\tilde{u}(s+h)-\tilde{u}(s)\right\|_{L^2(\Omega)}ds\leq& T_0\sup_{\begin{smallmatrix}0\leq s\leq T,\\ 0\leq s+\sigma h\leq T\end{smallmatrix}}h\int_0^1\left\|\tilde{u}'(s+\sigma h)\right\|_{L^2(\Omega)}d\sigma\notag\\
		\leq&T_0h\left\|\tilde{u}\right\|_{C^1\left([0, T_0); L^2(\Omega)\right)}.\notag
	\end{align}
	According to  inequality \eqref{2ndHoldercontinuous} of   Lemma \ref{2ndLip-G-Lem}, we have
	\begin{align}
		&\left\| \int_0^tT(t-s)\begin{pmatrix}[G(\tilde{w})](s+h)-[G(\tilde{w})](s)+\beta_p[\tilde{u}(s+h)-\tilde{u}(s)]\\ 0\end{pmatrix}ds\right\|_{L^2(\Omega)\times H^1(\Omega)}\notag\\
		=& M_0 \int_0^t\left\|[G(\tilde{w})](s+h)-[G(\tilde{w})](s)+\beta_p[\tilde{u}(s+h)-\tilde{u}(s)]\right\|_{L^2(\Omega)}ds\notag\\
		\leq& M_0 \int_0^t\beta_p\left\|\tilde{u}(s+h)-\tilde{u}(s)\right\|_{L^2(\Omega)}
		+\left\|[G(\tilde{w})](s+h)-[G(\tilde{w})](s)\right\|_{H^1(\Omega)}ds\notag\\
		\leq& hM_0\beta_pT_0\left\|\tilde{u}\right\|_{C^1\left([0, T_0); L^2(\Omega)\right)}+M_0L_G \int_0^t\left\|\begin{pmatrix}\tilde{v}(s+h)-\tilde{v}(s)\\  \tilde{w}(s+h)-\tilde{w}(s)\end{pmatrix}\right\|_{L^2(\Omega)\times H^1(\Omega)}ds\label{2diff-mild-solution-1}.
	\end{align}
	Combining \eqref{2diff-mild-solution}, \eqref{2Lip-est-1}, \eqref{2ndLip-est-2} and \eqref{2diff-mild-solution-1} gives
	\begin{align}
		\left\|\begin{pmatrix}\tilde{v}(t+h)-\tilde{v}(t)\\ \tilde{w}(t+h)-\tilde{w}(t)\end{pmatrix}\right\|_{L^2(\Omega)\times H^1(\Omega)}&\leq hM_0\left\|\left(\tilde{v}_0, \tilde{w}_0\right)\right\|_{D(\mathcal{A})}+hM_0\left( \frac{\kappa\left(L_G+1\right) }{2C}+\left\|G_0\right\|_{H^{1}(\Omega)}\right)\notag\\
		&+hM_0\beta_pT_0\left\|\tilde{u}\right\|_{C^1\left([0, T_0); L^2(\Omega)\right)}\notag\\
		&+M_0L_G \int_0^t\left\|\begin{pmatrix}\tilde{v}(s+h)-\tilde{v}(s)\\  \tilde{w}(s+h)-\tilde{w}(s)\end{pmatrix}\right\|_{L^2(\Omega)\times H^1(\Omega)}ds\notag.
	\end{align}
	Set $V_o=\left\|\left(\tilde{v}_0, \tilde{w}_0\right)\right\|_{D(\mathcal{A})}+\left( \frac{\kappa \left(L_G+1\right)}{2C}+\left\|G_0\right\|_{H^{1}(\Omega)}\right)+\beta_pT_0\left\|\tilde{u}\right\|_{C^1\left([0, T_0); L^2(\Omega)\right)}$.
	
	Gronwall's inequality then implies that
	\[\left\|\begin{pmatrix}\tilde{v}(t+h)-\tilde{v}(t)\\ \tilde{w}(t+h)-\tilde{w}(t)\end{pmatrix}\right\|_{L^2(\Omega)\times H^1(\Omega)}\leq M_0V_o\left(e^{M_0L_GT_0}\right)h.\]
	Therefore, \eqref{2Lip-mild-solu-inquality} holds for all $h\in(0, T]$ by setting $L_V=M_0V_o\left(e^{M_0L_GT_0}\right).$
\end{proof}

\subsection{Proof of Corollary \ref{2Holdercontinuity}}

\begin{proof}
	Observe that 
	\begin{align}
		\sup_{t\in[0, T]} \int_0^t\left\|\tilde{u}(s+h)-\tilde{u}(s)\right\|_{L^2(\Omega)}ds\leq& T_0\sup_{0\leq s<s+h\leq T}\left\|\tilde{u}(s+h)-\tilde{u}(s)\right\|_{L^2(\Omega)}\notag\\
		\leq&T_0\sup_{0\leq s<s+h\leq T}\left\|\tilde{u}(s+h)-\tilde{u}(s)\right\|_{H^2(\Omega)}\notag\\
		\leq&T_0h^\alpha\left[\tilde{u}\right]_{C^\alpha\left([0, T_0); H^2(\Omega)\right)}.\label{2ndLip-est-3}
	\end{align}
	According to \eqref{2diff-mild-solution}, \eqref{2Lip-est-1}, \eqref{2ndLip-est-2}, \eqref{2diff-mild-solution-1} and \eqref{2ndLip-est-3}, for $0\leq t<t+h\leq T$, with $h\in(0, T]$, it follows that
	\begin{align}
		\left\|\begin{pmatrix}\tilde{v}(t+h)-\tilde{v}(t)\\ \tilde{w}(t+h)-\tilde{w}(t)\end{pmatrix}\right\|_{L^2(\Omega)\times H^1(\Omega)}&\leq hM_0\left\|\left(\tilde{v}_0, \tilde{w}_0\right)\right\|_{D(\mathcal{A})}+h^\alpha M_0\beta_pT_0\left[\tilde{u}\right]_{C^\alpha\left([0, T_0); H^2(\Omega)\right)}\notag\\
		&+hM_0\left( \frac{\kappa\left(L_G+1\right)}{2C}+\left\|G(\tilde{w}_0)+\beta_p\tilde{u}_0\right\|_{H^{1}(\Omega)}\right)\notag\\
		&+M_0L_G \int_0^t\left\|\begin{pmatrix}\tilde{v}(s+h)-\tilde{v}(s)\\  \tilde{w}(s+h)-\tilde{w}(s)\end{pmatrix}\right\|_{L^2(\Omega)\times H^1(\Omega)}ds\notag.
	\end{align}
	Set $P_0= \frac{\kappa\left(L_G+1\right)}{2C}+\left\|G(\tilde{w}_0)+\beta_p\tilde{u}_0\right\|_{H^{1}(\Omega)}$. Gronwall's inequality then implies that
	\begin{align}
		\left\|\begin{pmatrix}\tilde{v}(t+h)-\tilde{v}(t)\\ \tilde{w}(t+h)-\tilde{w}(t)\end{pmatrix}\right\|_{L^2(\Omega)\times H^1(\Omega)}
		\leq& \left(e^{M_0L_GT_0}\right)M_0\left\{\left\|\left(\tilde{v}_0, \tilde{w}_0\right)\right\|_{D(\mathcal{A})}+P_0\right\}h\notag\\
		+&\left(e^{M_0L_GT_0}\right)M_0\beta_pT_0\left[\tilde{u}\right]_{C^\alpha\left([0, T_0); H_o^2(\Omega)\right)}h^\alpha \notag.
	\end{align}
	Notice that
	\[\left(e^{M_0L_GT_0}\right)M_0\left\{\left\|\left(\tilde{v}_0, \tilde{w}_0\right)\right\|_{D(\mathcal{A})}+P_0\right\}h\leq \left(e^{M_0L_GT_0}\right)M_0\left\{\left\|\left(\tilde{v}_0, \tilde{w}_0\right)\right\|_{D(\mathcal{A})}+P_0\right\}T_0^{1-\alpha}h^\alpha\notag,\]
	\begin{align}
		\left(e^{M_0L_GT_0}\right)M_0\beta_pT_0\left[\tilde{u}\right]_{C^\alpha\left([0, T_0); H_o^2(\Omega)\right)}\leq& \left(e^{M_0L_GT_0}\right)M_0\beta_pT_0\left\|\tilde{u}\right\|_{C^\alpha\left([0, T_0); H_o^2(\Omega)\right)}\notag\\
		\leq& \left(e^{M_0L_GT_0}\right)M_0\beta_pT_0\left(r+\left\|\tilde{u}_0\right\|_{H^2(\Omega)}\right)\notag\\
		\leq&\left(e^{M_0L_GT_0}\right)M_0\beta_pT_0\left( \frac{\kappa}{2C}+\left\|\tilde{u}_0\right\|_{H^2(\Omega)}\right)\notag.
	\end{align}
	Setting $L_U=\left(e^{M_0L_GT_0}\right)M_0\left\{\left\|\left(\tilde{v}_0, \tilde{w}_0\right)\right\|_{D(\mathcal{A})}+P_0\right\}T_0^{1-\alpha}+\left(e^{M_0L_GT_0}\right)M_0\beta_pT_0\left( \frac{\kappa}{2C}+\left\|\tilde{u}_0\right\|_{H^2(\Omega)}\right)$,
	and $L_U$ is a Lipschitz constant depending on
	\[\alpha, \ T_0,\ \kappa,\ \Omega,\ \beta_p,\ \beta_F,\ M_0= \sup_{t\in[0, \infty)}\left\|T(t)\right\|_{\mathcal{B}(\mathscr{X})},\ \left\|\left(\tilde{v}_0, \tilde{w}_0\right)\right\|_{D(\mathcal{A})},\  \left\|\tilde{u}_0\right\|_{H^2(\Omega)}, \ \left\|\tilde{w}_0\right\|_{H^{1}(\Omega)}.\]
	Therefore, \eqref{2Holdercontinuityformular} holds for all $h\in(0, T]$ and this concludes the proof of Corollary \ref{2Holdercontinuity}.
\end{proof}

\subsection{Proof of Theorem \ref{2nd-mild-solution-cor}}
\begin{proof}
	Let $T\in(0, T_0)$, $G_0=G\left(\tilde{w}_0\right)+\beta_p\tilde{u}_0$ and $\left(\tilde{v}, \tilde{w}\right)$ be the mild solution of the semilinear evolution equation \eqref{2nd-SWE-1} defined via \eqref{2ndmild-solu-form}. Take $\tilde{u}\in C\left([0, T]; B_{H^2}\left(\tilde{u}_0,  r\right)\right)\cap C^1([0, T]; L^2(\Omega))$ to be given such that  $\tilde{u}'(t)\in L^2(\Omega)$ is uniformly continuous for all $t\in [0, T]$.	
	
	We first prove the linear non-autonomous problem
	\be\label{2ndLNP}
	\begin{pmatrix}\tilde{p}(t)\\ \tilde{q}(t)\end{pmatrix}=T(t)\left(\begin{pmatrix}G_0\\ 0\end{pmatrix}+\mathcal{A}\begin{pmatrix}\tilde{v}_0\\ \tilde{w}_0\end{pmatrix}\right)+ \int_0^tT(t-s)\begin{pmatrix}[\mathcal{H}(\tilde{q})](s)+\beta_p\tilde{u}'(s)\\ 0\end{pmatrix}ds,
	\ee
	can be solved for $t\in[0, T]$. Here
	\be\label{2H-nonlinearity}
	[\mathcal{H}(\tilde{q})](s)= \frac{2\beta_F\tilde{q}(s)}{\left[\tilde{w}(s)+\theta_2\right]^{3}}=[G'(\tilde{w})q](s),\quad s\in [0, t].
	\ee
	We define a nonlinear operator $\Psi$ by
	\[\Psi:\ C\left([0, T]; L^2(\Omega)\times H_0^1(\Omega)\right)\longrightarrow C\left([0, T]; L^2(\Omega)\times H_0^1(\Omega)\right)\]
	\[\left[\Psi\left(\tilde{p}, \tilde{q}\right)\right](t)=T(t)\left(\begin{pmatrix}G_0\\ 0\end{pmatrix}+\mathcal{A}\begin{pmatrix}\tilde{v}_0\\ \tilde{w}_0\end{pmatrix}\right)+ \int_0^tT(t-s)\begin{pmatrix}[\mathcal{H}(\tilde{q})](s)+\beta_p\tilde{u}'(s)\\ 0\end{pmatrix}ds.\]
	For any
	$\left(\tilde{p}_1, \tilde{q}_1\right)$, $\left(\tilde{p}_2, \tilde{q}_2\right)\in C\left([0, T]; L^2(\Omega)\times H_0^1(\Omega)\right)$ and $\tilde{w}\in C\left([0, T]; H_0^1(\Omega)\right)$. It follows that
	\[\mathcal{H}(\tilde{q}_1)-\mathcal{H}(\tilde{q}_2)= \frac{2\beta_F}{\left[\tilde{w}+\theta_2\right]^{3}}\left[\tilde{q}_1-\tilde{q}_2\right]=G'(\tilde{w})(q_1-q_2)\in C\left([0, T]; H_0^1(\Omega)\right).\]
	Hence, according to the estimate \eqref{2ndLip-G-2} of Fr\'{e}chet derivative $G'\left(\tilde{w}\right)q$ from  Lemma \ref{2ndLip-G-Lem}, $\Psi$ is a contraction map on $C\left([0, T]; L^2(\Omega)\times H_0^1(\Omega)\right)$ for all $ T\in(0, T_0)$ since
	\begin{align}
		&\left\|\left[\Psi\left(\tilde{p}_1, \tilde{q}_1\right)\right](t)-\left[\Psi\left(\tilde{p}_2, \tilde{q}_2\right)\right](t)\right\|_{L^2(\Omega)\times H^1(\Omega)}\notag\\
		=&\left\| \int_0^tT(t-s)\begin{pmatrix}[\mathcal{H}(\tilde{q}_1)](s)-\mathcal{H}(\tilde{q}_2)](s)\\ 0\end{pmatrix}ds
		\right\|_{L^2(\Omega)\times H^1(\Omega)}\notag\\
		\leq& T\sup_{0\leq s\leq t\leq T}\|T(t-s)\|_{\mathcal{B}(\mathscr{X})}\sup_{t\in[0, T]}\left\|[\mathcal{H}(\tilde{q}_1)](t)-[\mathcal{H}(\tilde{q}_2)](t)\right\|_{L^2(\Omega)}\notag\\
		\leq&TM_0L_G \sup_{t\in[0, T]}\left\|\tilde{q}_1(t)-\tilde{q}_2(t)\right\|_{H^1(\Omega)}
		\leq \frac{1}{2}\sup_{t\in[0, T]}\left\|\begin{pmatrix}\tilde{p}_1(t)-\tilde{p}_2(t)\\ \tilde{q}_1(t)-\tilde{q}_2(t)\end{pmatrix}\right\|_{L^2(\Omega)\times H^1(\Omega)}\notag.
	\end{align}
	According to the Banach fixed point theorem, for given $\tilde{u}\in C\left([0, T]; B_{H^2}\left(\tilde{u}_0, r\right)\right)\cap C^1\left([0, T]; L^2(\Omega)\right),$ there exists a unique fixed point $(\tilde{p}, \tilde{q})\in C\left([0, T]; L^2(\Omega)\times H_0^1(\Omega)\right)$, such that $(\tilde{p}, \tilde{q})=\Psi(\tilde{p}, \tilde{q}).$ Hence the $\mathds{R}$-linear non-autonomous problem \eqref{2ndLNP} can be solved for $t\in[0, T]$.
	
	We next prove that $(\tilde{p}, \tilde{q})$ is the time derivative of the mild solution $(\tilde{v}, \tilde{w})$.

Let $0\leq t<t+h\leq T$ for some $h\in(0, T]$, equations \eqref{2ndmild-solu-form} and \eqref{2ndLNP} imply that
\begin{align}
	E(h,t)&:= \frac{1}{h}\begin{pmatrix}\tilde{v}(t+h)-\tilde{v}(t)\\ \tilde{w}(t+h)-\tilde{w}(t)\end{pmatrix}-\begin{pmatrix}\tilde{p}(t)\\ \tilde{q}(t)\end{pmatrix}\notag\\
	&=T(t) \frac{1}{h}\left(T(h)-I\right)\begin{pmatrix}\tilde{v}_0\\ \tilde{w}_0\end{pmatrix}-T(t)\mathcal{A}\begin{pmatrix}\tilde{v}_0\\ \tilde{w}_0\end{pmatrix}\notag\\
	&+ \frac{1}{h}\int_0^h\left\{T(t+h-s)\begin{pmatrix}[G(\tilde{w})](s)+\beta_p\tilde{u}(s)\\ 0\end{pmatrix}\right\}ds-T(t)\begin{pmatrix}G_0\\ 0\end{pmatrix}\notag\\
	&+ \int_0^tT(t-s)\begin{pmatrix}\frac{1}{h}\left\{[G(\tilde{w})](s+h)-[G(\tilde{w})](s)\right\}-[\mathcal{H}(\tilde{q})](s)\\ 0\end{pmatrix}ds\notag\\
	&+\int_0^tT(t-s)\begin{pmatrix}\beta_p\left[\frac{1}{h}\left\{\tilde{u}(s+h)-\tilde{u}(s)\right\}-\tilde{u}'(s)\right]\\ 0\end{pmatrix}ds\notag.
\end{align}
Let
\[E^{(1)}(h,t):=T(t) \frac{1}{h}\left(T(h)-I\right)\begin{pmatrix}\tilde{v}_0\\ \tilde{w}_0\end{pmatrix}-T(t)\mathcal{A}\begin{pmatrix}\tilde{v}_0\\ \tilde{w}_0\end{pmatrix},\]
\[E^{(2)}(h,t):= \frac{1}{h}\int_0^h\left\{T(t+h-s)\begin{pmatrix}[G(\tilde{w})](s)+\beta_p\tilde{u}(s)\\ 0\end{pmatrix}\right\}ds-T(t)\begin{pmatrix}G_0\\ 0\end{pmatrix},\]
\begin{align}
	E^{(3)}(h,t)&:= \int_0^tT(t-s)\begin{pmatrix} \frac{1}{h}\left\{[G(\tilde{w})](s+h)-[G(\tilde{w})](s)\right\}-[\mathcal{H}(\tilde{q})](s)\\ 0\end{pmatrix}ds\notag\\
	&+\int_0^tT(t-s)\begin{pmatrix}\beta_p\left[ \frac{1}{h}\left\{\tilde{u}(s+h)-\tilde{u}(s)\right\}-\tilde{u}'(s)\right]\\ 0\end{pmatrix}ds\notag.
\end{align}
We initially observe that
\begin{align}
	\lim_{h\rightarrow 0}\left\|E^{(1)}(h,t)\right\|_{L^2(\Omega)\times H^1(\Omega)}\hspace*{-0.1cm}\leq\hspace*{-0.1cm}\lim_{h\rightarrow 0} M_0\left\| \frac{1}{h}\left(T(h)-I\right)\begin{pmatrix}\tilde{v}_0\\ \tilde{w}_0\end{pmatrix}-\mathcal{A}\begin{pmatrix}\tilde{v}_0\\ \tilde{w}_0\end{pmatrix}\right\|_{L^2(\Omega)\times H^1(\Omega)}\hspace*{-0.2cm}
	:=\lim_{h\rightarrow 0}\Lambda_1(h)=0,\notag
\end{align}
\[\lim_{h\rightarrow 0} \frac{1}{h} \int_0^h\left\{T(h-s)\begin{pmatrix}G_0 \\ 0\end{pmatrix}\right\}ds=\begin{pmatrix}G_0 \\ 0\end{pmatrix}.\]
Because $ G(\tilde{w})\in C([0, T]; H^1(\Omega))$ and $\tilde{u}\in C\left([0, T]; H^2(\Omega)\cap H_0^1(\Omega)\right)$, 
\[\lim_{h\rightarrow 0}\sup_{s\in[0, h]}\left\|[G(\tilde{w})](s)-[G(\tilde{w})](0)+\beta_p[\tilde{u}(s)-\tilde{u}(0)]\right\|_{H^1(\Omega)}=0,\] hence
\begin{align}
	&\lim_{h\rightarrow 0}\left\|E^{(2)}(h,t)\right\|_{L^2(\Omega)\times H^1(\Omega)}\notag\\
	=&\lim_{h\rightarrow 0}\left\| \frac{1}{h} \int_0^h\left\{T(t+h-s)\begin{pmatrix}[G(\tilde{w})](s)+\beta_p\tilde{u}(s)\\ 0\end{pmatrix}\right\}ds-T(t)\begin{pmatrix}G_0\\ 0\end{pmatrix}\right\|_{L^2(\Omega)\times H^1(\Omega)}\notag\\
	=&\lim_{h\rightarrow 0}\left\|T(t) \frac{1}{h} \int_0^h\left\{T(h-s)\begin{pmatrix}[G(\tilde{w})](s)-G(\tilde{w}_0)+\beta_p(\tilde{u}(s)-\tilde{u}_0) \\ 0\end{pmatrix}\right\}ds\right\|_{L^2(\Omega)\times H^1(\Omega)}\notag\\
	\leq&\lim_{h\rightarrow 0}M_0\left\| \frac{1}{h} \int_0^h\left\{T(h-s)\begin{pmatrix}[G(\tilde{w})](s)-G(\tilde{w}_0)+\beta_p(\tilde{u}(s)-\tilde{u}_0)\\ 0\end{pmatrix}\right\}ds\right\|_{L^2(\Omega)\times H^1(\Omega)}\notag\\
	\leq&\lim_{h\rightarrow 0}M_0^2 \frac{h}{h}\sup_{s\in[0, h]}\left\|[G(\tilde{w})](s)-G(\tilde{w}_0)+\beta_p(\tilde{u}(s)-\tilde{u}_0)\right\|_{L^2(\Omega)}\notag\\
	=&\lim_{h\rightarrow 0}M_0^2\sup_{s\in[0, h]}\left\|[G(\tilde{w})](s)-[G(\tilde{w})](0)+\beta_p(\tilde{u}(s)-\tilde{u}(0))\right\|_{H^1(\Omega)}
	:=\lim_{h\rightarrow 0}\Lambda_2(h)=0.\notag
\end{align}
Define
$G_D(t,h):=\left[G(\tilde{w})\right](t+h)-\left[G(\tilde{w})\right](t)-\left[G'(\tilde{w})\right](t)\cdot\left[\tilde{w}(t+h)-\tilde{w}(t)\right]$,
\[E^{(3)}_1(h,t):= \int_0^tT(t-s)\begin{pmatrix}\frac{1}{h}G_D(s,h)\\ 0\end{pmatrix}ds,\]
\[E^{(3)}_2(h,t):= \int_0^tT(t-s)\begin{pmatrix}[G'(\tilde{w})](s)\left\{\frac{1}{h}\{\tilde{w}(s+h)-\tilde{w}(s)\}-\tilde{q}(s)\right\}\\0\end{pmatrix}ds,\]
\[E^{(3)}_3(h,t):= \int_0^tT(t-s)\begin{pmatrix}\beta_p\left\{\frac{1}{h}[\tilde{u}(s+h)-\tilde{u}(s)]-\tilde{u}'(s)\right\}\\ 0\end{pmatrix}ds.\]
We then write
\[E^{(3)}(h,t)=E^{(3)}_1(h,t)+E^{(3)}_2(h,t)+E^{(3)}_3(h,t).\]
As $\tilde{u}\in C\left([0, T]; B_{H^2}\left(\tilde{u}_0,  r\right)\right)\cap C^1\left([0, T]; L^2(\Omega)\right)$ is given such that the time derivative $\tilde{u}'(t)\in L^2(\Omega)$ is uniformly continuous for all $t\in [0, T]$, we obtain
\begin{align}
	\left\|E^{(3)}_3(h,t)\right\|_{L^2(\Omega)\times H^1(\Omega)}=&\left\| \int_0^tT(t-s)\begin{pmatrix}\beta_p\left\{\frac{1}{h}[\tilde{u}(s+h)-\tilde{u}(s)]-\tilde{u}'(s)\right\}\\ 0\end{pmatrix}ds\right\|_{L^2(\Omega)\times H^1(\Omega)}\notag\\
	\leq& T_0M_0\beta_p\sup_{\begin{smallmatrix}0\leq t<t+h\leq T\end{smallmatrix}}\left\|\frac{\tilde{u}(t+h)-\tilde{u}(t)}{h}-\tilde{u}'(t)\right\|_{L^2(\Omega)}\notag\\
	=&T_0M_0\beta_p\sup_{\begin{smallmatrix}0\leq t\leq t+\sigma h\leq T\\ 0\leq\sigma\leq 1\end{smallmatrix}}\left\|\frac{1}{h}\int_0^1\frac{d}{d\sigma}\left[\tilde{u}'(t+\sigma h)\right]d\sigma-\tilde{u}'(t)\right\|_{L^2(\Omega)}\notag\\
	=&T_0M_0\beta_p\sup_{\begin{smallmatrix}0\leq t\leq t+\sigma h\leq T\\ 0\leq\sigma\leq 1\end{smallmatrix}}\left\|\tilde{u}'(t+\sigma h)-\tilde{u}'(t)\right\|_{L^2(\Omega)}
	:=\Lambda_3(h)\rightarrow 0,\ \text{as}\ h\rightarrow 0. \notag
\end{align}
The bound  \eqref{2ndLip-G-2} of Fr\'{e}chet derivative $G'\left(\tilde{w}\right)$ from   Lemma \ref{2ndLip-G-Lem} implies
\begin{align}
	\left\|E^{(3)}_2(h,t)\right\|_{L^2(\Omega)\times H^1(\Omega)}\leq M_0L_G \int_0^t\left\|E(h,s)\right\|_{L^2(\Omega)\times H^1(\Omega)}ds\notag.
\end{align}
Notice that $G_D(t,h)\in H^1(\Omega)$. By the estimate \eqref{2Lip-mild-solu-inquality} of Corollary \ref{2Lip-mild-solu} and the inequality \eqref{2nduniformly-continuous-Fre-G} of   Lemma \ref{2ndLip-G-Lem}, the function $\tilde{w}$ is Lipschitz continuous with respect to $t\in [0, T]$ and $G'(\tilde{w})$ is uniformly continuous with respect to $t\in[0, T]$. Employing these facts gives
\begin{align}
	&\left\|E^{(3)}_1(h,t)\right\|_{L^2(\Omega)\times H^1(\Omega)}\notag\\
	=&T_0M_0\sup_{\begin{smallmatrix}0\leq t\leq t+\tau h\leq T\\ 0\leq\tau\leq1\end{smallmatrix}}\frac{1}{h}\left\| \int_0^1\left[G'(\tilde{w})\right](t+\tau h)-\left[G'(\tilde{w})\right](t)d\tau\cdot\left[\tilde{w}(t+h)-\tilde{w}(t)\right]\right\|_{ H^1(\Omega)}\notag\\
	\leq& T_0M_0\frac{L_V h}{h}\sup_{\begin{smallmatrix}0\leq t\leq t+\tau h\leq T\\ 0\leq\tau\leq1\end{smallmatrix}}\left\| \int_0^1\left[G'(\tilde{w})\right](t+\tau h)-\left[G'(\tilde{w})\right](t)d\tau\right\|_{\mathcal{B}\left(H_0^1(\Omega)\right)}\notag\\
	=&T_0M_0L_V\sup_{\begin{smallmatrix}0\leq t\leq t+\tau h\leq T\\ 0\leq\tau\leq1\end{smallmatrix}}\left\|\left[G'(\tilde{w})\right](t+\tau h)-\left[G'(\tilde{w})\right](t)\right\|_{\mathcal{B}\left(H_0^1(\Omega)\right)}
	:=\Lambda_4(h)\rightarrow 0 ,\ \text{as}\ h\rightarrow 0. \notag
\end{align}
Summing up, we have shown
\[\left\|E(h,t)\right\|_{L^2(\Omega)\times H^1(\Omega)}\leq\Lambda_1(h)+\Lambda_2(h)+\Lambda_3(h)+\Lambda_4(h)+M_0L_G \int_0^t\left\|E(h,s)\right\|_{L^2(\Omega)\times H^1(\Omega)}ds.\]
Gronwall's inequality thus implies the inequality
\[
\left\|E(h,t)\right\|_{L^2(\Omega)\times H^2(\Omega)}\leq\left(\Lambda_1(h)+\Lambda_2(h)+\Lambda_3(h)+\Lambda_4(h)\right)e^{tM_0L_G}
\]
 for $t\in [0, T]$. Letting $h\rightarrow 0^+$, we then deduce that the $\left(\tilde{v}, \tilde{w}\right)$ is differentiable from the right and the right  derivative of $\left(\tilde{v}, \tilde{w}\right)$ coincides with $\left(\tilde{p}, \tilde{q}\right)$. Because $\left(\tilde{p}, \tilde{q}\right)$ is continuous on $[0, T]$, by using Lemma \ref{time-derivative-continuity}, we conclude $\left(\tilde{v}, \tilde{w}\right)\in C^1\left([0, T]; L^2(\Omega)\times H_0^1(\Omega)\right)$. Since the function $\tilde{u}$ is given in $C\left([0, T]; B_{H^2}\left(\tilde{u}_0,  r\right)\right)\cap C^1([0, T]; L^2(\Omega))$, then $\left(G\left(\tilde{w}\right)+\beta_p\tilde{u}, 0\right)\in C^1\left([0, T]; L^2(\Omega)\times H_0^1(\Omega)\right)$. By Lemma \ref{IEE-S}, the mild solution $\left(\tilde{v}, \tilde{w}\right)$, defined by \eqref{2ndmild-solu-form}, uniquely solves the semilinear evolution equation \eqref{2nd-SWE-1} on $[0, T]$; $\left(\tilde{v}, \tilde{w}\right)$ is a unique strict solution of  \eqref{2nd-SWE-1} with 
 \[
 \left(\tilde{v}, \tilde{w}\right)\in C\left([0, T]; H_0^1(\Omega)\times \left\{H^2(\Omega)\cap H_0^1(\Omega)\right\}\right)\cap C^1\left([0, T]; L^2(\Omega)\times H_0^1(\Omega)\right)
 \] 
 for all $T\in(0, T_0)$.
\end{proof}
\bibliographystyle{abbrv}
\bibliography{Bibliography4}

\end{document}